\documentclass[12pt]{amsart}

\usepackage{amssymb,amsmath,mathrsfs}

\usepackage{color, epsfig, verbatim, subfigure, epsf, enumitem} 
\usepackage[pagebackref,hypertexnames=false, colorlinks, citecolor=red, linkcolor=red]{hyperref}

\usepackage{graphicx}

\usepackage{setspace}
\makeatletter
\@namedef{subjclassname@2020}{%
  \textup{2020} Mathematics Subject Classification}
\makeatother

\usepackage[T1]{fontenc}
\newtheorem{theorem}{Theorem}[section]

\newtheorem{lemma}[theorem]{Lemma}
\newtheorem*{lemma*}{Lemma}
\newtheorem{conjecture}[theorem]{Conjecture}

\newtheorem*{theorem*}{Theorem}
\newtheorem*{conjecture*}{Conjecture}
\newtheorem{corollary}[theorem]{Corollary}
\newtheorem*{corollary*}{Corollary}
\newtheorem{proposition}[theorem]{Proposition}
\newtheorem*{proposition*}{Proposition}

\theoremstyle{remark}
\newtheorem{remark}[theorem]{Remark}

\newtheorem{definition}[theorem]{Definition}
\newtheorem{example}[theorem]{Example}

\numberwithin{equation}{section}

\newcommand{\McC}{\raise.5ex\hbox{c}}

\newcommand{\D}{\mathbb{D}}
\newcommand{\T}{\mathbb{T}}

\newcommand{\C}{\mathbb{C}}

\newcommand{\R}{\mathbb{R}}
\newcommand{\p}{\mathfrak{p}}

\newcommand{\UHP}{\mathbb{H}} %%upper half plane
\newcommand{\cps}{\C\{z_1\}} %%convergent power series

\newcommand\xqed[1]{%
  \leavevmode\unskip\penalty9999 \hbox{}\nobreak\hfill
  \quad\hbox{#1}}
\newcommand\eox{\xqed{$\blacklozenge$}}  %%creates end of example symbol

\renewcommand{\Im}{\operatorname{Im}}

\begin{document}

%%%%% To ease editing, for IMPAN journals add:

\baselineskip=17pt

%%%%%%%%%%%%%%%%

%\title[Sample paper]{A sample paper for IMPAN journals}

\title[Stable polynomials and bounded rational functions]{Local theory of stable polynomials and bounded rational functions of several variables}

\author[Bickel]{Kelly Bickel}
\address{Bucknell University, Department of Mathematics, 360 Olin Science Building, Lewisburg, PA 17837, USA.}
\email{kelly.bickel@bucknell.edu}

\author[Knese]{Greg Knese$^*$}
\address{Washington University in St. Louis, Department of Mathematics, St. Louis,
MO 63130, USA.}
\email{$^*$corresponding author:\ geknese@wustl.edu}

\author[Pascoe]{James Eldred Pascoe}
\address{Drexel University, Department of Mathematics, 3141 Chestnut Street, Philadelphia, PA 19104.}
\email{jep362@drexel.edu}

\author[Sola]{Alan Sola}
\address{Stockholm University, Department of Mathematics, 106 91, Stockholm, Sweden.}
\email{sola@math.su.se}

\date{\today}

\begin{abstract}
%A template for articles in IMPAN journals in the \texttt{amsart} style. Using \texttt{pdflatex} is strongly preferred.
	We provide detailed local descriptions of stable polynomials
	in terms of their homogeneous decompositions, Puiseux expansions, 
	and transfer function realizations. 
	We use this theory to first prove that bounded rational functions on the polydisk
	possess non-tangential limits at every boundary point.
	We relate higher non-tangential regularity and distinguished boundary behavior of
	bounded rational functions to geometric properties of the zero sets of stable polynomials
	via our local descriptions.
	For a fixed stable polynomial $p$,
	we analyze the ideal of numerators $q$ 
	such that $q/p$ is bounded on the bi-upper half plane.  
	We completely characterize this ideal in several geometrically interesting situations
	including smooth points, double points, and ordinary multiple points of $p$.
	Finally, we analyze integrability properties of bounded rational functions and their derivatives on the bidisk.

\end{abstract}

%\subjclass[2020]{Primary XXXX; Secondary YYYY}


%\keywords{aaaa, bbbb, cccc}


\thanks{Grant support: KB supported by NSF DMS grant \#2000088; GK supported by NSF DMS grants \#1900816, \#2247702;
JEP supported by NSF DMS grants \#1953963, \#2319010}

\maketitle

\tableofcontents

\section{Introduction}\label{sec:intro}
	A multivariate \textbf{stable polynomial} $p \in \mathbb{C}[z_1, \dots, z_d]$ is a polynomial
	that does not vanish on a specified domain $\Omega \subseteq \mathbb{C}^d$.
	We generally take $\Omega$ to be the product of upper half planes $\mathbb{H}^d$,
	where \[\mathbb{H}:= \{ z \in \mathbb{C} : \Im(z) >0\},
	\]
	or the conformally-equivalent polydisk
	$\mathbb{D}^d$ with \[\mathbb{D}:= \{ z \in \mathbb{C} : |z|<1\}.\]
	Historically, the term ``stable polynomial" may have arisen in the context of stability of linear time-invariant systems 
	studied by Routh and independently by Hurwitz.
	Within the past few decades however, multivariate stable polynomials have emerged as critical tools for studying phenomena
	in fields including complex analysis, probability theory, control theory, electrical engineering, combinatorics,
	and dynamical systems, see \cite{BBL09, Gur08, GW06, Kne10, Kum02, lsv13, W11}.
	Some notable highlights include work of Br\"and\'en and Borcea \cite{bb09a, bb09b} classifying 
	operators preserving stable polynomials with applications to statistical physics;
	Kurasov and Sarnak  \cite{KS20} using stable polynomials to resolve questions on Fourier quasicrystals; and
	Marcus, Spielman, and Srivastava \cite{MSS1, MSS2} using the technique of
	interlacing families of stable polynomials to resolve major problems in graph theory and analysis.
	For a variety of reasons, which include the existence of simple reflection operations
	($z\mapsto \bar{z}$ or $z \mapsto 1/\bar{z}$), the strongest results/applications of
	 stable polynomials 
	are known in the settings of
	$\UHP^d$ and $\D^d$.  On the other hand, comparatively little is known about stable polynomials 
	on more general domains in $\C^d$.	
	
	Recently, the present authors 
	%\cite{bps18, bps19a, bps19b,Kne15} 
	established a variety of results about
	boundary regularity and integrability of rational functions.
	Along the way a number of \emph{ad hoc} results about the local behavior of stable polynomials were developed.  
	We now standardize and significantly extend this local theory of stable polynomials to create a toolbox powerful enough 
	to address most all of these previous results, at least one conjecture (Conjecture 5.2 of \cite{bps19a}), 
	and several new questions which are posed in Section \ref{sec:struct}. 
	Our local theory gives precise, structural information about the homogeneous expansions, Puiseux factorizations,
	and realization formulas of stable polynomials---thus illuminating exactly how stable polynomials 
	behave on the boundary of their respective zero-free regions.  
	In addition, we construct stable polynomials with more exotic boundary zero sets.

	We investigate the fundamental \textbf{numerator criterion question}:
	 \begin{quote}
	Given a stable polynomial $p$ on $\Omega$, for which polynomials $q$ is $\tfrac{q}{p}$ bounded on $\Omega$?
	\end{quote}
	 In the one-variable setting,
	the zero set of $q$ must include all of the boundary zeros of $p$ by the fundamental theorem of algebra. 
	 In several variables, a correct answer requires a detailed analysis of local zero set behavior.
	We answer the numerator criterion question for stable polynomials with several types of zero set behavior at a distinguished
	 boundary zero, including smooth points, double points, and ordinary multiple points.
	Below, the \emph{Full Numerator Criterion} (Conjecture \ref{conj:fullnum})
	conjectures a complete characterization of the ideal of admissible numerators. 
 
 	The local theory of stable polynomials is used to refine 
	 and extend a number of results from \cite{bps18, bps19a, bps19b, Kne15}  to 
	  general bounded rational functions.
	 The properties studied include measures of non-tangential polynomial approximation, 
	 geometric constraints on boundary regions that witness singularities, and $L^p$-integrability of rational functions and
	 their derivatives. 
	 The next two subsections provide more detailed overviews of these two complementary goals:
	developing a local theory of stable polynomials and using it to characterize the structure and regularity of general bounded rational functions. 

	\subsection{Local theory of stable polynomials on $\mathbb{H}^2$}
	The local theory of stable polynomials is presented in Section \ref{sec:local}, 
	where we collect, refine, and significantly extend
	 results from both \cite{bps18, bps19a, Kne15} and other sources.
	We note that Agler, M\McC Carthy and Stankus \cite{AMS06, AMS08} studied the local geometry of zero sets near the boundary of the polydisk.
	Moreover, in analytic combinatorics and asymptotics in several variables,
	 various authors (see e.g. \cite{PemWil}, \cite{Tsi93}) have also investigated local
	aspects of stable polynomials. 

	Our local theory has three main tools:
	\textbf{homogeneous expansions}, Weierstrass and \textbf{Puiseux factorizations}, and \textbf{realization formulas}. 
	When possible, we consider $d$-variable stable polynomials, but some techniques require the restriction to $d=2$.
	The $d=2$ case is special because we have two tools at our disposal: Puiseux series and 
	transfer function type realization formulas \cite{AglMcCbook}.
	The latter tool is unavailable in more variables	
	due to the failure of And\^o's inequality
	in three or more variables \cite{Par70, Var74}.
	%Section \ref{sec:local} also includes a variety of examples.

	Let $p$ be a stable polynomial (here taken to mean no zeros in $\mathbb{H}^d$)
	 with $p(0,\dots,0)=0.$ Write $p = A +iB$, where $A,B$ are polynomials with real coefficients. 
	 Define the \textbf{reflection} $\bar{p}$ of $p$ by \[\bar{p}(z) = \overline{ p(\bar{z})}= A(z)- i B(z).\] 
	 The reflection operation creates a natural dichotomy in the class of stable polynomials; 
	 each stable $p$ factors as $p=p_1 p_2$ where $p_1$ is \textbf{pure stable}
	 meaning it has no factors in common with $\bar{p}_1$ and $p_2$ is \textbf{real stable}
	meaning $p_2 = c \bar{p}_2$ for some constant $c$ with $|c|=1$.  
	See Section \ref{sub:global} for details.
	
	Our primary result on homogeneous expansions is:

	\begin{theorem*}[\textbf{Homogeneous Expansions}] 
	Assume $p \in \C[z_1,\dots, z_d]$ has no zeros in $\UHP^d$.
	Write $p = A+iB = \sum_j P_j = \sum_j (A_j +iB_j)$, where $P_j$, $A_j$, $B_j$ are homogeneous polynomials
	of degree $j$ and $A_j, B_j$ have real coefficients. Let $M$ be the smallest number with $P_M \not \equiv 0$. Then,
		\begin{itemize}
			\item[a.] $P_M$ has no zeros on $\mathbb{H}^d$ and there is a unimodular $\mu$ such that $\mu P_M$ has real coefficients. 
			\item[b.] If $\mu=1$ and $p$ is pure stable, then $P_M = A_M$, $B_{M+1} \not \equiv 0$, and $\frac{A_M}{B_{M+1}}$ maps $\mathbb{H}^d$ to $\mathbb{H}$. 
		\end{itemize}
	\end{theorem*}
	The condition $\mu=1$ can be arranged by replacing $p$ with $\mu p$.
	Our formal statements are given as Theorem \ref{thm:stablehomog} and Propositions \ref{prop:B1}, \ref{prop:AB}. 
	The above theorem is particularly useful for studying the behavior of $p$ in nontangential regions
	near $\mathbb{R}^d$, the distinguished boundary of $\UHP^d$.  
	When $d=2$, part (b) implies that the tangents
	of $A$ and $B$ must interlace in a particular way, see Proposition \ref{prop:tangent}.
	We note that part (a) of our Homogeneous Expansion theorem was previously obtained in 
	Atiyah-Bott-G\r{a}rding \cite{ABG70} and discussed in Pemantle-Wilson \cite{PemWil}. 
	%Additional results concerning homogeneous expansions as well as a variety of examples can be found in Subsection \ref{subsec:homoexp}.

	When $d=2$, more refined local information about stable polynomials can also be extracted via their Puiseux factorizations.  
	Our primary theorem about Puiseux factorizations illustrates the dichotomy between pure and real stable polynomials:
	\begin{theorem*} [\textbf{Puiseux Factorizations}]  Let $p$ have Weierstrass factorization 
	$p = u p_1\cdots p_k$ near $(0,0)$, where $u$ is a unit and each $p_j$ is an irreducible Weierstrass polynomial in $z_2$ of degree $M_j$. 
		\begin{itemize}
			\item[a.] Let $p$ be a pure stable polynomial. Then, near $(0,0)$
			\[
			p(z) = u(z) \prod_{j=1}^k \prod_{m=1}^{M_j} (z_2 + q_j(z_1) + z_1^{2L_j} \psi_j(\mu_j^m z_1^{1/M_j})),
			\]
			where each $L_j$ is a positive integer, $\mu_j = e^{2\pi i/M_j}$, $q_j \in \R[z]$ with $q_j(0)=0, q_j'(0)>0$,
			$\deg q_j < 2L_j$, and $\psi_j$ is holomorphic near $0$  with $\Im (\psi_j(0)) > 0$.
			\item[b.] Let $p$ be a real stable polynomial without monomial factors.  Then, each $M_j =1$ and near $(0,0)$ 
			\[
			p(z) = u(z) \prod_{j=1}^k  (z_2 + \phi_j(z_1)),
			\]
			where  each $\phi_j$ is holomorphic near $0$ with real coefficients, $\phi_j(0)=0,$ and $\phi_j'(0)>0$.
		\end{itemize}
	\end{theorem*}
	
	As complicated as item (a) above looks, 
	a conceptual breakthrough in the present paper is that for
	most applications the Puiseux series parts, i.e. the $\psi_j$, can be forgotten
	and the most important data consists of
	\begin{itemize} 
	\item  the even integers $2L_1,\dots, 2L_k$ called local contact orders
	\item the real polynomials $q_1,\dots, q_k$ 
	\item the associated multiplicities $M_1,\dots, M_k$.
	\end{itemize}
	%The above theorem provides important local information about $p$.
	The maximum $K:= \max \{2L_1,\cdots, 2L_k\}$ is 
	called the \textbf{contact order} of $p$ at $(0,0)$
	because it measures how the zero set of a pure stable $p$, denoted $\mathcal{Z}_p$,
	 approaches $\mathbb{R}^2$.
	The papers \cite{bps18, bps19a} used contact order 
	to quantify regularity properties of related rational functions.

	This data is also reflected in natural perturbations of a pure stable $p$.
	%One use of our Puiseux factorization theorem is the analysis of important perturbations of a pure stable $p$. 
	Writing $p = A+iB$ as before, it turns out that $A+tB$ is real stable for $t\in \R$ while it is pure stable for $t \in \UHP$. 
	We are able to deduce two key properties about the family of polynomials $A+tB$. 
	First, for all but finitely many $t\in \R$, 
	\begin{equation} \label{intro-AtB}
	A(z)+tB(z) = u(z;t) \prod_{j=1}^{k} \prod_{m=1}^{M_j} (z_2 + q_j(z_1) + z_1^{2L_j}\psi_{j,m}(z_1;t)),
	\end{equation}
	using local contact orders and real polynomials from the Puiseux Factorization Theorem.
	%where $M_j$, $q_j$, $2L_j$ are the terms appearing in the Puiseux Factorization Theorem for $p$, 
	Here $\psi_{j,m}(\cdot; t)$ is analytic at $0$ and $u(z;t)$ is a unit. This shows that the initial segments of
	the branches of the zero sets of $A+tB$ agree with both each other and those of $\mathcal{Z}_p$, resolving the main problem left open in \cite{bps19a} (Conjecture 5.2).
	Second, defining $K_{min}:= \min \{2L_1,\cdots, 2L_k\}$, the \textbf{universal contact order} of $p$ at $(0,0)$,
	we have that
	 in the homogeneous expansion of the unit in \eqref{intro-AtB}
	  \[
	  u(z;t) = 1 + \sum_{j\geq 1} u_j(z;t),
	  \]
	  the polynomials $u_j(z;t)$ are affine in $t$ for $j \leq K_{min}-2$. 
	These results appear in Theorem \ref{thm:segments} and Theorem \ref{thm:ucohomog} and prove critical in our later investigations of rational function regularity.

 In the $d=2$ setting,  useful global operator theoretic 
	formulas related to stable polynomials such as transfer function realizations and determinantal representations  have also been established, see 
	Grinshpan, Kaliuzhnyi-Verbovetskyi, Vinnikov, Woerdeman \cite{GKVVW16,GKVVW16b}, Woerdeman \cite{Woe13}, and
	previous work by the second author \cite{Kne20} which itself was based on recent work of Dritschel \cite{mD18}. 
Such formulas often encode local information about $p$.  We provide an overview of some of them in Subsection \ref{ss:rf}, but we focus now on a formula particularly suited
to the study of general bounded rational functions $q/p$.
Here $p$ is pure stable on $\UHP^2$ and we normalize so that $|q/p|<1$ in $\UHP^2$.
	%For example, if $p$ is a stable polynomial on $\mathbb{H}^2$ with total degree $n$,  then there exist $c\in \C$ and $n \times n$ matrices $A_0, A_1, A_2$ with
	%$\text{Im}(A_0), A_1,A_2 \geq 0$, $A_1+A_2 = I$ such that
	%\[
	%p(z) = c \det(A_0 + A_1 z_1 + A_2 z_2),
	%\]
	%see \cite{}. 
%	Let $p$ be (pure) stable on $\mathbb{H}^2$ and 
%assume $q \in \C[z_1,z_2]$ satisfies $|q/p| <1$ on $\mathbb{H}^2$.
	Let $\gamma : \mathbb{D} \rightarrow \mathbb{H}$ be a M\"obius map. Then $f: = \gamma\circ (q/p)$ is rational, maps $\mathbb{H}^2 \rightarrow \mathbb{H}$,
	and is what is called a \textbf{rational Pick function}.
	We show that, up to conformal equivalence, every rational Pick function satisfies a particularly simple formula involving a matrix with positive imaginary part (PIP).  

	\begin{theorem*}[\textbf{PIP Realizations}]  Let $f$ be a nonconstant rational Pick
	function on $\mathbb{H}^2$. Then there exist conformal self-maps $\sigma_1, \sigma_2$ of $\mathbb{H}$ with $\sigma_2(0)=0$ such that for $g := \sigma_1 \circ f \circ (\sigma_2, \sigma_2)$,
	there exist 
		\begin{itemize}
			\item $c \in \mathbb{C}$, $\alpha, \beta \in \C^n$ for some positive integer $n$,
			\item $n\times n$ matrices $P_1,P_2$ satisfying $0 \le P_1,P_2 \le I$, and $P_1+P_2=I$, 
			\item an $n\times n$ matrix $S$ satisfying $\Im(S):= \tfrac{1}{2i} (S-S^*) \ge 0$
		\end{itemize}
	such that
	\begin{equation}\label{pipintro} g(z) = c + \beta^* \left(S+ z_1 P_1 +z_2 P_2 \right)^{-1} \alpha \quad \text{for } z\in \mathbb{H}^2,\end{equation}
	and  $\lim_{t\searrow 0} g(it,it) \ne \infty$. 
	%where $I$ is the $n\times n$ identity matrix
	\end{theorem*}

Note the limit $\lim_{t\searrow 0} g(it,it)$ exists in $\overline{\UHP}\cup \{\infty\}$ 
since $g$ restricted to the
diagonal is a one variable rational Pick function.
%%%%%%%%%%%%%%%%%%%%%%%%%%
	The formula \eqref{pipintro} is not valid for all rational Pick functions; Pick functions with certain
	singular behavior at $\infty$ must be represented with a more complicated formula as in the Type IV realizations of Agler, Tully-Doyle and Young
	\cite{ATDY16, PTDPick}.
	In essence, the conformal maps applied to $f$ are designed to
	avoid these global complications which should be unimportant to a local
	theory.
%%%%%%%%%%%%%%%%%%%%%%%%%

	The PIP realization theorem gives access to 
	boundary singular behavior of not just $p$ but also to a family of 
	perturbations $q+ \lambda p$
	where again $|q/p|<1$.  
	This is in contrast to our previous results which, while very detailed,
	focus on the less general 
	family of perturbations $A+tB$ which are
	simply multiples of the perturbations $\bar{p}+ \lambda p$ (for
	proper choice of $\lambda$ depending on $t$---see Remark \ref{AtBunimod}).

  \subsection{Structure and regularity of bounded rational functions}  \label{sec:struct}
  % on $\mathbb{D}^2$ and $\mathbb{H}^2$} 
	This paper presents a number of applications of the preceding theory
	to the study of rational functions $f=q/p$ which are bounded and analytic
	on either $\UHP^d$ or $\D^d$.  
	Often we rescale our functions so that $|f|\leq 1$, in which case $f$ is called a
	\textbf{rational Schur function} or \textbf{RSF}.
	While stable polynomials will always be the denominators of RSFs,
	stable polynomials are directly connected to 
	a special class of RSFs called
	\textbf{rational inner functions} or \textbf{RIFs} on $\mathbb{D}^d$ (resp. $\UHP^d$). 
 	These are	rational functions $\phi$ which are holomorphic 
	on $\mathbb{D}^d$ (resp. $\UHP^d$) and whose boundary values satisfy $|\phi(\zeta)| =1$
	for almost every $\zeta \in \mathbb{T}^d$ (resp. $\R^d$). 
	Here, $\mathbb{T}$ is the unit circle and $\mathbb{T}^d$
	is the \textbf{distinguished boundary} of $\mathbb{D}^d$.  

	An example RIF on $\D^2$ is
	  \begin{equation} \label{eqn:rif1} \phi(z_1, z_2) = \frac{2z_1z_2-z_1-z_2}{2-z_1-z_2}.\end{equation}
	RIFs are multivariate analogues of a crucial class of functions called finite Blaschke products 
	(see Garcia, Mashreghi and Ross \cite{GMRBook}),
 but unlike finite Blaschke products, RIFs can have boundary singularities
	as exhibited by \eqref{eqn:rif1} at $(1,1)$.
	RIFs appear frequently in multivariate function theory investigations
	for several reasons.
	Every bounded analytic function on $\mathbb{D}^d$ can be approximated locally uniformly
	by constant multiples of rational inner functions and so, results about RIFs can sometimes be generalized
	to all bounded analytic functions, see \cite{Rud69, Kne08}. In one and two variables, RIFs are the canonical solutions
	to Nevanlinna-Pick interpolation problems  and serve as essential examples of functions which preserve matrix inequalities,
	see \cite{AglMcCbook, AMY12b, PTDRegalPath, PTDRoyal, PTDPick}.
	RIFs also connect operator theory with systems and control engineering 
	(see e.g. Kummert \cite{Kum89} and Ball, Sadosky, and Vinnikov \cite{BSV05}).
	  
	%We do not belabor the point beyond the introduction, 
	%but the study of rational functions on the polydisk and poly upper-half plane are essentially equivalent. 
	%Local issues are often best studied in the upper half-plane setting where we focus
	%on the origin as a boundary point.
	%In the upper half-plane setting, an \textbf{RIF} on $\UHP^d$
	%is a rational function $\psi$ analytic on $\mathbb{H}^d$ with $|\psi(x)|=1$ for almost every $x \in \mathbb{R}^d$.
		
	The study of RIFs is made easier by the fact that there is a simple description of all RIFs on
	$\D^d$ or $\UHP^d$.	
	A theorem of Rudin and Stout adapted to $\UHP^d$ says that every RIF on $\UHP^d$ is of the form $\bar{p}/p$ for some pure stable $p$; 
	see \cite[Chapter 5]{Rud69}---a similar description is possible on $\D^d$.  
	A large amount of prior work has focused on regularity results for RIFs while the
	present work aims for general RSFs where no such simple description is known.
	Indeed, our first question is about exactly this lack of a simple description.
	
	Sections \ref{sec:ntlimits}-\ref{sec:integrability} are guided by the following questions:
	\begin{itemize}
		\item[Q1:] \textbf{Numerator Characterization.} Given a pure stable polynomial $p$, 
		for which polynomials $q$ is $\tfrac{q}{p}$ bounded on $\UHP^2$? (Section \ref{sec:numerator})
		\item[Q2:] \textbf{Non-tangential regularity.} How much non-tangential regularity do bounded rational functions have at singularities on the distinguished boundary? (Section \ref{sec:ntlimits})
		\item[Q3:]  \textbf{Horn regions.} How does the singular behavior of bounded rational functions
		manifest itself on the distinguished boundary (say, on $\mathbb{T}^2$ or $\mathbb{R}^2$)? (Section \ref{sec:horns})
		\item[Q4:] \textbf{Derivative integrability.} Can we determine the $L^p$ integrability of the partial derivatives of bounded rational functions near boundary singularities? (Section \ref{sec:integrability})
	\end{itemize}

	%Below,  answers to (Q1)-(Q4) generally appear in the $\mathbb{H}^d$ setting, but analogous polydisk results are immediate, and  we provide examples from the polydisk setting.
%	Dividing by a constant does not affect regularity results and so, we often  restrict to rational functions bounded by $1$;
%	in those cases, we say $f = \tfrac{q}{p}$ is a \textbf{rational Schur function} or \textbf{RSF} on 
%$\mathbb{D}^d$ (resp. $\mathbb{H}^d$) if $|f| <1$ on  $\mathbb{D}^d$ (resp. $\mathbb{H}^d$). 

	Our first question (\textbf{Q1}) goes to the heart of the theory and seeks a characterization of RSFs.
	One might naively conjecture that  $\tfrac{q}{p}$ would be bounded on $\mathbb{H}^2$ if $q$ vanishes to at least the same order as $p$ at the boundary zeros of $p$.
	This condition is necessary but not sufficient.
	A complete answer requires a detailed local analysis
	of $p$ at a boundary zero.
	Let us consider without loss of generality a zero at $(0,0)$.
	Since $\bar{p}/p$ is a rational inner function, it is clear that 
	if $q = f_1 p + f_2 \bar{p}$ for some functions $f_1,f_2$
	analytic at $(0,0)$,
	we have that $q/p$ is bounded on a neighborhood
	of $(0,0)$ intersected with $\UHP^2$.
In other words, if $q$ belongs to the ideal generated by $p$ and $\bar{p}$
in the ring $R_0 = \C\{z_1,z_2\}$ of convergent power series centered at $(0,0)$,
then $q/p$ is bounded in $\UHP^2$ near $(0,0)$.
	This ideal, denoted $(p,\bar{p})R_0$,
	 leads to a characterization of numerators
	in the generic yet simplest case of boundary zeros
	where $p$ vanishes to order $1$.  Higher order vanishing makes things more complicated
	as will be seen.  
	Nonetheless, we have the following local characterization, which appears as Theorem \ref{thm:onebranch}:

	\begin{theorem}
		Let $p$ be a pure stable polynomial on $\mathbb{H}^2$ and assume $p$ vanishes to order $1$ at $(0,0)$.
		For any  $q\in \mathbb{C}[z_1, z_2]$, the function $\tfrac{q}{p}$ is
		bounded on a neighborhood of $(0,0)$ intersected with $\mathbb{H}^2$ if and only if $q \in (p,\bar{p}) R_{0}$.
	\end{theorem}

	The forward and more difficult implication amounts to showing that
	boundedness of a rational function $\tfrac{q}{p}$ along
	certain curves derived from the Puiseux factorization of $p$
	forces $q$ to be in the given ideal.
	To state our results for higher order vanishing, we make reference to the Puiseux Factorization Theorem for pure stable polynomials and
	its notations.
	 Let us call each $z_2+q_j(z_1)$ an initial segment with cutoff $2L_j$ and multiplicity $M_j$.  We are able to identify a large subset of numerators $q$ where 
	$q/p$ is bounded near $(0,0)$ in $\UHP^2$ and we also achieve a conceptual reduction of the problem.
	Define the following product ideal
	\[
	\mathcal{I}  = \prod_{j=1}^{k} (z_2+q_j(z_1), z_1^{2L_j})^{M_j} R_{0}.
	\]
	Here $(z_2+q_j(z_1), z^{2L_j})$ is the ideal generated by $z_2+q_j(z_1)$
	and $z_1^{2L_j}$; $(z_2+q_j(z_1), z^{2L_j})^{M_j}$ is the ideal
	generated by $M_j$ products of elements of the former ideal; and $\mathcal{I}$
	is the product of all such ideals. It is worth noting that $\mathcal{I}$ is generally much
	larger than $(p,\bar{p}) R_{0}$.
	Also, define the polynomial
	\[
	[p](z) = \prod_{j=1}^{k} (z_2+ q_j(z_1) + i z_1^{2L_j})^{M_j}.
	\]
	We say a function is \textbf{locally $H^{\infty}$} if it is analytic and bounded
	on a neighborhood of $(0,0)$ intersected with $\UHP^2$.  This is
	to avoid confusion with  the concept of ``locally bounded''.

	\begin{theorem} \label{thm-intro-numcrit}
		Let $p$ be a pure stable polynomial on $\UHP^2$ with $p(0,0)=0$.
		Let $f\in \C[z_1,z_2]$.
		Then,
		\begin{itemize}
			\item If $f \in \mathcal{I}$, then $f/p$ is locally $H^{\infty}$.
			\item  $f/p$ is locally $H^{\infty}$ if and only if $f/[p]$ 
			is locally $H^{\infty}$.
			\item Suppose $p$ either has a double point, an ordinary multiple point, or repeated segments 
			(i.e. all $q_j(z_1)$ are the same). If $f/p$ is locally $H^{\infty}$ then $f \in \mathcal{I}$.
		\end{itemize}
	\end{theorem}

	See Theorems \ref{thm:segmentinclusion}, \ref{thm:squarep}, 
	\ref{thm:reverseinclusion}.
	From the terminology of the theory of algebraic curves, a \emph{double point} occurs when $p$ vanishes
	to order $2$ at $(0,0)$ and an \emph{ordinary multiple point} occurs when $p$ vanishes to order $M$ 
	and has $M$ distinct tangents. 
	The first and last items are the basis for the following general conjecture.

	\begin{conjecture}[\textbf{Full Numerator Criterion}]  \label{conj:fullnum} 
	Let $p$ be a pure stable polynomial on $\mathbb{H}^2$. 
		For any $f \in \C[z_1,z_2]$, $f/p$ is locally $H^{\infty}$ if and only if
		$f \in \mathcal{I}$.
	\end{conjecture} 

After the posting of this paper in September 2021, 
	Conjecture \ref{conj:fullnum} was affirmed in a preprint of J. Koll\'ar \cite{Kollar}.
	While this supersedes the third item of Theorem \ref{thm-intro-numcrit} 
	(essentially Theorem \ref{thm:reverseinclusion}), the details
	of our proof may be of independent interest.  

%%%%%%%%%%%%%%%%%%%%%%%%%%%%%%%%%%%%%%%%%%%%%%%%%%%%%%%%%%%%%%%%%%

	The conjecture suggests that, in many cases, the regularity of RSFs 
	 should either mirror (or be better than) that of related RIFs. Our answers to (\textbf{Q2})-(\textbf{Q4}) align with that intuition. 
	For example, the investigation of (\textbf{Q2}) is motivated by the fact that RIFs $\phi$ possess non-tangential limits, denoted $\phi^*(x)$,
	at \emph{every} distinguished boundary point $x$, see \cite{Kne15}.  (Roughly, saying $z \rightarrow x$ non-tangentially in $\mathbb{H}^d$ 
	means the quantities $|z_i-x_i|$ and $\Im(z_i)$  are all comparable as $z=(z_1,\dots, z_d) \rightarrow x=(x_1, \dots, x_d)$.)
	The existence of RSF limits could perhaps \emph{a priori} be more precarious because their
	modulus functions could encode additional singular behavior. This is illustrated by Figure \ref{faveRSFmodplot},
	which displays the modulus of the two-variable bounded rational function on $\mathbb{D}^2$:
	\begin{equation}
		f(z_1,z_2)=\frac{(z_1-1)(z_2-1)}{2-z_1-z_2}
		\label{eq:faveRSF}
	\end{equation} 
	on $\mathbb{T}^2$ identified  with $[-\pi,\pi)^2$, which is clearly discontinuous at $(1,1)$. 

	\begin{figure}[h!]
	      {\includegraphics[width=0.44 \textwidth]{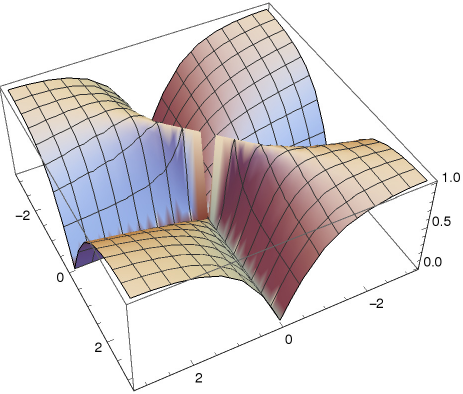}}
	  \caption{\textsl{Modulus of $f(z_1,z_2)=\frac{(z_1-1)(z_2-1)}{2-z_1-z
	_2}$ on $\mathbb{T}^2$.}}
	  \label{faveRSFmodplot}
	\end{figure}
	Despite such apparent obstructions, the following is true:

	\begin{theorem} \label{thm:introuhplim}
		If $f$ is a RSF in $\mathbb{H}^d$, then $f^*(x)$ exists at \textbf{every} point $x$ in $\mathbb{R}^d$.
	\end{theorem}

	Theorem \ref{thm:introuhplim} appears as Theorem \ref{thm:uhplim}. 
	Though we direct the reader to Section \ref{sec:ntlimits} for most details,
	we note that the local theory of stable polynomials also provides insights into higher-order regularity.
	Our Theorem \ref{thm:dird} characterizes when a RSF has directional derivatives and Theorem \ref{thm:nontanreg} characterizes 
	when it has non-tangential polynomial approximations to given orders at boundary points. 
	Our analysis leads to an interesting new result for RIFs quantifying the relationships 
	between contact order and non-tangential behavior; namely, \emph{if a pure stable polynomial $p$ has universal contact order $K$ at $x$, then $\tfrac{\bar{p}}{p}$
	has a non-tangential polynomial approximation of order $K-2$ at $x$}, see Theorem \ref{thm:ucoreg}. This is a partial converse to a key result (Theorem $7.1$) in \cite{bps18}.

%%%%%%%%%%%%%%%%%%%%%%%%%%%%%%%%%%%%%%%%%%%%%%%%%%%%%%%%%%%%%%%%%%

	As a complement to non-tangential behavior, (\textbf{Q3}) asks about``ultra-tangential'' behavior, namely behavior on the
	distinguished boundary. For RIFs, the papers \cite{bps18, bps19a}  addressed this by showing
	that if $\phi$ has a singularity at $(0,0)$, the unimodular level sets
	$\{ x \in \mathbb{R}^2: \phi(x) = \lambda\}$ for $\lambda \ne \phi^*(0,0)$ approach $(0,0)$ inside regions we call horns.
	A \emph{horn}  is a region with at least quadratic pinching (see Figure \ref{faveRSFlevelsplot}).
	Section \ref{sec:horns} shows that this result follows naturally, and more easily, from the Puiseux Factorization Theorem.
	 One way to interpret this RIF result is that if $(x_n)_n$ is a sequence in $\mathbb{R}^2$
	that manifests the discontinuity of $\phi$ at $(0,0)$, namely, if $x_n \rightarrow (0,0)$, but $(\phi(x_n))_n$ converges to some value different from $\phi^*(0,0)$,
	then the sequence $(x_n)_n$ becomes trapped in a union of horns. Our interpretation generalizes to the class of RSFs.

	\begin{theorem}\label{thm:introhorn}
		Suppose $f =\tfrac{q}{p}$ is an RSF with $p(0,0)=0$. If a sequence $(x_n)\subset \mathbb{R}^2$
		satisfies $x_n\to (0,0)$ and $f(x_n) \to c \ne f^*(0,0)$, then $(x_n)$ is eventually trapped in a finite union of horns.
	\end{theorem}

	 Theorem \ref{thm:introhorn} appears as Theorem \ref{thm:horn} and follows from a delicate local analysis
	of the formula from the PIP Realization Theorem.  A $\mathbb{T}^2$-horn associated to the function $f$ in
	\eqref{eq:faveRSF} is given in Figure \ref{faveRSFlevelsplot} below. As before, $\mathbb{T}^2$ is identified
	with $(-\pi,\pi]^2.$ As $f^*(1,1)=0$, our theorem shows that every non-zero level set of this RSF must eventually be caught inside a horn region along the line $y=-x$.

	\begin{figure}[h!]
	      {\includegraphics[width=0.3 \textwidth]{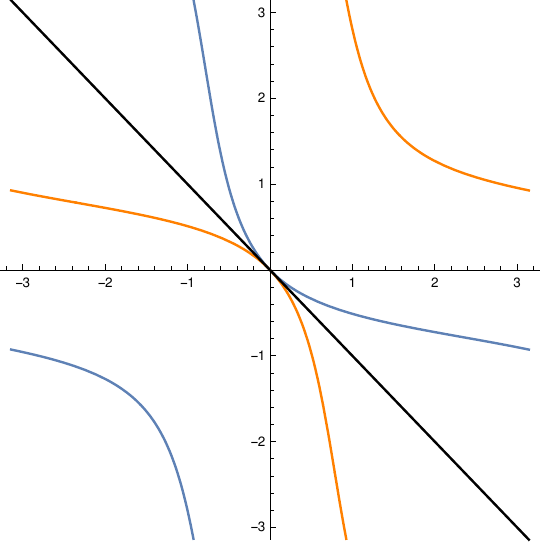}}
	  \caption{\textsl{A horn region for $f(z_1,z_2)=\frac{(z_1-1)(z_2-1)}{2-z_1-z_2}$. 
	The curves come from level sets with  $\lambda=1$ (black), $\lambda=\frac{1}{2}(1+i)$ (blue), and $\lambda=\frac{1}{2}(1-i)$ (orange).}}
	  \label{faveRSFlevelsplot}
	\end{figure}

%%%%%%%%%%%%%%%%%%%%%%%%%%%%%%%%%%%%%%%%%%%%%%%%%%%%%%%%%%%%%%%%%%

	Our last question (\textbf{Q4}) proposes an alternate measure of regularity on the distinguished boundary.
	Derivative integrability encodes singular behavior because it roughly measures the rate at which a function
	runs through different values near the singularity. We restrict (\textbf{Q4}) to the bidisk to ensure a bounded domain of integration--- allowing us to study
	global integrability of derivatives without added technical difficulties.
	For an RIF $\phi$, derivative integrability is known to be governed by contact order \cite{bps18, bps19a};
	indeed, $\tfrac{\partial \phi}{dz_1} \in L^\p(\mathbb{T}^2)$ if and only if $\p < \tfrac{K+1}{K},$ where $K$ is the maximum contact order of $\phi$ at its singularities on $\mathbb{T}^2$.  
	Section \ref{sec:integrability} shows that the partial derivatives of RSFs possess nice integrability properties:

	\begin{theorem}\label{thm:introii}
		Let $p$ be a stable polynomial on $\mathbb{D}^2$ with finitely many zeros on $\mathbb{T}^2$. Then there is a finite list of numbers $\p_1,\ldots, \p_M \in (0,\infty]$ 
		such that for any $q\in \mathbb{C}[z_1,z_2]$,  
		\begin{equation} \label{eqn:II} \sup_{\p' > 0}\left\{\p'\colon \partial_{z_1}(q/p)\in L^{\p'}(\mathbb{T}^2)\right\} \end{equation}
		is equal to one of the $\p_j$'s. Moreover, the number $M$ is bounded by an algebraic characteristic of $p$.
	\end{theorem}

	Theorem \ref{thm:introii} appears as Theorem \ref{thm:ii}. Under additional (generic) assumptions on $p$,
	in Theorem \ref{thm:RSFii} we obtain an exact characterization of the numbers $\p_1, \dots, \p_M$ in terms of contact order.
	For example, when $p=2-z_1-z_2$, the list of numbers satisfying \eqref{eqn:II} for some $q$ is exactly $\tfrac{3}{2}, 3, \infty.$
	Indeed, the function in \eqref{eq:faveRSF} has $\tfrac{\partial f}{dz_1} \in L^{\p}(\mathbb{T}^2)$ if and only if $\p<\frac{3}{2}$.
	Our exact characterization of $\p_1, \dots, \p_M$ relies on results in Section \ref{sec:numerator} (discussed earlier) and hence, relies on the Puiseux Factorization Theorem. 

	\subsection{Structure of the paper} 
		Section \ref{sec:local} details our local theory of stable polynomials,
		including results on homogeneous expansions, Puiseux factorizations, and realization formulas.
		Section \ref{sec:numerator} investigates the numerator criterion question while Sections \ref{sec:ntlimits}, \ref{sec:horns}, \ref{sec:integrability}
		respectively address non-tangential regularity, horn regions at singularities, and derivative integrability for RSFs.
		While this introduction provides an overview of key results, the reader should consult each section for precise definitions,
		additional results, and a variety of examples which are not mentioned here.  

%\section*{Acknowledgments}
%Part of this paper was completed during visits to Stockholm University and the University of Florida, 
%whom we thank for their hospitality and stimulating research environments.  
%The authors would like to sincerely thank Professor Ryan Tully-Doyle for
%helpful comments on a draft of this paper.
%Finally, thank you to the anonymous referees for numerous
%suggestions and a thorough reading of the paper.

%\newpage
%%%%%%%%%%%%%%%%%%%%%%%%%%%%%%%%%%%%%%%%%%%%%%%%%%%%%%%%%%%%%%%%%%
%%%%%%%%%%%%%%%%%%%%%%%%%%%%%%%%%%%%%%%%%%%%%%%%%%%%%%%%%%%%%%%%%%

\section{A local theory of stable polynomials} \label{sec:local}

\subsection{Some global theory of stable polynomials}\label{sub:global}
		We now describe a dichotomy for 
		stable polynomials already alluded to in the Puiseux factorization theorem.
		Very roughly, stable polynomials factor into a polynomial that 
		vanishes very little on the distinguished boundary and a polynomial that
		vanishes a lot on the distinguished boundary.  We also
		review some basic notions of stable polynomials on $\D^d$ compared to $\UHP^d$.

		 Recall that the reflection of a polynomial in the context of $\UHP^d$
		 is given by
		 \[
			\bar{p}(z) = \overline{p(\bar{z})}
		\]
		for $z \in \C^d$.  In the context of $\D^d$ the reflection 
		is degree dependent; if $p\in \C[z_1,\dots, z_d]$ has
		multidegree $n=(n_1,\dots, n_d)$ (i.e. degree $n_j$ in $z_j$)
		then the reflection of $p$ is
		\[
			\tilde{p}(z) = z^n  \overline{p(1/\bar{z}_1,\dots, 1/\bar{z}_d)} \text{ where } z^n := z_1^{n_1} \cdots z_d^{n_d}.
		\]
		Note the Cayley transform converts between the two notions of reflection: $\tilde{p}$ in the $\D^d$ setting versus $\bar{p}$ in the $\UHP^d$ setting.
		If $p\in \C[z_1,\dots, z_d],$ viewed as a function on $\D^d,$ has multidegree $n$ then we convert to a polynomial
		in the setting of $\UHP^d$ via
		\begin{align} \label{eqn:PH2}
			P(z)  &= (\mathbf{1} -i z)^n p\left( \frac{\mathbf{1}+iz}{\mathbf{1} - iz}\right)\\
			& = (1-iz_1)^{n_1}\cdots (1-iz_d)^{n_d} p\left( \frac{1+iz_1}{1-iz_1}, \dots, \frac{1+iz_d}{1-iz_d}\right) \nonumber
		\end{align}
		where $\mathbf{1} = (1,\dots,1) \in \C^d$ and we use convenient and temporary component-wise shorthands.
		Then, 
		\[
			\bar{P}(z) = (\mathbf{1} + i z)^n  \overline{p\left( \frac{\mathbf{1}+i\bar{z}}{\mathbf{1} -i\bar{z}}\right)}
			= (\mathbf{1}-iz)^n \tilde{p}\left( \frac{\mathbf{1}+iz}{\mathbf{1} - iz}\right) 
		\]
		shows a direct correspondence between the notions of reflection.
		It is straightforward to check that if $p\in \C[z_1,\dots, z_d]$ has
		no zeros in $\D^d$, then $\phi = \tilde{p}/p$ is a rational inner function on $\D^d$.
		Namely, $|\phi|  \leq 1$ in $\D^d$ and $|\phi|=1$ on $\T^d$ outside the zero set of $p$.
		(This is obvious if $p$ has no zeros on $\overline{\D^d}$ by the maximum principle;
		otherwise one can examine $p(rz)$ as $r\nearrow 1$.)   This type of homothety
		is unavailable in $\UHP$ so Cayley transform is the easiest
		way to see that if $p\in \C[z_1,\dots, z_d]$ has no zeros in
		$\UHP^d$ then $\bar{p}/p$ is a rational inner function on $\UHP^d$.

		Let us now review a basic dichotomy of stable polynomials.
		Any $p\in \C[z_1,\dots, z_d]$ with no zeros in $\D^d$ 
		can be factored into $p=p_1p_2$ where $p_1$ has no factors in common with $\tilde{p}_1$
		and $p_2$ is a constant multiple of $\tilde{p}_2$.  
		Indeed, writing $p = p_1 p_2$ and $\tilde{p} = q_1 p_2$ where $p_1$ and $q_1$ have
		no common factors we see that $p_2$ has no zeros in $\D^d \cup \{z\in \C: |z|>1\}^d$. 
		The rest follows from the next standard lemma.

		\begin{lemma}
			Any $q\in \C[z_1,\dots, z_d]$ with no zeros in $\D^d \cup \{z\in \C: |z|>1\}^d$ is a multiple of $\tilde{q}$.
		\end{lemma}
		\begin{proof}
			Note that for any $\tau \in \T^d$, the 
			one variable function $f(z):= q(z \tau)$ only has zeros
			on $\T$ implying $\tilde{f}$ is a multiple of $f$ and therefore $|\tilde{f}/f| = 1$.
			But $\tau^n \tilde{f}(z) =\tilde{q}(\tau z)$ 
			so that $\phi = \tilde{q}/q$ attains modulus $1$ inside $\D^d$.
			Now $\phi$ is rational inner and therefore, by the maximum principle, the function
			$\phi$ must be constant and unimodular.
		\end{proof}

		Note that all polynomial factors of $q$ in the lemma satisfy the same hypothesis and conclusion as the lemma.

		The following terminology is borrowed from \cite{AMS06}.
		We shall call polynomials of the type $p_1$ \emph{atoral stable}.
		Atoral stable polynomials arise as the denominators of rational inner functions.
		Polynomials of the type $p_2$ are called \emph{toral stable}.
		These arise as defining polynomials for unimodular level sets of rational inner functions.
		Namely, given a nonconstant rational inner function $\phi = z^{\alpha} \tilde{p}/p$ ($\alpha \in \mathbb{N}^d$)
		and $\mu \in \T$, the set
		\[
		\{ z\in \C^d: \phi(z) = \mu\} 
		\]
		can be described by $\{z\in \C^d: z^{\alpha} \tilde{p}(z) - \mu p(z) = 0\}$ 
		if we ignore zeros of $p$.  Note that $z^{\alpha} \tilde{p} - \mu p$ is
		toral stable because it is non-vanishing in $\D^d$ (as a limit of $z^{\alpha} \tilde{p} - w p$ for $w\to \mu$ with $|w| \searrow 1$)
		and is a constant multiple of its reflection.

		\begin{figure}
		%% stabletable.pdf
		 {\includegraphics[width=1.0 \textwidth]{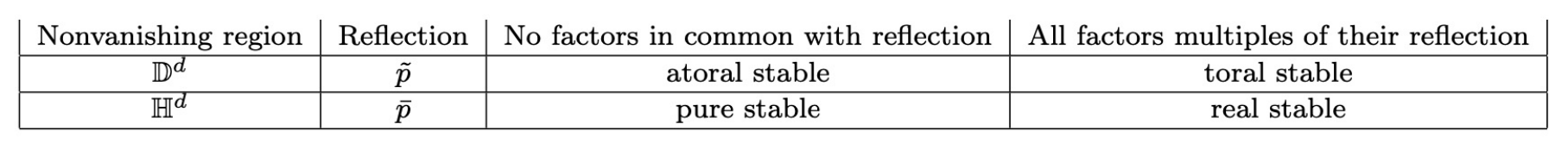}}
		
		\iffalse
		{\tiny \centerline{
		\begin{tabular}{|c|c|c|c|}
		 Nonvanishing region & Reflection & No factors in common with reflection & All factors multiples of their reflection \\ \hline
		 $\D^d$ & $\tilde{p}$ & atoral stable & toral stable \\ \hline
		 $\UHP^d$ & $\bar{p}$ & pure stable & real stable \\ \hline \end{tabular}}}
		 \fi
		 \caption{Dichotomies of stable polynomials}
		 \end{figure}

		In the upper half plane setting we have a dichotomy analogous to atoral versus toral.
		For $p\in \C[z_1,\dots, z_d]$ with no zeros in $\UHP^d$ we can factor
		$p = p_1 p_2$ where $p_1$ has no factors in common with $\bar{p}_1$
		and $p_2$ is a constant multiple of $\bar{p}_2$.  
		As above, $p_2$ will have no zeros in $\UHP^d\cup (-\UHP)^d$
		and this property alone implies $p_2$ is a multiple of $\bar{p}_2$.
		So, every factor of $p_2$ is a multiple of its reflection.
		We can then arrange for $p_2$ and all of its factors to have real coefficients by transferring a constant over to $p_1$.
		We shall call $p_1$ type polynomials \emph{pure stable} and type $p_2$ polynomials \emph{real stable}.
		``Pure stable'' is not common parlance.  Real stable refers to the fact
		that in one variable a real stable polynomial has all of its roots on the real axis.
		We end this section with remarks about homogeneous polynomials and distinguished boundary zero sets for stable
		polynomials in two variables.

		A homogeneous polynomial $P\in \C[z_1,\dots, z_d]$ with no 
		zeros in $\UHP^d$ is automatically real stable since $P(-z)$
		is a multiple of $P$.  If $d=2$, then such a homogeneous polynomial
		factors as
		\[
		P(z_1,z_2) = c \prod_{j=1}^{M} (a_j z_1 + b_j z_2)
		\]
		with $c \in \C, a_j,b_j \in \R$.
		Note $a z_1+b z_2$ is non-vanishing in $\UHP^2$ if and only if 
		$a ,b$ have the same sign, so we can further arrange $a_j,b_j \geq 0$ for all $j$
		by absorbing sign changes into $c$.  Thus, $P$ is a multiple of a homogeneous polynomial 
		in $\R_{+}[z_1,z_2]$, the polynomials with non-negative coefficients.

		In two variables, atoral (resp. pure) stable polynomials have finitely many zeros on $\T^2$ (resp. $\R^2$)
		which follows from B\'ezout's theorem since zeros on $\T^2$ 
		are zeros in common with $\tilde{p}$.
		%In fact an atoral stable polynomial $p$ of bidegree $n=(n_1,n_2)$ has
		%$2n_1 n_2$ zeros in common with $\tilde{p}$ on the product of
		%Riemann spheres $\C_{\infty}\times \C_{\infty}$ when counted using intersection
		%multiplicities.  Since zeros of $p$ on $\T^2$ are
		%necessarily common zeros of $\tilde{p}$, $p$ has
		%at most $2n_1n_2$ zeros on $\T^2$ when counted
		%via intersection multiplicity with $\tilde{p}$.  
		%The intersection multiplicity of a common zero of $p$ and $\tilde{p}$
		%on $\T^2$ is even so $p$ has at most $n_1n_2$ distinct zeros on $\T^2$.
		In the pure stable case one should
		keep in mind that there can be common zeros on $(\R\times\{\infty\})\cup(\{\infty\}\times \R) \cup \{(\infty,\infty)\}$
		(e.g. $p$ vanishes at $(\infty,\infty)$ if $z_1^{n_1} z_2^{n_2} p(1/z_1,1/z_2)$ vanishes at $(0,0)$).
		Toral stable and real stable polynomials in two variables have curves 
		of zeros on the distinguished boundaries $\T^2$ (resp. $\R^2$) and no isolated zeros.  Later on 
		we present a local parametrization theorem for real stable
		polynomials, Theorem \ref{thm:localpar}, which states that locally the zero set of a real stable 
		polynomial on $\R^2$ is a union of smooth curves.  Of course away from 
		finitely many singularities the zero set will locally consist of a single smooth curve.

	\subsection{Homogeneous expansions}\label{subsec:homoexp}

		%The following theorem is the key to Theorem (A) on non-tangential limits
		%of rational Schur functions.
		 We now discuss the Homogeneous Expansion
		Theorem from the introduction
		in Theorem \ref{thm:stablehomog} and Propositions \ref{prop:B1}, \ref{prop:AB}.

		\begin{theorem}\label{thm:stablehomog}
			Suppose $p\in \C[z_1,\dots, z_d]$ has no zeros in $\UHP^d$ and $p(0)=0$.
			We may decompose $p$ into homogeneous polynomials
			\[
			p(z) = \sum_{j=M}^{n} P_j(z)
			\]
			where $P_j\in \C[z_1,\dots,z_d]$ is homogeneous of degree $j$, $P_M \not\equiv 0$, and $n = n_1 + \dots + n_d$ is the total degree of $p$.    
			Then, $P_M$ has no zeros in $\UHP^d$
			and by homogeneity $P_M$ is necessarily real stable.
			In particular, there exists $\mu\in \T$ such that $\mu P_M \in \R[z_1,\dots, z_d]$.
		\end{theorem}

		Theorem \ref{thm:stablehomog} follows from work in Atiyah, Bott and G\r{a}rding\cite{ABG70} (Lemma 3.42) but as discussed in \cite{PemWil} (Proposition 11.1.6)
		it directly follows from Hurwitz's theorem applied to 
		\[
			P_M(z) = \lim_{t\to 0^{+}} \frac{1}{t^{M}} p(tz).
		\]

		\begin{example}\label{ex71}
			The polynomial $p(z_1,z_2) = 4-3z_1-3z_2+z_1^2z_2+z_1z_2^2$,
			taken from Example 7.1 in \cite{bps19a},
			has no zeros on $\D^2$.  Conformal mapping yields
			\[
			\begin{aligned}
			P(z_1,z_2) &= (1-iz_1)^2 (1-iz_2)^2 p\left(\frac{1+iz_1}{1-iz_1}, \frac{1+iz_2}{1-iz_2}\right) \\
			&= -4(z_2^2+ 4z_1z_2+ z_1^2 -2z_1^2 z_2^2 - 4iz_1z_2(z_1+z_2))
			\end{aligned}
			\]
			which has no zeros in $\UHP^2$.
			The bottom homogeneous term 
			\[
			-4(z_2^2+4z_1z_2+ z_1^2) = -4(z_2+(2+\sqrt{3})z_1)(z_2+(2-\sqrt{3})z_1)
			\]
			evidently has no zeros in $\UHP^2$ and has real coefficients.
			We return to this example in Examples \ref{ex71:return}, \ref{ex71puiseux}. 
			\eox
		\end{example}

		When $p\in \C[z_1,\dots, z_d]$ is pure stable its 
		homogeneous expansion has additional structure.
		The real and imaginary parts (via coefficients) of $p$ are
		\[
			A:=\tfrac{1}{2}(p+\bar{p}) \text{ and } B := \tfrac{1}{2i}(p-\bar{p})
		\]
		so that $A,B\in \R[z_1,\dots, z_d]$
		and $p=A+iB$.
\begin{remark}\label{AtBrs}
Notice that $A+tB$ is real stable for $t \in \R$ since
$A+tB = \frac{1}{2}( (1-it)p + (1+it) \bar{p}) = 0$
exactly when $\bar{p}/p = -\frac{1-it}{1+it}$ which happens to 
be a point on $\T$. Since $\bar{p}/p$ is a nontrivial
RIF this can only occur outside of $\UHP^d$.
Similarly, $B$ is also real stable.
\end{remark}

	\begin{proposition} \label{prop:B1}
		Suppose $p \in \C[z_1,\dots, z_d]$ is a pure stable polynomial
		and $p(0)=0$. 
		Write the homogeneous
		expansion of $p = \sum_{j\geq M} P_j$.  
		By Theorem \ref{thm:stablehomog},
		we may replace $p$ with a unimodular multiple so that $P_M = \bar{P}_M \in \R[z_1,\dots,z_d]$.
		Then, the lowest order homogeneous term of $A$ equals $P_M$ and $B$
		vanishes to order exactly $M+1$.
	\end{proposition}

	\begin{proof}
		We have arranged for the lowest order homogeneous term of $p$
		to have real coefficients so we must have $A_M = P_M$ and $B_M = 0$.
		Assuming $p$ and $\bar{p}$ have no common
		factors implies $\phi=\bar{p}/p$ is non-constant and $|\phi| < 1$ inside $\UHP^d$.
		Set $\tau = (1,\dots, 1)$.  
		This implies $\zeta \mapsto p(\zeta \tau)$ has at least
		one root in the lower half plane otherwise
		it would have all real zeros and $\bar{p}(\zeta \tau)/p(\zeta \tau)$
		would be constant and unimodular.
		Now,
		\[
		\begin{aligned}
		p(\zeta \tau) &= A_M(\tau) \zeta^M + (A_{M+1}(\tau)+i B_{M+1}(\tau))\zeta^{M+1} + \cdots\\
		&= A_M(\tau)\zeta^M( 1+ (a+ib)\zeta + \cdots) 
		\end{aligned}
		\]
		where $a+ib = (A_{M+1}(\tau)+i B_{M+1}(\tau))/A_M(\tau)$.
		On the other hand if we factor $p(\zeta \tau)$
		\[
		\begin{aligned}
		p(\zeta \tau) &= A_M(\tau)\zeta^M \prod_{j}(1+\alpha_j \zeta) \\
		&= A_M(\tau)\zeta^M(1 + (\sum_j \alpha_j) \zeta + \cdots )
		\end{aligned}
		\]
		where $\alpha_j$ are in the closed lower half plane
		we see that $a+ib = \sum_j \alpha_j$.
		But at least one $\alpha_j$ must be in the open lower
		half plane so $b\ne 0$ and hence $B_{M+1}(\tau)\ne 0$.
	\end{proof}

\begin{remark}\label{AtBunimod}
	The above proposition has implications for the geometry of 
	unimodular level sets of rational inner functions.
	The unimodular level sets of $\phi = \bar{p}/p$ are given by 
	$A+t B \equiv 0$ for $t\in \R$ or $B\equiv 0$ 
	(corresponding to $t =\infty$).
	Indeed, for $\mu \in \T$,
	 the zero set $\bar{p}-\mu p = 0$ is the same
	as $A(1-\mu) -i B(1+\mu) = 0$
	or $A+ tB = 0$ for $t= i\frac{\mu+1}{\mu-1} \in \R$ when $\mu \ne 1$
	and $B=0$ for $\mu=1$.
	Therefore, the polynomials defining the unimodular level
	sets of $\phi$ all have initial homogeneous term $A_M$ with 
	the exception of the level set $B=0$ which has initial
	homogeneous term $B_{M+1}$ of one degree higher.  
	In two variables, this has the more direct geometric 
	interpretation that all of the unimodular level curves with the exception of $B=0$ 
	have the same set of tangents; namely the factors of $A_M$.  
\end{remark}

	The next proposition refines the relationship between $A_M$ and $B_{M+1}$.
	Note that $A/B = i\frac{1+\bar{p}/p}{1-\bar{p}/p}$ maps $\UHP^d$ into $\UHP$. 

	\begin{proposition} \label{prop:AB}
		With the setup of the previous proposition,
		let $A_M, B_{M+1}$ be the lowest order homogeneous terms of $A, B$. 
		Then, $A_M/B_{M+1}$ is a Pick function, i.e. 
		\[
		\text{Im}\left(\frac{A_M}{B_{M+1}}\right) \geq 0 \text{ in } \UHP^d.
		\]
	\end{proposition}

	\begin{proof}
		Since $B$ has no zeros in $\UHP^d$, $B_{M+1}$ has no zeros in $\UHP^d$.
		Since $A/B$ is a Pick function we can let $t>0$ and $z \in \UHP^d$
		and consider
		\[
		0\leq t \text{Im}\left(\frac{A(tz)}{B(tz)}\right) = \text{Im}\left( \frac{A_M(z) + t A_{M+1}(z) +\cdots}{B_{M+1}(z) + tB_{M+2}(z) +\cdots}\right).
		\] 
		Send $t\to 0$ to see that $A_M/B_{M+1}$ is a Pick function.
	\end{proof}

	In two variables this has the more geometric interpretation
	that the tangents of $A_{M}$ interlace the tangents of $B_{M+1}$. 
	The following proposition encodes that fact and introduces an added level of generality. Namely, it considers pairs of homogeneous polynomials, which we still denote $A_M, B_{M+1},$ that could possess monomial factors.
	We will see in the next section that if $p = A+i B$ is pure stable then we can 
	arrange for $A_M$ to have no factors of $z_1$ or $z_2$.
	Hence, the following proposition handles some pairs $A_M, B_{M+1}$ that do not originate from a pure stable $p$.

	\begin{proposition} \label{prop:tangent}
	Write 
	\[
	A_M = a z_1^r \prod_{j=1}^{M-r}(z_2+a_j z_1), \quad 
	B_{M+1} = b z_1^s \prod_{j=1}^{M+1-s}(z_2+b_j z_1)\]
	where $0\leq a_1\leq\cdots \leq a_{M-r}$, $0\leq b_1 \leq \cdots \leq b_{M+1-s}$.
	Set $a_j=\infty$ for $j=M-r+1,\dots, M$ and $b_j = \infty$ for $j=M+2-s,\dots, M+1$ if $r$ or $s$ are nonzero.
	Suppose $A_M/B_{M+1}$ is  a Pick function.
	Then, $r=s$ or $r+1=s$ and
	\[
	b_1\leq a_1\leq b_2 \leq \cdots \leq a_{M} \leq b_{M+1}.
	\]
	\end{proposition}

	\begin{proof}
	Suppose first that $A_M,B_{M+1}$ have no
	factors of $z_1$. 
	Then, $A_M(z_1,z_2) = a \prod_{j=1}^{M} (z_2+a_j z_1)$
	for $a_j \geq 0$ and $B_{M+1}(z_1,z_2) = b \prod_{j=1}^{M+1}(z_2+b_j z_1)$
	for $b_j\geq 0$ because $A_M$ and $B_{M+1}$ and all of their
	(linear) factors are real stable.  
	Then, 
	\[
	z_2\mapsto \frac{A_M(1,z_2)}{B_{M+1}(1,z_2)} = \frac{a}{b} \frac{\prod_{j=1}^{M}(z_2+a_j)}{\prod_{j=1}^{M+1} (z_2+b_j)}
	\]
	is a one variable real rational Pick function.
	By Lemma 6.5 of \cite{bps19a}, the zeros (or rather their negatives) must interlace.  
	Namely, if we write $0\leq a_1 \leq \cdots \leq a_M$, $0\leq b_1 \leq \cdots \leq b_{M+1}$
	then
	\[
	b_1 \leq a_1 \leq b_2 \leq \cdots \leq a_M \leq b_{M+1}.
	\]
	Lemma 6.5 of \cite{bps19a} also says $a/b<0$.  
	If $A_M$ or $B_{M+1}$ has a factor of $z_1$,
	write $A_M = z_1^r A^{\flat}_{M-r}$, $B_{M+1} = z_1^s B^{\flat}_{M+1-s}$ where $A^{\flat}_{M-r}, B^{\flat}_{M+1-s}$ have no factors of $z_1$.
	For $t>0$
	\[
	t^{s-r} \frac{A_M(tz_1,z_2)}{B_{M+1}(tz_1,z_2)} = z_1^{r-s} \frac{A^{\flat}_{M-r}(tz_1,z_2)}{B^{\flat}_{M+1-s}(tz_1,z_2)}
	\]
	is still a Pick function and if we send $t\to 0$ we get the Pick function
	\[
	z_1^{r-s} \frac{A^{\flat}_{M-r}(0,z_2)}{B^{\flat}_{M+1-s}(0,z_2)} = c z_1^{r-s} z_2^{s-r-1}
	\]
	for some constant $c$.  This is only possible if $s=r+1$ or $s=r$ and $c<0$.  
	We view this situation as $A_M$ having $r$ infinite slopes which then implies $B_{M+1}$
	has $r$ or $r+1$ infinite slopes.  
	If $r=s$, $A^{\flat}_{M-r}$ and $B^{\flat}_{M+1-r}$ have $M-r$ and $M+1-r$ slopes that interlace 
	as before and the addition of infinite slopes does not change the interlacing property.
	If $s=r+1$, then $A^{\flat}_{M-r}(1,z)/B^{\flat}_{M-r}(1,z)$ is a one variable Pick
	function so the $M-r$ roots of the numerator and denominator interlace.
	Since the ratio of the leading coefficients is negative,
	Lemma 6.5 of \cite{bps19a} states that the smallest slope of $B_{M+1}$
	is smaller than the smallest slope of $A_M$.  
	Thus, the $M-r$-th slope of $B_{M+1}$ is at most the $M-r$-th slope of $A_M$
	and the remaining $r+1$ infinite slopes of $B_{M+1}$ interlace with the remaining $r$ 
	infinite slopes of $A_M$.
	\end{proof}

	If $p=A+iB$ is pure stable with $P_M = A_M$, then the above proposition says that the tangents of $A_M$ and $B_{M+1}$ interlace. 
	Since these are the initial homogeneous terms of $A$ and $B$, we
	can say that their tangents interlace as well as those of $A+tB$ and $B$.

	\begin{example} \label{ex71:return}
		Returning to Example \ref{ex71},
		we have
		\[
		A = -4(z_2^2+4z_1z_2+ z_1^2-2z_1^2z_2^2) \quad B = 4z_1z_2(z_1+z_2).
		\]
		So, 
		\[
		\frac{A_2}{B_3} = -\frac{z_2^2+4z_1z_2+z_1^2}{z_1z_2(z_1+z_2)}
		\]
		is a Pick function.
	\end{example}

	\begin{example}\label{ex153homog}
		The Pick function $A_M/B_{M+1}$ need not be especially interesting.
		Consider the polynomial with no zeros on $\D^2$
		\[
		p(z_1,z_2) = 4-5z_1-2z_2+2z_1 z_2+3z_1^2-z_1^2z_2-z_1^3z_2
		\]
		(taken from \cite{Kne15} Example 15.3)
		converted to the polynomial with no zeros on $\UHP^2$
		given by
		\[
		P(z_1,z_2) = z_1+z_2-2z_1^3-6z_1^2 z_2 - i(z_1^2+z_1z_2-4z_1^3z_2).
		\]
		(A rescaling was also involved.)
		We have 
		\[
		\frac{A_1}{B_2} =\frac{(z_1+z_2)}{-z_1(z_1+z_2)}= -\frac{1}{z_1}.
		\] 
		We return to this example in Examples \ref{ex153puiseux}, \ref{ex:153}, \ref{ex:153:ideal}. \eox
	\end{example} 

	\begin{remark}\label{rem:uhptodisk}
		The above properties of lowest order homogeneous
		terms transfer to the polydisk via Cayley transform.

		Suppose $p\in \C[z_1,\dots,z_d]$, $p(1,\dots, 1)=0$,  and $p$ has
		multi-degree $n=(n_1,\dots,n_d)$.
		Then,
		\[
		P(z) = (1-iz_1)^{n_1}\cdots (1-iz_d)^{n_d} p\left(\frac{1+iz_1}{1-iz_1},\dots, \frac{1+iz_d}{1-iz_d}\right) \in \C[z_1,\dots, z_d]
		\]
		has zeros in $\UHP^d$ if and only if $p$ has zeros in $\D^d$; also, $P(0)=0$.  
		In order to compare homogeneous decompositions
		we write
		\[
		p(1+z_1,\dots,1+z_d) = \sum_{j\geq M} p_j(z) \text{ and } P(z) = \sum_{j\geq M} P_j(z).
		\]
		Since $\frac{1+iz}{1-iz} = 1+\frac{2iz}{1-iz}$ one can directly
		check by looking at the lowest order terms that
		\[
		P_M(z) = (2i)^M p_M(z).
		\]
		Thus, if $p$ has no zeros in $\D^d$, then $p_M$ has
		no zeros in $\UHP^d$ (or $(c\UHP)^d$ for any $c\ne 0$ 
		by homogeneity).
		This is how a version of Theorem \ref{thm:stablehomog}
		was originally stated in \cite{Kne15}.  By this correspondence
		the theorem from \cite{Kne15} directly implies
		Theorem \ref{thm:stablehomog}.
		 
		The finer properties involving $A_M, B_{M+1}$ would be
		more technical to state in the context of the polydisk.
	\end{remark}

\subsection{Puiseux expansions} \label{sec:puiseux}
While homogeneous expansions can provide 
some useful local information about stable polynomials
in two and more variables, using Puiseux expansions
we can give a nearly complete local description
of stable polynomials in two variables.  
Given $p\in \C[z_1,z_2]$ with
no zeros in $\UHP^2$, $p(0,0)=0$, 
 we can factor it as
\[
p = z_1^{\alpha_1} z_2^{\alpha_2} u p_1 \cdots p_N
\]
where $\alpha_1,\alpha_2$ are non-negative integers, 
$u\in \C\{z_1,z_2\}$ is analytic and non-vanishing at $0$
and $p_j \in \cps[z_2]$ are monic and irreducible in $z_2$ 
with coefficients in the ring of convergent power series
$\cps$ that vanish at $0$ (i.e.\ irreducible Weierstrass polynomials).
Note $\C\{z_1,z_2\}, \cps$ denote rings of convergent power series.

By Puiseux's theorem each $p_j$ formally factors into
\[
\prod_{k=1}^{m} (z_2 - g(\mu^k z_1^{1/m}))
\]
for $g\in \cps$, $g(0)=0$, and $\mu = e^{2\pi i/m}$ (see \cite{Simon} Theorems 3.5.1, 3.5.2).  Alternatively,
the zero set of $p_j$ near $(0,0)$ is parametrized by the map
\[
t \mapsto (t^m, g(t))
\]
defined for $t$ in a neighborhood of $0\in \C$.
Notice that since $p$ has no zeros in $\UHP^2$, the above map has
the property that it is injective and maps into $\C^2\setminus \UHP^2$.  
The following theorem derived from \cite{Kne15} gives a detailed
description of the possible $g\in \cps$.  

\begin{theorem}\label{thm:localtypes}
Let $g\in \C\{z\}$, $g\not\equiv 0$, $g(0)=0$, and assume $t \mapsto (t^m,g(t))$ 
is injective into $\C^2\setminus \UHP^2$.  
Then, $\phi(t) := -g(t)$ is of one of the two following forms
\begin{itemize}

\item Pure stable type:
\[
\phi(t) = q(t^m) +t^{2mL} \psi(t)
\]
where 
\begin{itemize}[label=$\diamond$]
\item $L$ is a positive integer,
\item $q \in \R[t],$ where $q(0)=0,$ $q'(0)>0$, $\deg q < 2L$,
\item $\psi\in \C\{t\}$ with $\Im \psi(0) > 0$.
\end{itemize}
If $m>1$ then $\psi$ is not of the form $\psi(t) = h(t^m)$ for $h\in \C\{t\}$.
\item Real stable type: $m=1$, $\phi \in \R\{t\}$ (analytic with real coefficients) and $\phi(0)=0, \phi'(0)>0$.
\end{itemize}
\end{theorem}

We shall also say the associated Weierstrass polynomial
is of \emph{pure stable type} or \emph{real stable type} depending
on the type of the underlying function $g$ in its Puiseux expansion.
Using the above description, a Weierstrass polynomial of pure stable type
is of the form
\begin{equation} \label{WPpure}
\prod_{n=1}^{m} (z_2 + q(z_1) + z_1^{2L} \psi(\mu^n z_1^{1/m}))
\end{equation}
where $\mu = \exp(2\pi i/m)$.   

An interesting paper with some antecedents of this theorem is \cite{Tsi93} (see Sections 3-4).

%{\color{green}
\begin{proof}[Discussion of the proof of Theorem \ref{thm:localtypes}]
This theorem is actually a correction, a refinement, and a 
rephrasing
of Lemma C.3 from \cite{Kne15}.
Said Lemma C.3 states that given $\phi$ analytic 
in a neighborhood of $0 \in \C$, $\phi(0)=0$, and
$t\mapsto (t^k, \phi(t))$ injective into $\C^2 \setminus \UHP^2$,
then
$\phi$ has the following form
\[
\phi(t) = \sum_{j=1}^{M} a_j t^{jk} + b t^{2kL} + \sum_{j=2kL+1}^{\infty} b_j t^j 
\]
where $a_1<0, a_2,\dots, a_M \in \R$, and $\arg b \in (\pi ,2\pi)$.
A minor correction to this statement is
that there is also the case that $\phi$ is 
of the form
\begin{equation} \label{realphi}
\phi(t) = \sum_{j=1}^{\infty} a_j t^j
\end{equation}
where $k=1$, $a_1<0$ and $a_j \in \R$ for $j=2,3,\dots$.
This does not affect any of the results in \cite{Kne15}
because the real coefficient case was not relevant there.
In any case, the proof given in \cite{Kne15} is essentially correct
up to the point where it proves $\phi$ is necessarily of 
the form $\phi(t) = a t^k + \text{ higher order }$ with $a<0$.
However at this point there are two cases; 
either $\phi$ is of the form
\[
\phi(t) = \sum_{j=1} a_{j} t^{jk}
\]
with $a_j \in \R$ for all $j$
or
\[
\phi(t) = \sum_{j=1}^{M} a_j t^{jk} + b t^L + \text{ higher order }
\]
where either $b \notin \R$ or $L$ is not a multiple of $k$
(we allow for $M=1$).
For this second case the proof proceeds as written in \cite{Kne15}.
For the first case, we must have $k=1$ because
we assume $t\mapsto (t^k, \phi(t))$ is injective.

Theorem \ref{thm:localtypes} can now be deduced
immediately from our corrected Lemma C.3 (with $g(t)$ in
the statement of Theorem \ref{thm:localtypes} playing the role of 
$\phi$ in the statement of Lemma C.3 in \cite{Kne15}).
\end{proof}
%}

The two types in Theorem \ref{thm:localtypes} are directly related to our global 
dichotomy of stable polynomials from Section \ref{sub:global}.
If $p$ is pure stable then locally it can only have 
irreducible Weierstrass factors of pure stable type---a 
real stable factor would imply infinitely many zeros
on $\R^2$.
If $p$ is real stable then locally it can only have
irreducible Weierstrass factors of real stable type
if we ignore monomial factors.
Indeed, if we had a pure stable factor parametrized via $\phi$ as above 
then 
for $x>0$, $p(x, -\phi(x^{1/m})) \equiv 0$.
But 
\[
\Im(\phi(x^{1/m})) = x^{2L}(\Im(\psi(0)) + O(x^{1/m}))
\]
is positive for small $x>0$.  So, $p$
has roots in $\R\times (-\UHP)$ which we could
perturb to get roots in $(-\UHP)^2$.

The above theorem also shows that 
real stable polynomials locally
factor into smooth branches near a zero on $\R^2$ 
since real stable
factors are degree one Weierstrass polynomials.
This was a main result of \cite{bps19a}, but it now follows
directly from Theorem \ref{thm:localtypes}.
The precise statement follows.

\begin{theorem}[Local parametrization of real stable polynomials \cite{bps19a}] \label{thm:localpar}
Let $p\in \R[z_1,z_2]$ have no zeros in $\UHP^2$ and $p(0,0)=0$. Assume that $p(0,\cdot), p(\cdot,0) \not\equiv 0$.  Then, there
exist $r,R>0$ and
$\phi_1,\dots, \phi_N\in \R\{z_1\}$ convergent on $|z_1|\leq r$ with $\phi_j(0)=0, \phi_j'(0) > 0$ 
such that 
$\mathcal{Z}_p\cap \R^2$
is described by
\[
\bigcup_{j=1}^{N} \{(x_1,x_2): x_2+\phi_j(x_1) = 0\}
\]
for $(x_1,x_2) \in (-r,r)\times(-R,R)$.
\end{theorem}

The conditions $p(0,\cdot), p(\cdot,0)\not\equiv 0$
simply rule out monomial factors which can only 
contribute zero sets $z_1=0$ or $z_2=0$.

\begin{example} \label{ex71puiseux}
Let us again return to Example \ref{ex71}.
Now, $p(0,z_2) = 4z_2^2$ so we can write
\[
p(z_1,z_2) = \underset{=: u}{\underbrace{-4(1-4iz_1-2z_1^2)}}\left(z_2^2 + \frac{4z_1(1-iz_1)}{1-4iz_1-2z_1^2} z_2+ \frac{z_1^2}{1-4iz_1-2z_1^2}\right).
\]
Near $(0,0)$ we can factor $p/u$ explicitly into two degree one factors of pure stable type
$(z_2+f_1(z) )(z_2+ f_2(z))$
where
\[
f_1(z) := z_1\frac{(2(1-iz_1)+\sqrt{3-4iz_1-2z_1^2})}{1-4iz_1-2z_1^2} =
(2+\sqrt{3})z_1 + i \left(6+\frac{10}{\sqrt{3}}\right)z_1^2+\cdots
\]
\[
f_2(z) := z_1\frac{(2(1-iz_1)-\sqrt{3-4iz_1-2z_1^2})}{1-4iz_1-2z_1^2}
=
(2-\sqrt{3})z_1 + i \left(6-\frac{10}{\sqrt{3}}\right)z_1^2+\cdots
\]

On the other hand, $A = 4(z_2^2+4z_1z_2+z_1^2-2z_1^2z_2^2)$ factors
\[
\begin{aligned}
A &= 4(1-2z_1^2)\left( z_2^2 + \frac{4z_1}{1-2z_1^2} z_2 + \frac{z_1^2}{1-2z_1^2}\right) \\
&= 4(1-2z_1^2)\left( z_2+ z_1\frac{2+\sqrt{3+2z_1^2}}{1-2z_1^2}\right)\left( z_2+ z_1\frac{2-\sqrt{3+2z_1^2}}{1-2z_1^2}\right)
\end{aligned}
\]
into two degree one factors of real stable type. \eox
\end{example}

\begin{example}\label{ex153puiseux}
Consider again Example \ref{ex153homog}
\[
P(z_1,z_2) = z_1+z_2-2z_1^3-6z_1^2 z_2 - i(z_1^2+z_1z_2-4z_1^3z_2)
\]
which although it only vanishes to order 1 at $(0,0)$, its pure stable Puiseux factorization
still has some complexity to it.
Observe
\[
\begin{aligned}
P(z_1,z_2) &= \underset{=:u(z)}{\underbrace{(1-iz_1-6z_1^2+4iz_1^3)}}\left( z_2 + z_1 \frac{1-iz_1-2z_1^2}{1-iz_1-6z_1^2+4iz_1^3}\right) \\
&= u (z)(z_2+ z_1+4z_1^3+24 z_1^5 + 8i z_1^6 + O(z_1^7)).
\end{aligned}
\]
Similarly, $A(z_1,z_2) = z_1+z_2-2z_1^3-6z_1^2 z_2$ has the (trivial) Puiseux factorization
\[
\begin{aligned}
A(z_1,z_2) &= (1-6z_1^2)\left(z_2 + z_1 \left(1+ \frac{4z_1^2}{1-6z_1^2}\right) \right) \\
&= (1-6z_1^2)(z_2+z_1 + 4z_1^3+24z_1^5 + O(z_1^7)).
\end{aligned}
\]
\eox
\end{example}

We have the following local parametrization theorem for
pure stable polynomials, a direct result of Theorem \ref{thm:localtypes} and Puiseux's theorem.

\begin{theorem} \label{thm:purelp}
Let $p \in \C[z_1,z_2]$ be pure stable
and vanish to order $M$ at $(0,0)$.  
Factor $p = u p_1\dots p_k$ into a unit $u\in \C\{z_1,z_2\}$
and irreducible Weierstrass polynomials of pure stable type.
We may write each $p_j$ 
\[
p_j(z) = \prod_{m=1}^{M_j} (z_2 + q_j(z_1) + z_1^{2L_j} \psi_j(\mu_j^m z_1^{1/M_j}))
\]
with all of the data associated to a Puiseux expansion of pure stable type from Theorem \ref{thm:localtypes}.
\end{theorem}

The following corollary describes the possible lowest order
homogeneous terms of $p$.

\begin{corollary}\label{cor:purelp}
Assume the setup and conclusion of the previous theorem.
The lowest order
homogeneous term of $p$ must be of the form
\[
P_M = c \prod_{j=1}^{k} (z_2+a_j z_1)^{M_j}.
\]
for $a_j>0$ and $c\in \C$.
In particular, if we write $p=A+i B$ and multiply $p$ by a constant to force
$P_M=A_M$, then $A_M$ has no factors of $z_1$ or $z_2$
and $B_{M+1}$, the lowest order homogeneous term of $B$ by Proposition \ref{prop:B1},
has at most one factor of $z_1$ and at most one factor of $z_2$.
\end{corollary}

\begin{proof}
Multiplying out $p_j$ in Theorem \ref{thm:purelp} we get
\[
p_j(z) = (z_2+q_j'(0)z_1)^{M_j} + \text{ higher order terms}
\]
and multiplying $p_1,\dots, p_k$ together we get the desired form
for $P_M$ since the unit $u$ only affects this up to a constant multiple.

The observation about $A_M$ is evident while the claim about $B_{M+1}$ 
follows from Proposition \ref{prop:tangent} (set $r=0$ in that proposition
to see that $B_{M+1}$ has either $0$ or $1$ factors of $z_1$).
\end{proof}

In particular, if $k=1$ (i.e. we have a single irreducible Weierstrass factor)
then the corresponding algebraic curve has a single tangent line through $(0,0)$;
i.e. the lowest order homogeneous term of $p$ is a power of a linear
polynomial $(az_1+bz_2)^M$.  This is more
generally true. See \cite{Abh} Chapter 18, where this fact
(stated in slightly different language) is called the \emph{Tangent Lemma}.
Stated contra-positively, if $p$ has more than one tangent at $(0,0)$
then it is locally reducible---i.e. has more than one irreducible Weierstrass factor.
If $p$ has all distinct tangents at $(0,0)$
(i.e. an ordinary multiple point)
 then it is a product of degree one Weierstrass polynomials.

\subsubsection{Stable polynomials with non-trivial Puiseux expansions}

To our knowledge, the literature contains no examples of stable polynomials 
with non-trivial Puiseux expansions.  It turns out that there is a way
to build such examples from our local description.  
We present an example and a theorem generalizing the example.

\begin{example}
A simple example of a non-trivial pure stable branch is
\[
h(z) = z+ i z^2 + c z^{5/2}
\]
where $c$ is a constant.  We will see that $|c|\leq \sqrt{2}$ is
a convenient assumption.
Concretely, for $z=x+iy$ with $0<|z|< \tfrac{1}{2}$ and $y\geq 0$
\[
\begin{aligned}
\Im h(z) \geq y + x^2-y^2 - |c||z|^{5/2} & \geq y(1-|z|) + x^2 - |c| |z|^{5/2}\\
&\geq  |z|^2(1-\sqrt{2} |z|^{1/2}) > 0.
\end{aligned}
\]
 We can then build the polynomial
\[
p_1(z_1,z_2) = (z_2+z_1+ i z_1^2 + c z_1^{5/2})(z_2+z_1+ i z_1^2 - c z_1^{5/2}) 
= (z_2+z_1+iz_1^2)^2 - c^2 z_1^5
\]
which has no zeros when $0<|z_1|<1/2$, $\Im z_1 \geq 0$, and $\Im z_2 \geq 0$.
In particular, $p_1$ is non-vanishing on $\{z_1: |z_1-i/4| < 1/4\}\times \UHP$.
The map $\sigma(z_1) = \frac{z_1}{1- 2 iz_1}$ sends $\UHP$ onto $\{z_1: |z_1-i/4|<1/4\}$
so that
\[
\begin{aligned}
p_2(z_1,z_2) &= (1-2 iz_1)^5p_1(\sigma(z_1),z_2) \\
&= (1-2iz_1)((1-2iz_1)^2 z_2 + z_1(1-2i z_1) + i z_1^2)^2 - c^2 z_1^5
\end{aligned}
\]
has no zeros in $\UHP^2$ for $|c| \leq \sqrt{2}$.  \eox
\end{example}

The next theorem gives a
lower bound on the imaginary part of Puiseux branches of pure stable type
which lets us generalize this example.

\begin{theorem} \label{PuiseuxLowerBound}
Let 
\[
h(z) = q(z) + z^{2L} \psi( z^{1/m})
\]
be a branch of pure stable type as in Theorem \ref{thm:localtypes}.  
Namely, $q\in \R[z]$ with $\deg q<2L$, $q(0)=0$, $a:=q'(0)>0$ and $\psi \in \C\{t\}$
with $b:=\Im \psi(0)>0$.  
We assume $z^{1/m}$ has a branch cut in the lower half plane
so that $h$ is analytic on a domain containing $\{z: 0<|z|<\epsilon \text{ and } \Im z \geq 0\}$.
Then, there exist $r,c>0$ such that for $0<|z|<r$ and $\Im z \geq 0$ we have
\[
\Im h(z) \geq c |z|^{2L}.
\]
In particular, $h(z) \ne 0$ for $0<|z|<r$ and $\Im z \geq 0$.
\end{theorem}

\begin{proof}
Write $z=x+iy$.  Note that $\Im (z^j) = y O(|z|^{j-1})$ for $j \geq 2$
since $z^j - \bar{z}^j = (z-\bar{z})\sum_{k=0}^{j-1} z^k\bar{z}^{j-1-k}$.  
This implies
\[
\Im q(z) = a y(1 - O(|z|))
\]
for $|z|$ sufficiently small
since $q$ has real coefficients.
Note next that $\text{Re}(z^{2L}) = x^{2L} + y O(|z|^{2L-1})$
by directly expanding $(x+iy)^{2L}$.  
Then, 
\[
\begin{aligned}
\Im (z^{2L} \psi(z^{1/m}))  &\geq (\Im z^{2L}) \text{Re}(\psi(0)) + b \text{Re}(z^{2L}) -O(|z|^{2L+1/m})\\
&\geq b x^{2L} - y O(|z|) - O(|z|^{2L+1/m}).
\end{aligned}
\]
Combining we get
\[
\Im h(z) \geq ay(1-O(|z|)) + b x^{2L} - O(|z|^{2L+1/m})
\geq c|z|^{2L}
\]
for some $c>0$ when $|z|$ is small enough and $y\geq 0$.
\end{proof}

If we choose $\psi$ above to be a polynomial then
the resulting product 
\[
p_1(z_1,z_2) = \prod_{j=1}^{m} (z_2+ q(z_1) + z_1^{2L} \psi(\mu^j z_1^{1/m}))
\]
with $\mu = \exp(2\pi i/m)$ 
will be a polynomial with no zeros in $(\D_{2r}\cap \UHP)\times \UHP$
for $r$ sufficiently small. 
(One can even choose $\psi$ to be rational and clear denominators.)
Note that $p_1$ has no zeros on the domain $\D_{r}(i r)\times \UHP$.
The map $z \mapsto \sigma(z) = \frac{2r z}{1-i z}$ sends $\UHP$ onto $\D_{r}(ir)=\{z: |z-ir|<r\}$.
Then,
\[
p_2(z_1,z_2) = (1-iz_1)^N p_1(\sigma(z_1), z_2)
\]
is a polynomial for $N$ large enough and has no zeros in $\UHP^2$.  
Note that since $\sigma(0)=0$, $p_2$ has a non-trivial
Puiseux factorization around $(0,0)$.

\subsubsection{Initial Segments and Contact Order} \label{subsec:seg} 

Given a pure stable polynomial $p\in \C[z_1,z_2]$ with $p(0,0)=0$
with real and imaginary coefficient decomposition $p = A+ i B$
recall that $A+tB$ is real stable for $t\in \R$ and $B$ is also real stable (Remark \ref{AtBrs}). 
How do the Puiseux expansions of $p$ and $A+tB$ compare?

%{\color{green}
Theorem \ref{thm:segments} below, which is proven in the next subsection, 
says that generically in $t$ there is an exact correspondence
between the smooth branches of $A+tB$ and the Puiseux branches of $p$.
Basically, $p$ factors locally as
\[
u(z) \prod_{j=1}^{k} \prod_{m=1}^{M_j} (z_2 + q_j(z_1) + z_1^{2L_j}\psi_{j}(z_1^{1/M_j}))
\]
and we prove that $A+tB$ generically in $t$ factors \emph{analytically}
in the same general way:
\[
A(z)+tB(z) = u(z;t) \prod_{j=1}^{k} \prod_{m=1}^{M_j} (z_2 + q_j(z_1) + z_1^{2L_j}\psi_{j,m}(z_1;t)).
\]
Furthermore, we can show that $\psi_{j,m}(0;t)$ must depend
on $t$, so that this generic expansion is the best possible
in the sense that we cannot extend one of these branches
further in a way that is independent of $t$.

In order to state this more precisely,
let us say that a real polynomial $q \in \R[z_1]$ is an order $N$ initial
segment of a branch of a two variable polynomial $p$
if the branch is of the form $z_2 + \phi(z_1)$ 
where $\phi(z_1) -q(z_1)$ vanishes to order $N$ or higher. 
We allow for $\phi$ to be a Puiseux series or a bona fide 
analytic function.

Now, Theorem \ref{thm:purelp} says $p$ has a complete list of initial
segments $q_j$ occurring $M_j$ times with order $2L_j$.
Theorem \ref{thm:segments} below says the same holds
generically for $A+tB$ and furthermore
none of these generic segments can be extended independently of $t$.
Specifically, we \emph{cannot} write down a ``better'' 
factorization
\[
A(z)+tB(z) = u(z;t) \prod_{j=1}^{k} \prod_{m=1}^{M_j} (z_2 + \tilde{q}_j(z_1) + z_1^{2\tilde{L}_j}\tilde{\psi}_{j,m}(z_1;t)).
\]
with $\tilde{L}_j \geq L_j$ for all $j$ and $\tilde{L}_j > L_j$ for some $j$.
Here is the formal statement.

\begin{theorem} \label{thm:segments}
Assume the setup, conclusion, and notations of Theorem \ref{thm:purelp}.
Set as usual $A = (p+\bar{p})/2$, $B = (p-\bar{p})/(2i)$.
Then, for all but finitely many $t \in \R$ we can factor
\[
A(z)+tB(z) = u(z;t) \prod_{j=1}^{k} \prod_{m=1}^{M_j} (z_2 + q_j(z_1) + z_1^{2L_j}\psi_{j,m}(z_1;t))
\]
where $\psi_{j,m}(z_1;t) \in \R\{z_1\}$, $u(z;t) \in \C\{z_1,z_2\}$ for each $t$ and $u(0;t) \ne 0$.

We can do no better in the sense that we \textbf{cannot} find:
 a pair $(j,m)$, a polynomial $\tilde{q}_j$ of degree less than $2\tilde{L}_j > 2L_j$ and
an analytic function $\tilde{\psi}_{j,m}(z_1;t)$ such that the factor above corresponding to $(j,m)$ can generically be replaced with
\[
z_2 + \tilde{q}_j(z_1) + z_1^{2\tilde{L}_j}\tilde{\psi}_{j,m}(z_1;t).
\]
where $\tilde{q}_j-q_j$ vanishes to order at least $2L_j$.
\end{theorem}
We prove the above theorem later in this section.

%Thus, what we call the initial segments of branches of $A+tB$ are generically prescribed
%by $p$ and these cannot be generically \emph{better} prescribed.  

Of particular geometric interest in the above data 
are the highest order initial segments.
Set $K = \max\{2L_1,\dots, 2L_k\}$ and for
clarity reorder so $K=2L_1\geq 2L_2\geq \cdots \geq 2L_k$.
The quantity $K$ has a couple of interpretations.
%}

First, $K$ measures asymptotically how closely the
zero set of $p$ or $\bar{p}$ approaches the distinguished
boundary $\R^2$.  
Indeed, for any branch of $p$, say the branch 
\begin{equation} \label{pjbranch}
z_2 + q_j(z_1) + z_1^{2L_j} \psi_j(\mu_j^r z_1^{1/M_j})=0,
\end{equation}
if we take $z_1=x\in \R$ then on this branch
\begin{equation} \label{branchco}
\Im z_2 = -x^{2L_j} (\Im \psi_j(0) + O(x^{1/M_j}))
\end{equation}
vanishes to the precise order $2L_j$.  
This implies that for $x$ close to $0$
\begin{equation} \label{uhpco}
\inf\{ |\Im z_2|: p(x,z_2) = 0\} \approx |x|^K.
\end{equation}
Here ``$\approx$'' means bounded above and below
by constants.
The number $K$ satisfying \eqref{uhpco} is called
the \emph{contact order} of $p$ at $(0,0)$.
The notion of contact order is conformally
invariant in the sense that if we use a Cayley
transform from $\UHP^2$ to $\D^2$ sending $(0,0)$ to $(1,1)$
and convert $p$ to an atoral
stable polynomial $q$
then for $z_1 \in \T$ close to $1$
\begin{equation} \label{eqn:CO}
\inf\{ 1-|z_2|: \tilde{q}(z_1,z_2) = 0 \} \approx |1-z_1|^K.
\end{equation}
A main result of \cite{bps18} connects
contact orders to integrability of derivatives
of rational inner functions.

A second interpretation of $K$ is that 
for any $t_1,t_2 \in \R$ outside of some
finite set $S$ of exceptional points,
$A+t_1 B$ and $A+t_2 B$ have branches
say $z_2+ \phi(z_1;t_1)$ and $z_2+\phi(z_1;t_2)$
such that $\phi(z_1;t_1)-\phi(z_1;t_2)$ 
vanishes to order $K$.  Furthermore, $K$
is the largest integer with this property.
This interpretation of $K$ was referred to as 
the \emph{order of contact} of $p$.  
One can convert this to a statement
about unimodular level sets of rational inner functions
on the bidisk via Cayley transform.
These two interpretations of $K$, contact order and order of contact, 
 were how this material
was originally approached in \cite{bps18, bps19a}
and the following fundamental result of \cite{bps19a}
can be established from Theorem \ref{thm:segments} by the above discussion.

\begin{theorem}[\cite{bps19a} Theorem 3.1] \label{thm:OCCO}
Order of contact equals contact order.
\end{theorem}

Theorem \ref{thm:segments} also constitutes a 
resolution to Conjecture 5.2 of \cite{bps19a} about
the concepts of \emph{fine contact order} and \emph{fine order
of contact}.   Fine contact order refers to 
the contact order of an individual branch of the zero 
set of $p$ as in equation \eqref{branchco} 
where it is shown that an individual branch \eqref{pjbranch} has
contact order $2L_j$.
The fine contact orders associated to $p$ can be read off 
from Theorem \ref{thm:purelp}: we have fine contact 
order $2L_j$ occuring $M_j$ times for $j=1,\dots, k$.
Fine order of contact refers to the order of vanishing
of the difference of two different (analytic) branches of $A+tB$.
Namely, if $z_2+ \phi(z_1,t_1)$ is a branch
of $A+t_1B$ and $z_2+\phi(z_1;t_2)$ is a branch of
$A+t_2 B$ then the fine order of contact
of these two branches is the order of vanishing of $\phi(z_1;t_1)-\phi(z_1;t_2)$.
Conjecture 5.2 of \cite{bps19a} stated that
one could generically (with respect to $t$) group the branches
of $A+t_1 B$ and $A+t_2 B$ so that the fine order of contacts
exactly match the fine contact orders of $p$.
Theorem \ref{thm:segments} does exactly this.
The next subsection is occupied with its proof.

\subsubsection{Proof of Theorem \ref{thm:segments}}

For convenience we will use $(x,y)$ in place
of the variables $(z_1,z_2)$.
Assume the setup of Theorems \ref{thm:purelp} and \ref{thm:segments}.
Consider 
\begin{equation} \label{blowup}
p(x, x^n y - q(x)) 
\end{equation}
where $n \geq 1$, $q \in \R[x]$ and $q(0)=0$. 
Let $O(n,q; p)$ be the largest integer $N$ such that $x^N$ divides \eqref{blowup}.
Intuitively, $O(n,q;p)$ helps us measure
the presence of $q(x)$ as an initial segment
of branches of $p$.  It is difficult to tease
out exactly what $O(n,q;p)$ is measuring and
the key idea that follows is that it is more
fruitful to see how this quantity \emph{changes} 
with respect to $n$.

Note that $O(n,q; \bar{p}) = O(n,q;p)$ because $q$ has real coefficients.
Since $A+tB$ is a linear combination of $p$ and $\bar{p}$,
\begin{equation} \label{basicO}
\begin{aligned}
O(n,q; A+tB) &\geq O(n,q;p) \text{ and } \\
O(n,q;A+tB) &= O(n,q;p) \text{ for all $t$ with at most one exception.}
\end{aligned}
\end{equation}
Indeed, if we had $O(n,q;A+tB)> O(n,q;p)$ for $t=s_1,s_2$ distinct then
since $p$ is a linear combination of $A+s_1 B, A+s_2 B$ we would have
\eqref{blowup} divisible by a higher power of $x$.

As in the previous section, we say that
a branch $y+q_j(x) + x^{2L_j} \psi(\mu_j^r x^{1/M_j})$ has initial
segment $q\in \R[x]$ of order $N$
if 
\[
q_j(x) -q(x) + x^{2L_j} \psi(\mu_j^r x^{1/M_j})
\]
vanishes to order $N$ or higher. 

\begin{lemma} \label{lem:segment}
Given $q \in \R[x]$ with $q(0)=0$,
the quantity 
\begin{equation} \label{branchcounter}
O(n+1, q; p) - O(n, q; p)
\end{equation}
counts the number of branches of $p$ such that $q$ is an initial segment of order $n+1$.
Similarly, 
\[
O(n+1,q;A+tB)- O(n,q;A+tB)
\]
counts the number of branches of $A+tB$ such that $q$ is an initial segment of order $n+1$.
\end{lemma}

\begin{proof}
Let us look at a single branch $y + q_j(x) + x^{2L_j}\psi_j(\mu_j^r x^{1/M_j})$ of $p$.
Now, if $q$ is an initial segment of order $n+1$ then
\begin{equation} \label{branchj}
q_j(x)- q(x) + x^{2L_j}\psi_j(\mu_j^r x^{1/M_j})
\end{equation}
vanishes to order $n+1$ or higher.  This is only
possible if $n+1 \leq 2L_j$ since 
$\psi_j(0)$ has non-zero imaginary part.
Then, 
\begin{equation} \label{pk1}
x^{n+1} y - q(x) + q_j(x) + x^{2L_j}\psi_j(\mu_j^r x^{1/M_j})
\end{equation}
has largest $x$-factor $x^{n+1}$ while
\begin{equation} \label{pk}
 x^{n}y - q(x) + q_j(x) + x^{2L_j}\psi_j(\mu_j^r x^{1/M_j}).
\end{equation}
has largest $x$-factor $x^n$.
So, the contribution to \eqref{branchcounter} is $1$.

On the other hand, if \eqref{branchj} vanishes to
order $s<n+1$ then
both \eqref{pk1} and \eqref{pk} are divisible by $x^s$ but no higher power of $x$
making the contribution to \eqref{branchcounter} equal to $0$.

The argument for $A+tB$ is easier because its branches 
are analytic and have real coefficients.  In particular,
there is no ``$\psi$'' term to worry about.
\end{proof}

As a basic illustration of the lemma we can prove
a more precise version of Theorem \ref{thm:OCCO}.
Suppose as before $K=2L_1 \geq 2L_2 \geq \cdots \geq 2L_k$.
Then, $O(K, q_1;p) > O(K-1, q_1;p)$ while 
\begin{equation} \label{orderstabilize}
O(n+1,q;p) = O(n, q;p)
\text{ for any } q\in \R[x] \text{ and } n\geq K.
\end{equation}
Now for all but at most one $t$, say $t\ne t_0$,
$O(K-1,q_1;A+t B) = O(K-1, q_1; p)$
so that for $t\ne t_0$
\[
O(K, q_1; A+tB) - O(K-1, q_1; A+tB) >0
\]
which implies that for  $t\ne t_0$, $A+tB$
has a branch with initial segment $q_1$ of order $K$.
This implies the order of contact is at least $K$.
On the other hand, suppose 
there exist two values of $t$, say $t_1\ne t_2$,
and a real polynomial $q$ that is an initial segment
of order $\tilde{K}> K$ of $A+t_1 B$ and $A+t_2 B$.
Then, by \eqref{basicO} we can say without loss of generality that $O(\tilde{K}, q; A+t_1 B) = O(\tilde{K}, q; p)$ 
and then
\[
O(\tilde{K}, q; p) =O(\tilde{K}, q; A+t_1 B)  > O(\tilde{K}-1,q;A+t_1 B) \geq O(\tilde{K}-1,q;p)
\]
contradicts \eqref{orderstabilize}.

\begin{proof}[Conclusion of the proof of Theorem \ref{thm:segments}]
With Lemma \ref{lem:segment} in hand the main
issue now is a combinatorial one.
We can count occurrences of initial segments
in branches of $p$ and these must agree generically
with $A+tB$ but the maximal initial segments
of $p$ can overlap in a variety of ways
so we must perform a type of inclusion-exclusion analysis.
Let us write the data from Theorem \ref{thm:purelp}
as triples $(q_j(x), L_j, M_j)$.
If we ever have $q_i=q_j$ and $L_i=L_j$ let us
regroup our triple into $(q_i(x), L_i, M_i+M_j)$.
Do this as many times as necessary so that we have a list
\[
(Q_1(x), K_1, N_1), \dots (Q_r(x), K_r, N_r)
\]
with the pairs $(Q_j, K_j)$ all distinct, $K_j$'s non-increasing,
and $\sum_j N_j = M$, where $M$ is the order of vanishing of $p$ at $(0,0)$.

With all of this preparation, 
\[
b_j := O(2K_j, Q_j; p) - O(2K_j-1, Q_j; p)
\]
equals the number of branches of $p$ with
initial segment $Q_j$ of order $2K_j$.  Then, for all but $2$ values of $t$,
 $A+t B$ has
$b_j$ branches with initial segment $Q_j$ of order $2K_j$.  
Since $K_1$ is maximal, $b_1 = N_1$.  
If $Q_2$ is an initial segment of order $2K_2$ of $Q_1$ 
then $b_2 = N_1 +N_2$
otherwise $b_2=N_2$.  
In either case we deduce that $p$ and
hence generically $A+tB$
has $N_1$ branches with initial segment $Q_1$
and $N_2$ different branches with initial segment $Q_2$.
In general,
\[
b_j = \sum_{i \in S_j} N_i \text{ where }  S_j=\{ i\leq j: Q_j \text{ is an initial segment of } Q_i \text{ of order } 2K_j\}.
\]
%{\color{green}
Since $j \in S_j$, 
\[
b_j = N_j + \sum_{i<j, i \in S_j} N_i
\]
and this shows $N_j$ can be
computed in terms of $\{N_1,\dots,N_{j-1}\}\cup\{b_j\}$.
By induction then, $N_j$ can be computed entirely
in terms of $\{b_1,\dots, b_j\}$
with knowledge of the sets $S_j$ above.
%}
%entirely using ``segment'' relations between the $Q_i$'s.
Now, $N_j$ equals the number of branches
of $p$ 
that have initial segment $Q_j$ but cannot extend to a 
longer initial segment (the $\psi$ terms in the 
Puiseux expansion for $p$ block this). 
Since we can compute $N_j$ using $\{b_1,\dots, b_j\}$
we see that this statement holds generically for $A+tB$.
Thus, generically $A+tB$ has $N_j$ branches 
with initial segment $Q_j$ that do not extend to 
longer initial segments.  
This proves Theorem \ref{thm:segments}.
\end{proof}

Notice that each use of $b_j$
requires us to potentially avoid $2$ values
of $t$ in $A+tB$.  Thus, $A+tB$
has the structure in Theorem \ref{thm:segments}
for all but at most $2r$ values of $t$,
where $r$ (see above) is at most $k$,
the number of irreducible Weierstrass factors
of $p$ at $(0,0)$.  
A simpler yet cruder statement
would be to say that $A+tB$ has the desired
factorization for all but $2M$ values of $t$
since $M$, the order of vanishing at $(0,0)$,
can be read off fairly easily from $p$
and equals the total number
of branches.

\subsubsection{Switching variables}

Thus far the two variables $z_1,z_2$ have played distinct roles.
In this section, we prove that 
when we switch variables in Theorem \ref{thm:purelp},
 the cutoff data $2L_1,\dots, 2L_k$ is preserved and
 the initial segments in $z_2$ can be computed from those
 in $z_1$.
 In \cite{bps19a} (Theorem 4.1) it was already shown that contact order does not
depend on whether we examine contact order with respect to $z_1$ or $z_2$.

In order to state our theorem on switching variables
we need to introduce some notation.
Let $q \in \C[z_1]$ be a polynomial with $q(0) = 0, q'(0) \ne 0$
and let $L \geq 1$.
We define $I_{2L}(q) \in \C[z_2]$ to be the polynomial of degree less than
$2L$ such that
\[
I_{2L}(q)(z_2) - q^{-1}(z_2) 
\]
vanishes to order at least $2L$.
Here $q^{-1}$ is the analytic functional inverse
of $q$ guaranteed by the inverse function theorem and $I_{2L}(q)$ is just 
the power series of $q^{-1}$ cut off past degree $2L-1$.

\begin{theorem} \label{thm:switch}
Assume the setup and conclusion of Theorem \ref{thm:purelp}.  
Then,
\[
p(z) = v(z) \prod_{j=1}^{k} \prod_{m=1}^{M_j} (z_1 - I_{2L_j}(q_j)(-z_2) + z_2^{2L_j} \tilde{\psi}_j(\mu_j^m z_2^{1/M_j}))
\]
where $v(z_1,z_2) \in \C\{z_1,z_2\}$ is a unit, $\tilde{\psi}_j \in \C\{z_2\}$, and $\mu_j = \exp(2\pi i/M_j)$.
\end{theorem}

Thus, the cutoffs are all preserved and the initial segments get transformed
to $-I_{2L_j}(q_j)(-z_2)$.  

\begin{proof}
As in Theorem \ref{thm:purelp}, let $p\in \C[z_1,z_2]$ be a pure stable polynomial.
To begin, we explain why
the number and degrees of the irreducible Weierstrass polynomials of $p$ match when we switch variables. 
An irreducible Weierstrass polynomial in $z_2$
of pure stable type is of the form \eqref{WPpure}
\[
g(z)= \prod_{m=1}^{M} (z_2 + q_1(z_1) + z_1^{2L_1} \psi_1(\mu^m z_1^{1/M}) ).
\]
Since $q_1(0)=0,q_1'(0)>0$ we see that $g(z_1,0)$ vanishes to order $M$
and therefore
this can be factored into a unit $v\in \C\{z_1,z_2\}$ times a Weierstrass polynomial in $z_1$
\[
g(z) = v(z)( z_1^M + a_1(z_2) z_1^{M-1} + \cdots + a_M(z_2)).
\]
This Weierstrass polynomial in $z_1$ must be irreducible.
If it factored the resulting irreducible Weierstrass
polynomial factors of pure stable type 
would be equal to a unit times Weierstrass polynomials
in $z_2$ by the same reasoning; this would contradict irreducibility
of $g$ with respect to $z_2$.
This proves that the irreducible Weierstrass polynomial factors in $z_2$
of $p$ are unit multiples of the irreducible Weierstrass polynomial factors
in $z_1$ of $p$.  

Again using $g$ for an irreducible Weierstrass polynomial of pure stable 
type we also have a pure stable type Puiseux factorization with respect to $z_1$
\[
g(z) = v(z) \prod_{m=1}^{M}(z_1 + q_2(z_2) + z_2^{2L_2}\psi_2(\mu^m z_2^{1/M}))
\]
and we would like to know $L_1=L_2$ and we would like to compute $q_2$ from
$q_1$.  We emphasize that $q_2$ is not the same as in
the statement of Theorem \ref{thm:switch}---since we are isolating a 
single irreducible Weierstrass polynomial we think it is safe to have $q_1$ and $q_2$
play new roles strictly during this proof.

Let 
\[
\phi_1(t) = -(q_1(t^M) + t^{2L_1M} \psi_1(t))
\text{ and } \phi_2(t) = -(q_2(t^M) + t^{2L_2 M} \psi_2(t))
\]
and note 
\[
\phi_1(t) = t^M( -q_1'(0) + \dots) \text{ and } \phi_2(t) = t^M(-q_2'(0)+ \dots)
\]
because $q_1(0)=q_2(0)=0$ and $q_1'(0),q_2'(0)\ne 0$.
By these representations $\phi_1, \phi_2$ both have
analytic $M$-th roots near $0$
\[
\phi_3(t) = \phi_1(t)^{1/M} = t((-q_1'(0))^{1/M} + \dots) 
\]
and
\[
 \phi_4(t) = \phi_2(t)^{1/M} = t((-q_2'(0))^{1/M} + \dots). 
\]
These $M$-th roots are determined by a choice of
$M$-th root of $-q_1'(0), -q_2'(0)$.
Notice $\phi_3'(0) = (-q_1'(0))^{1/M}, \phi_4'(0) = (-q_2'(0))^{1/M}$
are both nonzero so $\phi_3, \phi_4$ have
local analytic functional inverses near $0$.

Locally $g$'s zero set can be parametrized via the injective maps
\[
t \mapsto (t^M, (\phi_3(t))^{M}) \text { or } t \mapsto ((\phi_4(t))^{M}, t^M).
\]
Then, $(\phi_3(\phi_4(t)))^M = t^M$
which implies that  $\phi_3(\phi_4(t))$ is an $M$-th root of unity
times $t$.
Namely, $\phi_3(t)$ and $\phi_4(t)$ are functional
inverses up to a multiple of an $M$-th root of unity.
Reverting back to $\phi_1$ we have
\[
t^M = \phi_1( \phi_4(t))
\]
and using our formula for $\phi_1$ we have
\[
\begin{aligned}
-t^M &= q_1( \phi_4(t)^M) + (\phi_4(t))^{2L_1M} \psi_1( \phi_4(t))\\
&= q_1(\phi_2(t)) + \phi_2(t)^{2L_1} \psi_1(\phi_4(t))\\
&= q_1(\phi_2(t)) + (-q_2'(0))^{2L_1} \psi_1(0) t^{2L_1M} + O(t^{2L_1M+1})
\end{aligned}
\]
where the last term is a stand-in for some analytic function vanishing to order at least $2L_1M+1$.
The last equality follows from
\[
\begin{aligned}
\phi_2(t)^{2L_1} \psi_1(\phi_4(t))& = 
t^{2L_1M}( -q_2'(0)+ O(t))^{2L_1} (\psi_1(0)+ O(t))\\
&= t^{2L_1M} (-q_2'(0))^{2L_1} \psi_1(0) + O(t^{2L_1M+1})
\end{aligned}
\]
using initial expansions of $\phi_2, \psi_1$, and $\phi_4$.

Expanding the composition with $q_1$ we have
\[
q_1( -q_2(t^M) - t^{2L_2M} \psi_2(t)) = q_1(-q_2(t^{M})) 
- t^{2L_2M} q_1'(0) \psi_2(0) + O(t^{2L_2M+1})
\]
which follows from examining $q_1$ term by term via
\[
( -q_2(t^M) - t^{2L_2M} \psi_2(t))^{j} = (-q_2(t^M))^{j} + O(t^{2L_2M+1})
\]
which holds since $q_2$ vanishes at $0$.
We single out the linear term 
\[
( -q_2(t^M) - t^{2L_2M} \psi_2(t)) = -q_2(t^M) - t^{2L_2M}\psi_2(0) + O(t^{2L_2M+1})
\]
to keep track of the contribution of the first non-real complex term.
Altogether
\[
\begin{aligned}
-t^M =& q_1(-q_2(t^{M})) 
- t^{2L_2M} q_1'(0) \psi_2(0) + O(t^{2L_2M+1})\\
&+ t^{2L_1M} (-q_2'(0))^{2L_1} \psi_1(0) + O(t^{2L_1 M+1}).
\end{aligned}
\]
Recall that $\psi_1(0), \psi_2(0) \in \UHP$ while
$q_1'(0), q_2'(0)  \ne 0$, $q_1,q_2 \in \R[t]$.
In order for the imaginary coefficients
on the right hand side to vanish we must have $L_1 = L_2$.
Then, more simply put
\[
-t^M = q_1(-q_2(t^M)) + O(t^{2L_1M}).
\]
We can replace $s=-t^M$ to see that
\[
s = q_1(-q_2(-s)) + O(s^{2L_1}).
\]
In words, $-q_2(-s)$ is a truncation of the 
power series for the functional inverse of $q_1$ (truncated
below order $2L_1$);
namely, $-q_2(-s) = I_{2L_1}(q_1)(s)$---exactly what we wanted to show.
\end{proof}

\subsubsection{Universal contact order} \label{sec:uco}
Assume the setup and conclusion of Theorem \ref{thm:purelp}.
In \cite{bps18, bps19a} contact order, which is the maximum of 
the cutoffs in Theorem \ref{thm:purelp}
\[
\max\{2L_1,\dots, 2L_k\},
\]
appeared naturally
in the study of integrability of derivatives of rational inner functions.
A condition which is easier to analyze in the context
of boundary regularity of rational inner functions 
is \emph{universal contact order}.
The \emph{universal contact order} of $p$ is
the minimum 
\[
K_{min} = \min\{2L_1,\dots, 2L_k\}.
\]
Its geometric content is that all branches of $p$
have this order of contact or higher with the distinguished boundary.
In other terms, generically we can match up 
branches of $A+t_1 B$ and $A+t_2 B$ so that
they all have order of contact $K_{min}$ or higher.

\begin{remark} \label{bidiskuco}
These interpretations of universal contact order
are conformally invariant and give a sensible geometric
way of defining this concept for atoral stable
polynomials (i.e. the bidisk setting).  
The explanation parallels that of
contact order in Section \ref{subsec:seg}.
\end{remark}

Recall the factorization from Theorem \ref{thm:segments}
\begin{equation} \label{AtBfactor}
A(z)+tB(z) = u(z;t) \prod_{j=1}^{k} \prod_{m=1}^{M_j} (z_2 + q_j(z_1) + z_1^{2L_j}\psi_{j,m}(z_1;t))
\end{equation}
which holds for all but finitely many $t\in \R$.
Here $u(z;t) \in \R\{z_1,z_2\}$ is a uniquely determined unit for each $t$ but it
is not clear how $u$ depends on  $t$.  
Universal contact order tells us something about this dependence.

\begin{theorem}\label{thm:ucohomog}
   Assume $p\in \C[z_1,z_2]$ is pure stable and 
   has universal contact order $K_{min}\geq 2$ (an even integer).  
  Write $p = A+i B$ into real and imaginary (coefficient) polynomials
  and if necessary, multiply by a constant so that the lowest homogeneous term $P_M$ satisfies $P_M =  A_M$
  and the coefficient of $z_2^M$ in $A_M$ is $1$. 
  Then, for generic values of $t$, 
  the local factorization \eqref{AtBfactor} of $A+tB$ 
  has the property that in the homogeneous expansion of the unit
  \[
  u(z;t) = 1 + \sum_{j\geq 1} u_j(z;t),
  \]
  the polynomials $u_j(z;t)$ are affine in $t$ for $j \leq K_{min}-2$.  
\end{theorem}	
 
 This theorem is used later when we study boundary regularity of rational inner functions.
 Note that the theorem is vacuous for $K_{min}=2$ as it should be.
 
 \begin{proof}
Given universal contact order $K=K_{min}$ we can rewrite \eqref{AtBfactor}
in a more convenient form where we ``forget'' some currently irrelevant 
information (i.e. portions of initial segments beyond order $K$)
\[
A(z)+tB(z) = u(z;t) W(z;t) = u(z;t) \prod_{j=1}^{M} (z_2 + h_j(z_1) + z_1^{K} \psi_j(z_1;t)).
\]
 Here $h_j \in \R[z_1]$ are polynomials with $h_j(0)=0$, $h'_j(0) >0$, $\deg h_j < K$
 and $\psi_j(z_1;t) \in \R\{z_1\}$.
Expanding we have 
\[
W(z;t) := 
 A_M(z) + R_{M+1}(z) + \dots + R_{M+K-2}(z) + F_{M+K-1}(z;t)
 \]
 where $R_{j}$ is a degree $j$ homogeneous polynomial and $F_{M+K-1}$
 is analytic in $z$ and vanishes to order at least $M+K-1$.
 The $R_j$ have no $t$ dependence but $F_{M+K-1}$
 may.
 Next, we consider the effect on the unit $u(z;t)$
 in the factorization \eqref{AtBfactor}.
 Write $u(z;t) = 1 + \sum_{j\geq 1} u_j(z;t)$
 where $u_j$ is homogeneous of degree $j$ in $z$.
 Then, for fixed $z$ 
 \[
 \frac{A(\lambda z)+tB(\lambda z)}
 {W(\lambda z;t)} = 1 + \sum_{j\geq 1} \lambda^j u_j(z;t)
 \]
 can be viewed as an analytic function of $\lambda$ and
 because the numerator and denominator vanish to order $M$
 the result is analytic and non-zero at $0$ whenever $A_M(z)\ne 0$.
    We can perform division of power series 
   (via solving a triangular system) to conclude that 
   $u_j(z;t)$ is affine with respect to $t$.  
   Indeed, 
   \[
   A_{M+1}(z) + t B_{M+1}(z)  = A_M(z) u_1(z;t)  + R_{M+1}(z)
   \]
   implies 
   \[
   u_1(\cdot;t) = \frac{A_{M+1}-R_{M+1}}{A_M} + t \frac{B_{M+1}}{A_M}
   \]
   and then recursively
   \[
   u_n(\cdot;t) = A_M^{-1} (A_{M+n} + t B_{M+n} - (R_{M+1} u_{n-1} + \cdots + R_{M+n}))
   \]
   is affine with respect to $t$ for $n \leq K-2$.  
   \end{proof}

This concludes our local theory related to Puiseux factorizations.

\subsection{Realization formulas} \label{ss:rf}

Our third and final method for analyzing local behavior
of stable polynomials is via transfer function realization formulas.
This technique is decidedly ``global''  and restricted to two dimensions but can be effectively 
utilized for local questions.  
It is difficult to disentangle this topic from our applications
but it has been so important for theorem discovery and proof
that it deserves some discussion here.

Let $p \in \C[z_1,z_2]$ be atoral stable with multidegree $n=(n_1,n_2)$
and corresponding rational inner function $\phi = \tilde{p}/p$.
Then there exists a $(1+|n|)\times (1+|n|)$ unitary $U = \begin{bmatrix} A & B \\ C & D\end{bmatrix}$
such that
\[
\phi(z) = A + B z_P (I - D z_P)^{-1} C = A + B(I-z_P D)^{-1} z_P C
\]
where $z_P = z_1 P + z_2 (I-P)$ for some $|n|\times |n|$ projection $P$ onto
an $n_1$ dimensional space.
Conversely, every such formula produces a rational inner function.  
This is easiest to see from the formula
\[
\begin{bmatrix} A & B \\ C & D \end{bmatrix} \begin{bmatrix} 1 \\ z_P (I-D z_P)^{-1} C \end{bmatrix} = \begin{bmatrix} \phi(z) \\ (I-Dz_P)^{-1} C\end{bmatrix}.
\]
By Cramer's rule and a proper accounting of degrees, one can show that
\[
p(z) = p(0) \det(I-D z_P).
\]
Thus, every atoral stable polynomial has a contractive determinantal representation
of the above form.
When $|\alpha| =1$, $p(z)- \alpha \tilde{p}(z)$ is toral stable and 
has a unitary determinantal representation
\[
p(z) - \alpha\tilde{p}(z) = (p(0)-\alpha \tilde{p}(0)) \det(I - V_{\alpha} z_P)
\]
where
\[
V_{\alpha} = D + \frac{\alpha}{1-\alpha A} CB
\]
is a unitary for $|\alpha| = 1$.  It is also contractive for $|\alpha|\leq 1$
and the above formulas hold for $\alpha \ne 1/A = p(0)/\tilde{p}(0)$.  Since $V_{\alpha}$
is unitary for  $\alpha\in \T$, we also see that $D$ is a rank one
perturbation of a unitary.
This is presented
in Section 9 of \cite{Kne19}.

One can use a Cayley transform to get determinantal representations of 
polynomials with no zeros on $\UHP^2$.
If $p \in \C[z_1,z_2]$ has no zeros on $\UHP^2$ and total degree $n$
then there exist a constant $c\in \C$ and $n \times n$ matrices $A_0, A_1, A_2$ satisfying
$\text{Im}(A_0), A_1,A_2 \geq 0$, $A_1+A_2 = I$ such that
\[
p(z) = c \det(A_0 + A_1 z_1 + A_2 z_2).
\]
If $p$ is real stable one can refine the representation so that $\text{Im}(A_0)= 0$. 
For details and additional discussion in the generic stable case, we refer the reader to Theorem 3.2 in \cite{Kne19b}. For details about the real stable refinement, we recommend the reader consult Theorem 6.6 in \cite{BB10} and Corollary 1 in \cite{Kne16}, with the caveat that their discussions occur in the language of hyperbolic polynomials. 
(An even deeper result of Helton-Vinnikov \cite{HV07} implies we can take $A_0,A_1,A_2$ to
be real symmetric.) 
In terms of local theory, the structure of the (possible) kernel of $A_0$ in relation
to the operators $A_1,A_2$ can reveal properties of the the zero set of $p$ near $(0,0)$.
However, the material in the previous sections on homogeneous and Puiseux expansions
seems more appropriate for understanding atoral/pure stable polynomials
and their associated toral/real stable perturbations (namely, $p-\alpha \tilde{p}$ in the
polydisk setting and $A+tB$ in the upper half plane setting).
More general perturbations, described next, are closely related to rational non-inner Schur functions,
and more general realization formulas are an effective tool
for their study.  

Consider a nonconstant rational Schur function on $\D^2$, namely $f = q/p$, where $p,q\in \C[z_1,z_2]$
have no common factors,
$p$ has no zeros in $\D^2$ and $|f(z)|\leq 1$ on $\D^2$.
By Theorem \ref{thm:nontanchar} below, $\mathcal{Z}_p\cap \T^2 \subset \mathcal{Z}_q\cap \T^2$
and therefore any potential toral factors of $p$ would be factors of $q$ and can 
be divided out.  So, $p$ is necessarily atoral stable.
Now, since
$p(z) + w q(z) \ne 0$ for $z \in \D^2$ and  $w\in \overline{\D}$ 
we get a more general family of perturbations
of stable polynomials
than simply $p+ w \tilde{p}$
by considering $p+w q$.
Further interest in perturbations of stable polynomials comes from
a problem studied in \cite{Kne19} of characterizing the extreme points
of real rational Pick functions.

Rational Schur functions on $\D^2$ possess contractive
transfer function realizations.  
Indeed, by Theorem 1.3 in \cite{Kne20}, there is a finite-dimensional Hilbert space $\mathcal{H}$,  
a contraction $U = \begin{bmatrix} A & B \\ C & D \end{bmatrix}$ on  $\mathbb{C} \oplus \mathcal{H}$, 
and a projection $P$ on $\mathcal{H}$ such that 
\begin{equation} \label{eqn:ftfr}
 f(z) = A + B(I-z_PD)^{-1}z_PC \text{ for } z \in \D^2,
 \end{equation}
where $z_P = z_1 P + z_2 (I-P)$. 
Note that not only is $U$ now merely a contraction, but we also do not have clear control
on the (finite) dimension of $\mathcal{H}$.
As we have seen, local behavior is made more apparent
in the upper half plane setting so we perform change of variables
to $f:\D^2 \to \D$ to obtain a rational Pick function $g: \UHP^2 \to \UHP$.
Note that we alter both the domain and range of $f$
to obtain more natural formulas.  
Specifically,  define a conformal map $\gamma: \mathbb{D} \rightarrow \UHP$  by 
\begin{equation} \label{eqn:gammaps}
\gamma(z) =i \frac{1+z}{1-z}  \ \ \text{ so that } \gamma^{-1}(w) = \frac{w-i}{w+i}
\end{equation}
and set 
\begin{equation} \label{eqn:g}
g(w) = \gamma \circ f (\gamma^{-1}(w_1), \gamma^{-1}(w_2)).
\end{equation}
 
\begin{proposition} \label{prop:simpleg}
Taking $g, \mathcal{H}, U$ defined as above, if $I-U$ is invertible, 
then $g$ satisfies
\begin{equation} \label{eqn:realization} 
g(w) = c - \left \langle (w_P + S)^{-1} \alpha, \beta \right \rangle_{\mathcal{H}}
\end{equation}
for $w\in \mathbb{H}^2$, where
\begin{equation} \label{eqn:T}  
T:=i (I +U)(I-U)^{-1} := \begin{bmatrix} c & \beta^* \\ \alpha & S \end{bmatrix}, \text{ for some } c \in \C, \alpha, \beta \in \mathcal{H}, \text{ and } S \in \mathcal{L}(\mathcal{H}). 
\end{equation}
%\color{blue} 
Since $U$ is a contraction, 
$\text{Im}(T)= \frac{1}{2i}(T-T^*)\geq 0$, 
$c\in \UHP$, and $\text{Im}(S)\geq 0$.
%\color{black}
\end{proposition}
This formula follows from Theorems 4.1 and 4.2 in \cite{bptd20}. 
 Since both sides are rational functions, the formula
holds in $\UHP^2$ as well as any points where both sides do not have a pole.
We shall refer to $g$'s formula as a PIP (positive imaginary part) realization.
The formula for $g$ can also be used to construct rational Pick functions since
a function $g$ as in \eqref{eqn:realization} satisfies
\[
\begin{aligned}
\Im g(w) =& \left\langle \Im(T) \begin{pmatrix} 1 \\ -(w_P + S)^{-1} \alpha \end{pmatrix}, \begin{pmatrix} 1 \\ -(w_P + S)^{-1} \alpha \end{pmatrix}\right\rangle \\
&+ \langle \Im(w_P) (w_P + S)^{-1} \alpha, (w_P+S)^{-1} \alpha \rangle
\end{aligned}
\]
which is non-negative whenever $w \in \UHP^2$.   
The representation \eqref{eqn:realization} does not represent all rational Pick functions.
In particular, if $g$ satisfies \eqref{eqn:realization}
then  $\lim_{t \to \infty} g(it,it)  = c \ne \infty$
which is not true for all rational Pick functions.
Producing a unified representation for all rational Pick functions in two variables
turns out to be somewhat technical,
and the paper \cite{ATDY16} produces what are called type IV Nevanlinna representations
to cover all cases.  These representations are intricate and necessarily so.
This seems to be more of an issue of the behavior of $g$
at $\infty$ which is not what we are interested in here.
To account for this we find a conformal self-map of $\UHP$ 
that fixes $0$ and perturbs $\infty$ so that
an arbitrary rational Pick function is conformally equivalent
to one of the form \eqref{eqn:realization}.  

\begin{theorem}\label{makehsimple}
Let $h:\UHP^2 \to \UHP$ be a non-constant rational Pick function.
Then, there exist automorphisms 
$\sigma_1,\sigma_2: \UHP \to \UHP$ 
where $\sigma_2(0)=0$
such that $g(w) = \sigma_1 ( h ( \sigma_2(w_1),\sigma_2(w_2)))$
has a realization as in \eqref{eqn:realization} and \eqref{eqn:T}
with $g^*(0,0) = \lim_{t\searrow 0} g(it,it) \ne \infty$.
\end{theorem}

\begin{proof}
Let $m_1$ be a M\"obius transformation sending
$\D$ to $\UHP$ and $1$ to $0$.
Then, $f = m_1^{-1} \circ h \circ (m_1,m_1):\D^2 \to \D$
is an RSF.  As discussed above $f$ has a contractive
transfer function realization as in \eqref{eqn:ftfr}.
If necessary we replace $f$ with a unimodular multiple
in order to guarantee $f^{*}(1,1):=\lim_{r\nearrow 1} f(r,r) \ne 1$.  
This limit exists because $\zeta \mapsto f(\zeta,\zeta)$ is a one variable RSF.    
Since $\dim \mathcal{H} <\infty$ and $f$ is a nonconstant rational Schur function, 
there is a $\lambda \in \mathbb{T}$ with $\lambda \ne 1$ such that $(1-\lambda D)$ is invertible, $f$ is continuous at $(\lambda, \lambda)$, and $f(\lambda, \lambda) \ne 1$. Then \eqref{eqn:ftfr} extends to $(\lambda, \lambda)$. Define $\tilde{f}$ by $\tilde{f}(z) = f(\lambda z)$ and define a contraction $\tilde{U}$ on  $\mathbb{C} \oplus \mathcal{H}$  by
 \[ \tilde{U} = \begin{bmatrix} \tilde{A} & \tilde{B} \\ \tilde{C} & \tilde{D} \end{bmatrix} = \begin{bmatrix} A & B \\ \lambda C & \lambda D \end{bmatrix}.\]
Then for $z\in \mathbb{D}^2$, we have
\[ \tilde{f}(z) = \tilde{A} + \tilde{B}(I-z_P\tilde{D})^{-1}z_P \tilde{C}.\]
By our choice of $\lambda$,  $(I-\tilde{D})$ is invertible and 
\[ 
1-\tilde{A} - \tilde{B}(I-\tilde{D})^{-1}\tilde{C} = 1 - f(\lambda, \lambda) \ne 0.
\]
These two facts paired with the inverse formula for block $2\times 2$ matrices imply that $I-\tilde{U}$ is invertible. 
Using $\gamma$ as in \eqref{eqn:gammaps}, 
set $\tilde{g} := \gamma \circ \tilde{f} \circ \gamma^{-1}$
and
\[  \tilde{T}:=i (I +\tilde{U})(I-\tilde{U})^{-1} := \begin{bmatrix} c & \beta^* \\ \alpha & \tilde{S} \end{bmatrix}, \text{ for } c \in \mathbb{C}, \ \alpha, \beta \in \mathcal{H}, \text{ and } \tilde{S} \in \mathcal{L}(\mathcal{H}).\]
By Proposition \ref{prop:simpleg}, 
$\tilde{g}$ is a rational Pick function on $\UHP^2$ and for $w \in \UHP^2$,
\[\tilde{g}(w) = c - \left \langle (w_P + \tilde{S})^{-1} \alpha, \beta \right \rangle_{\mathcal{H}}.\]
We define our proposed $g$ by $g(w) = \tilde{g}(w + \gamma(\bar{\lambda})).$ 
Then setting $S = \tilde{S} + \gamma(\bar{\lambda}) I$, it is easy to see that 
\[ T:=\begin{bmatrix} c & \beta^* \\ \alpha & S \end{bmatrix} \]
still has positive imaginary part and 
 \[
 g(w) = c - \left \langle (w_P + S)^{-1} \alpha, \beta \right \rangle_{\mathcal{H}}.
 \]
 In the course of the proof, $g$ was obtained by applying M\"obius maps to $h$
 and pre-composing $h$ with
 $\sigma_2(w_1) := m_1( \lambda \gamma^{-1}(w_1 + \gamma(\bar{\lambda})))$
 in each component.
 Evidently, $\sigma_2(0)= 0$ and $g$ satisfies $g^*(0,0) = \gamma (f^*(1,1)) \in \mathbb{C}$, since $f^*(1,1) \ne 1.$
\end{proof}

With this
in hand then we can dig into the kernel structure
of $S$ in order to understand the behavior of $g$ 
near $(0,0)$ via realizations.

\begin{theorem} \label{localrealization}
Suppose $g:\UHP^2 \to \UHP$ is rational, non-constant,
and satisfies $g^*(0,0) := \lim_{t\searrow 0} g(it,it) \ne \infty$.
Further suppose $g$ possesses a PIP realization as in Proposition \ref{prop:simpleg}.
Let $\widehat{S}= S|_{\text{Range}(S)}:\text{Range}(S) \to \text{Range}(S)$ be 
the compression of $S$ to $\text{Range}(S)$.
Then, 
\begin{itemize}
\item $\alpha,\beta$ belong to  $\text{Range}(S)$, 
\item $\widehat{S}$ is invertible with positive imaginary part, and
\item \[
\begin{aligned}
g(w) &= c - \left \langle \left(\widehat{S}+w_{22} - w_{21} w_{11}^{-1} w_{12}\right)^{-1}  \alpha,  \beta  \right \rangle\\
&=c - \left \langle \widehat{S}^{-1}\left(I+(w_{22} - w_{21} w_{11}^{-1} w_{12}) \widehat{S}^{-1} \right)^{-1}  \alpha,  \beta  \right \rangle
\end{aligned}
\]
where $w_P  = w_1 P + w_2(I-P) =\begin{bmatrix} w_{11} & w_{12} \\ w_{21} & w_{22} \end{bmatrix}$ is
the block decomposition of $w_P$ according to $(\text{Range}(S))^{\perp} \oplus \text{Range}(S)$.
\end{itemize}
\end{theorem}

The value of this new decomposition is that singular behavior at $(0,0)$ is encapsulated within the
term $w_{22} - w_{21} w_{11}^{-1} w_{12}$.  

\begin{lemma} \label{lem:pip} Let $\mathcal{H}$ be a finite dimensional Hilbert space and assume that an operator $T$ on $\mathbb{C} \oplus \mathcal{H}$ 
with positive imaginary part is
given by
\[ T = \begin{bmatrix} c & \beta^* \\ \alpha & S \end{bmatrix}, \text{ for } c \in \mathbb{C}, \ \alpha, \beta \in \mathcal{H}, \text{ and } S \in \mathcal{L}(\mathcal{H}).\]
Then, $\alpha - \beta \in \text{Range}(S)$.
Moreover, $\text{Ker}(S) = \text{Ker}(S^*)$, so  $\mathcal{H}= \text{Ker}(S) \oplus \text{Range}(S)$ and with respect to this decomposition,
  $S = \begin{bmatrix} 0 & 0 \\ 0 & \widehat{S} \end{bmatrix}$ for $\widehat{S} =P_{\text{Range}(S)}S|_{\text{Range}(S)}.$
\end{lemma}

\begin{proof} We first prove the assertion about $\text{Ker}(S)$. Since $T$ has positive imaginary part, so does $S$. Thus, if $x \in \text{Ker}(S)$, then 
\[ \langle \text{Im}(S) x, x \rangle = 0, \text{which implies } \| \text{Im}(S)^{1/2} x \| =0 \text{, which implies } \text{Im}(S) x =0,\]
which shows that $S^*x =0$. A symmetric argument gives the reverse containment so $ \text{Ker}(S) = \text{Ker}(S^*)$. From this, $ \text{Ker}(S) = \left( \text{Range}(S) \right)^{\perp}$ and so $\mathcal{H} = \text{Ker}(S) \oplus \text{Range}(S)$. Writing $S$ with respect to this decomposition immediately gives  $S = \begin{bmatrix} 0 & 0 \\ 0 & \widehat{S} \end{bmatrix}$.

Now write $\alpha - \beta  = \gamma_1 + \gamma_2$, where $\gamma_1 \in \text{Range}(S)$ and $\gamma_2 \in \text{Ker}(S)$. Then apply $\text{Im}(T)$ to vectors $z \oplus m\gamma_2 \in \mathbb{C} \oplus \mathcal{H}$ where $m \in \mathbb{C}$.  If $\gamma_2 \ne 0$, appropriate choices of $z$ and $m$ will give contradictions to $\text{Im}(T) \ge 0$. Thus, $\gamma_2=0$ and $\alpha -\beta \in \text{Range}(S).$
\end{proof}

\begin{proof}[Proof of Theorem \ref{localrealization}]
By Lemma \ref{lem:pip} we have that $\hat{S}$ is invertible with positive imaginary part. 
To see that $\alpha,\beta \in \text{Range}(S)$,
we write operators and $\beta=\beta_1 + \beta_2$ using the decomposition $\text{Ker}(S) \oplus \text{Range}(S)$
and compute
\[ \begin{aligned}
g^*(0,0) =& c - \lim_{t \searrow 0} \left \langle (it I+S)^{-1} \alpha, \beta \right \rangle \\ 
 =&  c - \lim_{t \searrow 0} \left \langle (it I+S)^{-1}  (\alpha-\beta), \beta \right \rangle - \lim_{t \searrow 0} \left \langle (it I+S)^{-1} \beta, \beta \right \rangle \\
 =& c-\lim_{t\searrow} \left \langle \begin{bmatrix} it  & 0 \\ 0 & it + \widehat{S} \end{bmatrix}^{-1} \begin{bmatrix} 0 \\ \alpha - \beta \end{bmatrix} ,  \begin{bmatrix} \beta_1 \\ \beta_2 \end{bmatrix} \right \rangle \\
&- \lim_{t \searrow 0} \left \langle  \begin{bmatrix} it  & 0 \\ 0 & it + \widehat{S} \end{bmatrix}^{-1} \begin{bmatrix} \beta_1 \\ \beta_2 \end{bmatrix}, \begin{bmatrix} \beta_1 \\ \beta_2 \end{bmatrix}  \right \rangle \\
 =& c-\lim_{t\searrow} \left \langle (it +\widehat{S})^{-1} ( \alpha - \beta),\beta_2 \right \rangle\\
 & -\lim_{t \searrow 0} \left \langle  \begin{bmatrix} 1/(it)  & 0 \\ 0 & (it + \widehat{S})^{-1} \end{bmatrix} \begin{bmatrix} \beta_1 \\ \beta_2 \end{bmatrix}, \begin{bmatrix} \beta_1 \\ \beta_2 \end{bmatrix}  \right \rangle\\
=& c- \left \langle \widehat{S}^{-1} ( \alpha - \beta),\beta_2 \right \rangle - \lim_{t \searrow 0} \frac{1}{it} \| \beta_1 \|^2 - \left \langle \widehat{S}^{-1} \beta_2,\beta_2 \right \rangle.
\end{aligned} 
\]
Since this limit exists, $\beta_1 = 0$ and we have $\beta = \beta_2 \in \text{Range}(S)$. As $\beta, \alpha-\beta \in \text{Range}(S)$, so is $\alpha.$ 

Now we can rewrite
\[ g(w) = c-  \left \langle \left(\begin{bmatrix} w_{11} & w_{12} \\ w_{21} & w_{22} \end{bmatrix}  +\begin{bmatrix} 0 & 0 \\ 0 & \widehat{S} \end{bmatrix} 
\right)^{-1} \begin{bmatrix} 0 \\ \alpha \end{bmatrix} , \begin{bmatrix} 0 \\ \beta \end{bmatrix} \right \rangle, \]
where $w_p$ has been written using the decomposition $\mathcal{H} =\text{Ker}(S) \oplus \text{Range}(S).$ Since $\text{Im}(w_P) = \text{Im}(w_1) P +  \text{Im}(w_2) (I-P)$, it is a strictly positive operator for $w \in \UHP^2$.
Then $ \text{Im}(S+w_P)$ is strictly positive as well and so $(S+w_P)^{-1}$ exists. Let $Y = P_{\ker(S)} P P_{\ker(S)}$. Then $0 \le Y \le I$ and 
\[ w_{11} = w_1 Y + w_2 (I-Y)\]
has strictly positive imaginary part for $w \in \UHP^2$. This implies that $w_{11}$ is invertible. 
Then, omitting some calculations, the inverse formula for block $2\times 2$ operators
implies that $\widehat{S}+w_{22} - w_{21} w_{11}^{-1} w_{12}$
is invertible
and one can  show 
\[ 
\begin{aligned}
g(w) &= c - \left \langle \left(\widehat{S}+w_{22} - w_{21} w_{11}^{-1} w_{12}\right)^{-1}  \alpha,  \beta  \right \rangle\\
&=c - \left \langle \widehat{S}^{-1}\left(I+(w_{22} - w_{21} w_{11}^{-1} w_{12}) \widehat{S}^{-1} \right)^{-1}  \alpha,  \beta  \right \rangle.
\end{aligned}\]
\end{proof}

%%%%%%%%%%%%%%%%%%%%%%%%%%%%%%%%%%%%%%%%%%
\section{Non-tangential boundary regularity} \label{sec:ntlimits}

We now discuss non-tangential limits and more general non-tangential regularity
of rational inner functions and rational Schur functions.  
The first main result here is that RSFs have non-tangential limits at \emph{every} boundary point.
We proceed to give a necessary and sufficient condition for higher order boundary
regularity in terms of homogeneous expansions.
Some of the essence of these ideas is in \cite{Kne15}
but many ideas have been simplified and written for the upper half plane.
The final portion of this section goes deeper into two variable RIFs.
One of the main goals is to give a partial converse
to a theorem in \cite{bps18, bps19a} which 
says that the non-tangential boundary regularity of an RIF implies 
a certain amount of contact order of the associated
stable polynomial.
We show  that
universal contact order (see Section \ref{sec:uco}) 
of a stable polynomial implies non-tangential boundary
regularity of the associated RIF.

\subsection{Non-tangential limits of RSFs} 
In \cite{Kne15}, several approaches and basic results for studying boundary behaviors of rational functions on both $\mathbb{D}^2$ and more generally, on $\mathbb{D}^d$ were established. 
Extending these techniques, we give a proof of existence of non-tangential limits for bounded rational functions $f=\frac{q}{p}$ at every $\zeta \in \mathbb{T}^d$. In order to prove our theorem, we only need to study singular points $\tau \in \mathbb{T}^d$ where $p(\tau)=0$.
As this is a local question, we analyze it at $0$ in the upper half plane setting $\UHP^d$.

Broadly speaking we define non-tangential approach regions to a boundary 
point to be regions where the
distance to the boundary point in question
 is comparable to the distance to the boundary.
 This notion is invariant under conformal maps between $\D$ and $\UHP$
 so our results have straightforward conversions between
 $\D^d$ and $\UHP^d$.
To precisely define non-tangential approach regions to $0$ via  $\UHP^d$, we define
\begin{equation} \label{Dz}
D_z=\{|z_1|,\dots, |z_d|, \Im z_1, \dots, \Im z_d\}
\end{equation}
for any $z\in \UHP^d$.  Then, a non-tangential approach region to $0$ via  $\UHP^d$ is a set
\[
AR_c = \{z\in \UHP^d: c\geq x/y \geq 1/c \text{ for any } x,y\in D_z\}
\]
for $c\geq 1$.  Letting $z\to 0$ non-tangentially is equivalent
to letting $r:=|z_1| \to 0$ while restricting $z \in AR_c$.

Then, a rational function $f = \frac{q}{p}$ on $\UHP^d$ is non-tangentially bounded at $0$
if it is bounded on $AR_c\cap \{z\in\UHP^d: |z_1|<r\}$ for $c\geq 1$ and $r>0$ sufficiently small.
Similarly, $f$ has a non-tangential limit $\omega$ at $0$ if the limit
\[
f^*(0)=\lim_{\underset{z\in AR_c}{z\to 0}} f(z)
\]
exists and equals $\omega$.

Luckily, it is not necessary to dwell on these definitions as they have direct connections to 
homogeneous expansions. 
%\color{blue} 
Such connections were established in \cite{Kne15} in the setting of the polydisk. 
The following theorem records the 
upper half plane analogues of those key results from \cite{Kne15}.  
%\color{black}

\begin{theorem} \label{thm:nontanchar}
Let $p\in \C[z_1,\dots, z_d]$ have no zeros in $\UHP^d$ and assume
$p$ vanishes to order $M$ at $0$.  Let $q\in \C[z_1,\dots,z_d]$ and $f=q/p$.  Then,
\begin{enumerate}
\item $f$ is non-tangentially bounded at $0$ in $\UHP^d$ if and only if $q$ vanishes to
order at least $M$ at $0$.
\item $f$ has a non-tangential limit at $0$ via $\UHP^d$ if and only if $q$ vanishes to order at least $M$
and the $M$-th order homogeneous term in $q$, say $Q_M$, is a constant multiple of the
$M$-th order homogeneous term in $p$, namely $P_M$; i.e.\ we have $Q_M = bP_M$.
In this case we have
\[
f^*(0)=\lim_{z\to 0, z\in AR_c} f(z) = b.
\]
\end{enumerate}
\end{theorem}

 The main theorem proved in this section is the following:

\begin{theorem}\label{thm:uhplim}
Assume $p,q \in \C[z_1,\dots,z_d]$, $p$ has no zeros in $\UHP^d$,
and $f = q/p$ is bounded and analytic in $\UHP^d$.
Then, $f$ has a non-tangential limit at $0$ via $\UHP^d$.
\end{theorem}

In light of Theorem \ref{thm:nontanchar},
the only thing to prove is that $Q_M = b P_M$ where
$p$ vanishes to order $M>0$ with lowest order homogeneous term $P_M$, 
and $q$ has $M$-th order homogeneous term $Q_M$ (which could be zero).  This follows immediately from Lemmas \ref{lem:limits1} and \ref{lem:limits2} given below. 
Note that we are assuming $f$ is non-constant and $M>0$ since otherwise
the result is trivial.

\begin{lemma} \label{lem:limits1}
Assuming the above setup,
we have $|f|<c$ in $\UHP^d$
if and only if the $d+1$ variable polynomial
\[
c(w+i) p(z) - (w-i) q(z) \in \C[z_1,\dots,z_d, w]
\]
has no zeros in $\UHP^{d+1}$.
\end{lemma}

\begin{proof}
We have $c>|\frac{q}{p}|$ in $\UHP^{d}$ if and only if
\[
c - \zeta \frac{q(z)}{p(z)}
\]
is non-vanishing for $\zeta \in \overline{\D}, z\in \UHP^d$
which happens if and only if
\[
c - \zeta \frac{q(z)}{p(z)}
\]
is non-vanishing for $\zeta \in \D, z\in \UHP^d$
by the maximum principle (since $f$ is assumed 
nonconstant).
This is equivalent to
\[
c - \frac{w-i}{w+i} \frac{q(z)}{p(z)}
\]
being non-vanishing for $w \in \UHP, z\in \UHP^d$
which in turn holds if and only if
\[
c(w+i)p(z) - (w-i) q(z)
\]
is non-vanishing on $\UHP^{d+1}$.  
\end{proof}

\begin{lemma}  \label{lem:limits2}
Assuming the above setup, if $f$ is bounded on $\UHP^d$ 
then $Q_M$ is a constant multiple
of $P_M$.
\end{lemma}

\begin{proof}
Recall that Theorem \ref{thm:stablehomog} says that for $p\in\C[z_1,\dots,z_d]$ with
no zeros in $\UHP^d$ and $p(0)=0$, the lowest order term $P_M$ in the homogeneous expansion
of $p$ has no zeros in $\UHP^d$ and is a multiple of a polynomial with real coefficients.
We we may assume without loss of generality that $P_M \in \R[z_1,\dots z_d]$.

 Choose $c >0$ so that $|\frac{q}{p}|<c$  in $\UHP^d$.
Then $|e^{i\theta} \frac{q}{p}|$ is also bounded by $c$ for every $\theta\in \mathbb{R}$.
By Lemma \ref{lem:limits1} and Theorem \ref{thm:stablehomog}, 
the lowest order homogeneous
term of 
\[
c(w+i)p(z) - (w-i) e^{i\theta}q(z)
\]
has real coefficients up to
a unimodular multiple.
The bottom homogeneous
term is 
\[
i(c P_M + e^{i\theta} Q_M).
\]
Thus,
for every $\theta \in \mathbb{R}$ there exists $\psi \in \mathbb{R}$ 
such that 
\[
e^{i\psi}(c P_M + e^{i\theta} Q_M)
\]
has real coefficients.

Write the coefficients of $P_M,Q_M$ as $p_{\alpha}, q_{\alpha}$.
Consider two distinct indicies $\alpha, \beta$ 
where $p_{\alpha}\ne 0$.
Then,
\[
e^{i\psi}(cp_{\alpha} + e^{i\theta} q_{\alpha}) \text{ and } e^{i\psi}(cp_{\beta} + e^{i\theta} q_{\beta}) 
\]
are both real-valued.  So,
\[
(cp_{\alpha} + e^{i\theta} q_{\alpha})(cp_{\beta} + e^{-i\theta} \overline{q_{\beta}})
\]
is real for all $\theta$. Viewing this as a trigonometric polynomial
we see that the coefficient of $e^{i\theta}$ and the coefficient of $e^{-i\theta}$
must be conjugate so 
\[
p_{\beta} q_{\alpha} = p_{\alpha} q_{\beta}
\]
and therefore
\[
q_{\beta} = \frac{q_{\alpha}}{p_{\alpha}} p_{\beta}.
\]
This holds for an arbitrary index $\beta$, so
this implies $Q_M = \frac{q_{\alpha}}{p_{\alpha}} P_M$.
\end{proof}

Similar to \cite{Kne15}, we can also study the existence of boundary directional derivatives. 
For $v \in \UHP^d$ the directional derivative of $f$ at $0$ in direction $v$ is given by
\[
D_v f(0) = \lim_{r\to 0^{+}} \frac{f(rv)-f^*(0)}{r}
\]
where $f^{*}(0)$ is the non-tangential limit of $f$ at $0$.

\begin{theorem} \label{thm:dird}
Let $p,q \in \mathbb{C}[z_1, \dots, z_d]$ 
and assume $p$ has no zeros in $\UHP^d$. If $f = \frac{q}{p}$
is bounded in $\UHP^d$,
then $D_vf(0)$ exists for every $v \in \UHP^d$.  
\end{theorem}

\begin{proof}
As above, we assume $p$ vanishes to order $M>0$ else $f$ is smooth at $0$.  
Let us write out homogeneous expansions, $p = \sum_{j=M}^{n} P_j$, $q= \sum_{j=M}^{n} Q_j$
where $n$ is  the maximum of the total degrees of $p$ and $q$.  By Theorems \ref{thm:nontanchar} and \ref{thm:uhplim},
we have $Q_M = bP_M$ and $f^*(0)=b$.
Then,
\[
\begin{aligned}
\frac{f(rv) - b}{r} &= \frac{1}{r} \left(\frac{q(rv)}{p(rv)} - b\right) \\
&=   \frac{1}{r} \frac{Q_{M+1}(rv) -b P_{M+1}(rv) + \sum_{j>M+1} (Q_j(rv)-bP_j(rv))}{ P_{M}(rv) + \sum_{j>M}   P_j(rv)}\\
&= \frac{Q_{M+1}(v) -b P_{M+1}(v) + \sum_{j>M+1} r^{j-M-1} (Q_j(v)-bP_j(v))}{ P_{M}(v) + \sum_{j>M} r^{j-M}  P_j(v)}.
\end{aligned}
\]
Sending $r\to 0^{+}$ we get
\[
D_vf(0)= \frac{ Q_{M+1}(v)-b P_{M+1}(v)}{P_{M}(v)},
\]
which exists because $P_M$ is nonvanishing on $\mathbb{H}^d.$
\end{proof}

This translates easily to the polydisk $\D^d$.
In particular, for each direction $-\delta$ pointing into $\mathbb{D}^d$ at $\tau$ let $D_{-\delta}f(\tau)$ denote the associated directional derivative
\[D_{-\delta}f(\tau) = \lim_{r \rightarrow 0^+} \frac{ f(\tau-\delta r)-f^*(\tau)}{r},\]
where $f^*(\tau)$ is the non-tangential limit of $f$ at $\tau$.  

\begin{theorem} \label{thm:derivative} Let $p,q \in \mathbb{C}[z_1, \dots, z_d]$ 
and assume $p$ has no zeros in $\mathbb{D}^d$. If $f = \frac{q}{p}$
is bounded in $\D^d$,
  then for every $\tau \in \mathbb{T}^d$, $f$ has a directional derivative $D_{-\delta}f(\tau)$ for every direction $-\delta$ pointing into $\mathbb{D}^d$ at $\tau.$  In particular, the directional derivative at $\tau = (1,\dots, 1)$ is given by
\[
D_{-\delta}f(\tau) = \frac{ Q_{M+1}(\delta) - f^*(\tau) P_{M+1}(\delta)}{P_M(\delta)}
\]
where $P_M, P_{M+1}, Q_{M+1}$ are associated homogeneous terms of $p,q$.
\end{theorem}

\begin{remark} \label{rem:grad} Since Theorem \ref{thm:derivative} gives formulas for the directional derivatives, one can easily test whether a given $f$ has a non-tangential gradient at $\tau$ (i.e.\ whether the directional derivative formula is linear in $\delta$).  If $q$ vanishes to order $N$ greater than $M+1$, the formula $D_{-\delta}f(\tau)  \equiv 0$ holds and $f$ trivially has a non-tangential gradient. If $N=M+1$ or $N=M$, then $f$ has a non-tangential gradient at $\tau$ if and only if $P_{M}$ is a factor of $Q_{M+1}-f^*(\tau) P_{M+1}.$
\end{remark}

\begin{example} Let $p(z) = 2-z_1-z_2$, so that the associated RIF is $\phi(z) = \frac{2z_1z_2-z_1-z_2}{2-z_1-z_2}$. Then
\[ f(z)  := (1-z_1) \frac{2z_1z_2-z_1-z_2}{2-z_1-z_2} + 1\]
is in $H^{\infty}(\mathbb{D}^2)$,
the space of bounded analytic functions on $\D^2$.
 Writing $f = \frac{q}{p}$ and computing the homogeneous expansions at $\tau=(1,1)$  gives
\[ p(1+z_1, 1+z_2) = -z_1-z_2 \ \ \text{ and } \ \ q(1+z_1, 1+z_2) = -z_1 -z_2 -z_1^2-z_1z_2-2z_1^2z_2,\]
so $N=M=1$ and  
\[ \begin{aligned}
P_M(z) =- (z_1 +z_2), &\quad  P_{M+1}(z) = 0,\\
 Q_M(z) = -(z_1 +z_2), &\quad Q_{M+1}(z) = -z_1(z_1+z_2). \end{aligned}\]
Then Remark \ref{rem:grad} implies that $f$ has a non-tangential gradient at $(1,1)$. Moreover,  $f$ is bounded on $\overline{\mathbb{D}^2} \setminus \{(1,1)\}$. Thus, if we define $f(1,1) := f^*(1,1)=1$, then $f$ is continuous on $\overline{\mathbb{D}^2}$.  In contrast, one can show that $\phi$ is not continuous on $\overline{\mathbb{D}^2}$ and does not have a non-tangential gradient at $(1,1)$. 
%% See \cite{Kne15} 
\eox
\end{example}

\subsection{Higher non-tangential boundary regularity of rational functions}

In this section, we study when an analytic function $f:\UHP^d \to \C$ 
has a nontangential polynomial approximation of order $k$ at $0$.
We specifically look at when 
there exists
a polynomial $F\in \C[z_1,\dots,z_d]$ of degree at most $k$ such that
as $z \to 0$ non-tangentially within $\UHP^d$ we have
\begin{equation} \label{eqn:Ckdef}
f(z)  = F(z) + o(r^{k})
\end{equation}
where $r=|z_1|$. One can take any equivalent quantity in $D_z$ in place of $r=|z_1|$ (recall \eqref{Dz}).
It turns out that for certain rational functions
when this non-tangential ``little-o'' condition holds, then a non-tangential 
``big-O'' condition automatically
holds.
Following terminology in the literature, see \cite {AM14, Kne15}, we will say $f$ is ``non-tangentially $C^k$'' at $0$ if \eqref{eqn:Ckdef} holds. However, we caution the reader that this does not actually imply (even non-tangentially) continuity of the $k^{th}$ derivative of $f$ near $0$. 

It possible to characterize in simple algebraic
terms when a rational function is non-tangentially $C^k$.
First, write $f = q/p$ with $p,q\in \C[z_1,\dots,z_d]$ and $p$ having no
zeros in $\UHP^d$.  We assume at the very least
that $f$ has a non-tangential limit at $0$.
As in Theorem \ref{thm:nontanchar}, we can write 
\[
\begin{aligned}
p &= P_M + P_{M+1} + \text{ higher order terms} \\
q &= b P_M + Q_{M+1} + \text{ higher order terms}.
\end{aligned}
\]  
For fixed $z \in \C^d$, the one variable function
\[
\lambda \mapsto f(\lambda z) = \frac{b P_M(z) + \lambda Q_{M+1}(z) + \cdots}{P_M(z) + \lambda P_{M+1}(z) + \cdots }
\]
is analytic for $\lambda$ near $0$ when $P_M(z)\ne 0$.
We can expand into a power series
\begin{equation} \label{lamexpand}
f(\lambda z) = \sum_{j\geq 0} F_j(z) \lambda^j
\end{equation}
using power series division.  Notice that the $F_j$ are well-defined functions on $\{z: P_M(z) \ne 0\} \supset \UHP^d$.
While $F_j$ is homogeneous of order $j$, it need not be a polynomial.  Doing 
the power series division one can recursively show that the $F_j$ are rational with
denominator $P_M^j$.  
For instance, $F_0 = b$, $F_1 = \frac{Q_{M+1} - bP_{M+1}}{P_M}$,
\[
F_2 = \frac{Q_{M+2} -bP_{M+2} - P_{M+1} F_1}{P_M}.
\]

\begin{theorem}\label{thm:nontanreg}
Let $f:\UHP^d \to \C$ be analytic and rational $f=q/p$.  
Consider the expansion \eqref{lamexpand}.
Then, $f$ is non-tangentially $C^k$ at $0$ via $\UHP^d$ if and only if
the sum
\[
F(z) := \sum_{j=0}^{k} F_j(z) \text{ belongs to } \C[z_1,\dots,z_d],
\]
in which case we have non-tangentially
\[
f(z) = F(z) + O(r^{k+1})
\]
where $r = |z_1|$ or any other comparable quantity in $D_z$.
\end{theorem}

\begin{proof}
If $f$ is non-tangentially $C^k$, then there exists $G\in \C[z_1,\dots, z_d]$ of degree at most $k$
such that $f(z) - G(z) = o(r^{k})$ non-tangentially at $0$.  
In particular, for fixed $z\in \UHP^d$, 
\[
F(\lambda z) - G(\lambda z) =  (f(\lambda z) - G(\lambda z)) - (f(\lambda z) - F(\lambda z)) = o(|\lambda|^{k}) 
\]
which is only possible if $F\equiv G$.

Conversely, if $F \in \C[z_1,\dots, z_d]$ then by construction, $q-F p$ vanishes to order at least $M+k+1$
while $|P_M(z)| > c r^M$ in a non-tangential approach region.
This implies that
\[
f(z) - F(z) = \frac{q(z)-F(z) p(z)}{P_M(z) (1+ \sum_{j\geq 1} \frac{P_{M+j}(z)}{P_M(z)})}  = \frac{O(r^{k+1})}{1 + O(r)} = O(r^{k+1}).
\]
\end{proof}

\begin{example} \label{ex153:regularity}
Consider Example \ref{ex153homog} (also studied in Example \ref{ex153puiseux})
\[
P(z) = A(z) + i B(z) = (z_1+z_2-2z_1^3-6z_1^2 z_2) - i (z_1^2+z_1z_2-4z_1^3 z_2)
\]
and the associated rational Pick function
\[
f = -B/A = \frac{z_1^2+z_1z_2-4z_1^3 z_2}{z_1+z_2-2z_1^3-6z_1^2 z_2}.
\]
To examine its regularity we look at
\[
f(\lambda z) = \lambda \frac{z_1^2+z_1z_2-4z_1^3 z_2 \lambda^2}{z_1+z_2-(2z_1^3+6z_1^2 z_2)\lambda^2} = \sum_{j\geq 1} \lambda^j F_j(z).
\]
Performing the power series division we get
\[
F_1(z) = z_1, F_2=0, F_3(z) = 2z_1^3, F_4=0, F_5(z) =\frac{4z_1^5(z_1+3z_2)}{z_1+z_2}
\]
which shows $f$ is non-tangentially $C^4$ but not $C^5$.  
We will be able to read this directly from the Puiseux expansion in Example \ref{ex153puiseux}
using Theorem \ref{thm:ucoreg} in the next section.
It says that $f$ is non-tangentially $C^4$ at $(0,0)$ 
since $P$ has universal contact order $6$.
\eox
\end{example}

\subsection{Universal contact order implies boundary regularity}

Let $p \in \C[z_1,z_2]$ be pure stable and $p(0,0)=0$.
Define the rational inner function on $\UHP^2$, $\phi = \bar{p}/p$.
In this section we show that a universal contact order condition on $p$ implies non-tangential
regularity of $\phi$ at $(0,0)$.  It is more revealing to study
the associated Pick function $f = -B/A$, where $p = A+i B$ is
the decomposition of $p$ into real and imaginary (coefficient) polynomials.
It will not be difficult to then convert back and forth between $f$ and $\phi$ 
since
\[
\frac{1+if}{1-if} = \phi \quad \text{and}\quad f= i \frac{1-\phi}{1+\phi}  = i(1-\phi)(1-\phi + \phi^2 + \cdots).
\]

We shall normalize $p$ so that the lowest order homogeneous
term of $p$, say $P_M$, has non-negative real coefficients and the coefficient of $z_2^M$ in $P_M$ equals $1$. 
%\color{blue} 
As a note, Corollary \ref{cor:purelp}
describes both how to normalize $P_M$ to have non-negative real coefficients and gives a formula for $P_M$ that shows $z_2^M$ must have a nonzero coefficient. 
%\color{black}
Then by Proposition \ref{prop:B1}, 
$B$ vanishes to order $M+1$ and we therefore have $f^*(0,0)=0$.

\begin{theorem}\label{thm:ucoreg}
If $p$ has universal contact order at $(0,0)$ given by the
even integer $K_{min} \geq 2$ then 
$f$ is non-tangentially $C^{K_{min}-2}$ at $0$.
\end{theorem}

\begin{proof}
The key ingredients are Theorem \ref{thm:ucohomog} and Theorem \ref{thm:nontanreg}.
For simplicity in the proof we write $K_{min} = K$.
As in the proof of Theorem \ref{thm:ucohomog}, let
\begin{equation}\label{eqn:WP}
A(z) + t B(z) = u(z;t) W(z;t)
\end{equation}
be our Weierstrass preparation theorem factorization of $A+tB$.
For generic $t$,
\[
W(z;t) = A_M(z) + R_{M+1}(z) + \cdots + R_{M+K-2}(z) + F_{M+K-1}(z;t)
\]
where the $R_{M+j}$ are homogeneous polynomials of the indicated order
and do not depend on $t$
while $F_{M+K-1}$ is analytic in $z$, vanishes to order at least $M+K-1$,
and may depend on $t$.
By Theorem \ref{thm:ucohomog}, the homogeneous decomposition
\[
u(z;t) = 1 + \sum_{j\geq 1} u_j(z;t)
\]
satisfies 
$u_j(z;t) = G_j(z) + t H_j(z)$ 
for $j\leq K-2$ and generic $t\in \R$, 
where $G_j,H_j$ are homogeneous polynomials of order $j$.

Let $(z)^{M-K-1}$ denote the ideal in $\C\{z_1,z_2\}$ generated by homogeneous polynomials
of degree $M-K-1$.  This is often just a convenient notation for disregarding higher order
terms.
If we examine \eqref{eqn:WP} modulo $(z)^{M+K-1}$ then
\[
\begin{aligned}
&\sum_{j=0}^{K-2} A_{M+j}(z) \\
&= (1+ \sum_{j=1}^{K-2} G_{j}(z))(A_M(z) + R_{M+1}(z) + \cdots + R_{M+K-2}(z)) \text{ mod } (z)^{M+K-1}
\end{aligned}
\]
\[\begin{aligned}
&\sum_{j=1}^{K-2} B_{M+j}(z) \\
&= (\sum_{j=1}^{K-2} H_{j}(z))( A_M(z) + R_{M+1}(z) + \cdots + R_{M+K-2}(z)) \text{ mod } (z)^{M+K-1}.
\end{aligned}\]
As a result
\begin{equation} \label{eqn:WPdiv}
\begin{aligned}
&\frac{\sum_{j=1}^{K-2} B_{M+j}(z)+ (z)^{M+K-1} }{\sum_{j=0}^{K-2} A_{M+j}(z)+(z)^{M+K-1} } \\ =& 
\frac{(\sum_{j=1}^{K-2} H_{j}(z))( A_M(z) + R_{M+1}(z) + \cdots + R_{M+K-2}(z))}{(1+ \sum_{j=1}^{K-2} G_{j}(z))(A_M(z) + R_{M+1}(z) + \cdots + R_{M+K-2}(z))} \\
&=\frac{(\sum_{j=1}^{K-2} H_{j}(z))}{(1+ \sum_{j=1}^{K-2} G_{j}(z))}. 
\end{aligned}
\end{equation}
According to Theorem \ref{thm:nontanreg}, the
regularity of $f$ is governed by whether the initial terms of
\[
f(\lambda z) = \sum_{j\geq 1} \lambda^j F_j(z)
\]
are polynomials.
Note
\[
f(\lambda z) = - \frac{\sum_{j=1}^{K-2} \lambda^j B_{M+j}(z) + \text{ higher order terms}}{\sum_{j= 0}^{K-2} \lambda^j A_{M+j}(z) + \text{ higher order terms}}
%= \frac{(\sum_{j=1}^{K-2} \lambda^j H_{j}(z))}{(1+ \sum_{j=1}^{K-2} \lambda^j G_{j}(z))} 
\]
so the terms $F_j$ up to $j=K-2$ match those of \eqref{eqn:WPdiv}.
But the last function in \eqref{eqn:WPdiv} is analytic at $0$
so its homogeneous terms are polynomials.
Therefore, each $F_j$ for $j\leq K-2$
must be a polynomial.
Furthermore, since $B$ and $A$ have real coefficients, the $F_j$ have real coefficients.
\end{proof}

\subsection{Intermediate Loewner class and $B^J$ points}
In this section we connect the previous result to some 
work in \cite{bps18,bps19a}.  Since various past results 
are stated in upper half plane or polydisk settings 
a certain degree of flexibility is required from the reader.

Past work used the concept of a $B^J$ point 
to formulate non-tangential regularity.
In turn, the definition of a $B^J$ point is based on a class of Pick
functions (analytic maps from $\UHP^2$ to $\UHP$)
called the intermediate L\"owner class which is denoted $\mathcal{L}^{J_-}$. 
The class $\mathcal{L}^{J_-}$ was originally defined in \cite{Pas18}
 using behavior at $(\infty, \infty)$; we present a modified version 
of the class here, as was done in \cite{bps18},  using behavior at $(0,0)$ instead.

\begin{definition} For a positive integer $J$,  a two-variable Pick function $g:\UHP^2\to \UHP$ is in the \emph{intermediate L\"owner class at $(0,0)$}, denoted $\mathcal{L}^{J_-}$, if $ \lim_{s\searrow 0}|g(is, is )| =0$
 and if for $1 \le j \le 2J-2$, there exist homogeneous polynomials $G_j$ of degree $j$ with real coefficients such that
 \begin{equation}  \label{eqn:NTE2}  g(w) = \sum_{j=1}^{2J-2} G_j(w) +O \left( |w|^{2J-1} \right) \quad \text{ non-tangentially}.
 \end{equation}
\end{definition}

Define the following conformal maps 
\[
\begin{aligned}
\alpha \colon  \mathbb{D} \rightarrow \UHP, \ \alpha(z) := i \left[ \frac{1-z}{1+z} \right] \ \ \text{ and } \ \
\alpha^{-1} \colon \UHP \rightarrow \mathbb{D}, \ \alpha^{-1}(w) := \frac{1+iw}{1-iw}.
\end{aligned}
\]
 Using those,  we can translate $\mathcal{L}^{J-}$ to $\mathbb{D}^2$ and define $B^J$ points:
\begin{definition}  \label{def:BJ} Let $\phi$ be a RIF on $\mathbb{D}^2$ with a singularity at $\tau =(1, 1) \in \mathbb{T}^2$ with non-tangential value $1 \in \mathbb{T}$. Define 
\begin{equation} \label{eqn:gdef} g_{\tau} (w):= \alpha\left(\phi\left( \alpha^{-1}(w_1),  \alpha^{-1}(w_2) \right) \right) = i \left[ \frac{1- \phi( \alpha^{-1}(w_1), \alpha^{-1}(w_2))}{1+\phi(  \alpha^{-1}(w_1),  \alpha^{-1}(w_2))} \right].\end{equation}
Then $\tau$ is a \emph{$B^J$ point of $\phi$} if $g_{\tau}$ is in the intermediate L\"owner class  $\mathcal{L}^{J_-}$ at $(0,0)$.
\end{definition}

Combining Theorem $7.1$ from  \cite{bps18} with Theorem $4.1$ from \cite{bps19a} yields the following connection between $B^J$ points and contact order.
We still let $\tau = (1,1)$ here.

 \begin{theorem} 
 \label{thm:BJ1} Let $p$ be an atoral stable polynomial on $\mathbb{D}^2$ with $p(\tau)=0$. Let $\phi = \frac{\tilde{p}}{p}$ and assume $\phi^*(\tau)=1$. If $\tau$ is a $B^J$ point of $\phi$, then the contact order $K$ of $p$ at $\tau$ satisfies $K/2 \ge J.$ 
 \end{theorem} 

The previous section gives a partial converse to this. Specifically, assume that $p$ has universal contact order $K_{min}$ (see Remark \ref{bidiskuco})
at $\tau$ and define a pure stable $P= A+iB$ on $\mathbb{H}^2$ via \eqref{eqn:PH2}. Then $P$ has universal contact order $K_{min}$ at $(0,0)$. Setting $\psi =\bar{P}/P$, we have $\psi^*(0,0)=1$, so Theorem \ref{thm:nontanchar} implies that $A_M = P_M$.  Setting $g= -B/A$, Theorems \ref{thm:nontanreg} and \ref{thm:ucoreg} show that there exist homogeneous polynomials $G_j$ of degree  $j$ such that 
\[g(w) = \sum_{j=1}^{K_{min}-2} G_j(w) +O \left( |w|^{K_{min}-1} \right) \quad \text{ non-tangentially}.\]
The last sentence of the proof of Theorem \ref{thm:ucoreg} also shows that the $G_j$ have real coefficients. Tracking through the definitions gives
\[ g (w)= \alpha\left(\phi\left( \alpha^{-1}(w_1),  \alpha^{-1}(w_2) \right) \right),\]
which implies $\phi$ has a $B^{K_{min}/2}$ point at $\tau$. We can summarize this as follows:
 
\begin{corollary} 
 \label{cor:BJ1} Let $p$ be an atoral stable polynomial on $\mathbb{D}^2$ with $p(\tau)=0$. Let $\phi = \frac{\tilde{p}}{p}$ and assume $\phi^*(\tau)=1$. If $K_{min}$ is the universal contact order of $p$ at $\tau$, then $\tau$ is a $B^{K_{min}/2}$ point of $\phi$.
 \end{corollary}

\section{Horn regions and more general regularity}\label{sec:horns}

What can be said about non-non-tangential boundary behavior of RIFs or RSFs?  
One way to examine this behavior is to look ultra tangentially, specifically its behavior on the
distinguished boundary. 
%\color{blue} 
Because RIFs only take unimodular values on $\mathbb{R}^2,$ we can analyze their ultra-tangential behavior by studying their unimodular level sets restricted to $\mathbb{R}^2.$
This was addressed in \cite{bps18, bps19a} (in the bidisk setting) where it was basically shown that for RIFs $\phi$ and unimodular $e^{i\theta_0} \ne \phi^*(0,0)$, the unimodular level sets 
\[ \left \{ (x_1, x_2) \in \mathbb{R}^2: \phi( x_1, x_2) = e^{i\theta_0} \right\} \]
are constrained to horn shaped regions near $(0,0)$. This is illustrated in Figure \ref{hpfavehorn}, which shows that  the unimodular level sets pass through the singularity $(0,0)$ within a region that has at least quadratic pinching. Such detailed understanding of level sets of a RIF was in turn used to establish results concerning integrability of their partial derivatives, see \cite[Section 5]{bps18}.

\begin{figure}[h!]
	      {\includegraphics[width=0.44 \textwidth]{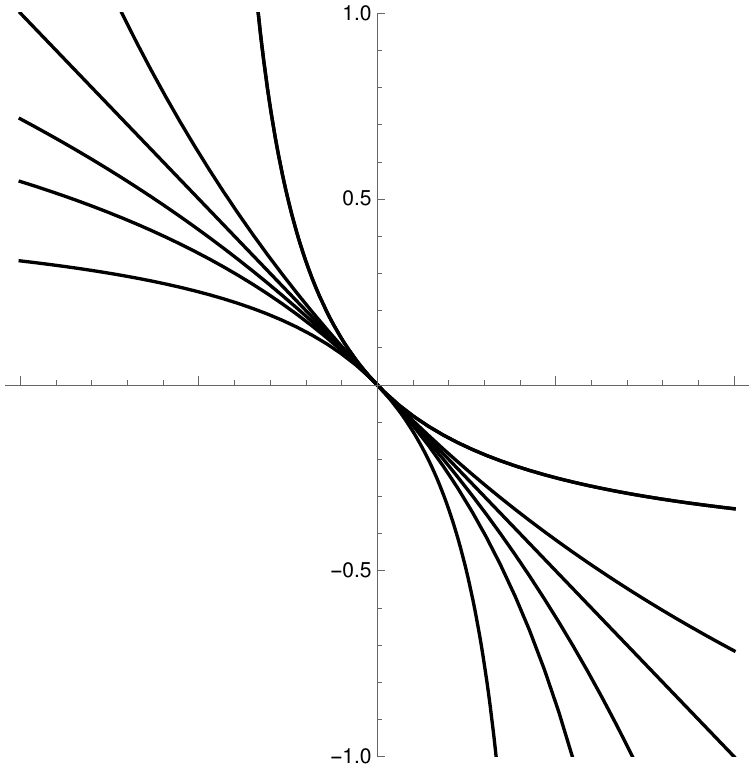}}
	  \caption{\textsl{Several unimodular level sets of $\phi=-\frac{x_1+x_2+2ix_1x_2}{x_1 + x_2-2ix_1x_2}$ in $\mathbb{R}^2$, indicating the presence of a horn region at $(0,0)$.}}
	  \label{hpfavehorn}
	\end{figure}

The local theory from Section \ref{sec:local} also reveals this level set behavior
and we pursue this line of thought in the next subsection. 
%\color{blue} 
Specifically, in prior work, the authors used realization theory to establish this. Below, we show how constrained level set behavior actually follows from  Puiseux series analysis. 

After that we delve into the more difficult question
of understanding  distinguished boundary behavior of RSFs.  Our goal is to obtain an analogue of the level set behavior described above for RIFs. In this setting, some functions, such as the one given in Figure \ref{faveRSFlevelsplot}, possess some level sets with curves that lie in $\mathbb{R}^2$. As shown in that figure, these level sets appear to approach the singularity within a region with quadratic pinching.

However, for general RSFs, it is no longer the case that  level sets will necessarily include entire curves that lie within $\mathbb{R}^2$.
Instead, in what follows, we show that if a level set (not associated to $\phi^*(0,0)$) contains a sequence of points on $\mathbb{R}^2$ converging to $(0,0)$, then those points have to get caught in a horn region that is independent of the particular sequence. This analysis is complicated and requires the use of our local realization formula from Section \ref{ss:rf}.
%\color{black}

\subsection{Horn regions for RIFs via Puiseux expansions}
Let $p$ be pure stable and let $\phi = \bar{p}/p$ be the associated RIF on $\UHP^2$.
Write $p = A+iB$ where $A = \frac{1}{2}(p+\bar{p})$, $B = \frac{1}{2i}(p-\bar{p})$.
If $p$ vanishes to order $M>0$ at $(0,0)$, then we multiply
$p$ by a constant so that $P_M = A_M$, and $B$ vanishes
to order $M+1$ by the Homogeneous Expansion
Theorem from the introduction.
The unimodular level sets of $\phi$ coincide with
the zero sets on $\R^2$ of $A- t B$ for $t\in \R$ or $B$ (which
corresponds to $t=\infty$) if we omit the points
of $\R^2$
such as $(0,0)$ where $p=0$ and $\phi$ is 
not defined.
These unimodular level sets also coincide with the
real level sets of the Pick function $f = A/B$ where again $f =\infty$
corresponds to $B=0$.  
In what follows, 
since we are mostly discussing behavior on $\R^2$ it is convenient
to use the variable $x=(x_1,x_2)$ instead of $z=(z_1,z_2)$.
As described in Section \ref{sec:puiseux}, $A-tB$
factors into $M$ analytic and real branches with negative slope at $(0,0)$.
Namely, there exist
 $\psi(x_1;t,j) \in \R\{x_1\}$  where $t\in \R$, $j=1,\dots, M$ with $\psi(0;t,j)=0, \psi'(0;t,j) > 0$ where
\[
A(x)-t B(x) = u(x;t) \prod_{j=1}^{M} (x_2+\psi(x_1;t,j))
\]
and we order the functions so that $\psi(x_1;t,j)$ is increasing with respect to  $j$ for fixed $x_1>0$ and $t\in \R$.
We can describe the ``level region'' described by $s_1\leq A/B \leq s_2$ in 
a punctured neighborhood of $(0,0)$ in $\R^2$ as a union of 
regions trapped between graphs of our analytic branches.
We will show this for positive $x_1$ since a similar 
result holds for negative $x_1$.  Note that
the ordering of the
branches $\{\psi(x_1;s,j)\}_{j}$ can change going from $x_1>0$ to $x_1<0$
and this is why we focus on $x_1>0$.

\begin{theorem}
Given the above setup,  for $s_1<s_2$ in $\R$, there exist $r,R>0$ such that
for $x \in (0,r)\times (-R,R)\setminus\{(0,0)\}$
the level region
\[
\{x : s_1 \leq A(x)/B(x) \leq s_2\}
\]
is given by
\begin{equation} \label{levelregions}
\bigcup_{j=1}^M \{(x_1,x_2): x_2 \in [-\psi(x_1;s_1,j),-\psi(x_1;s_2,j)]\}.
\end{equation}
\end{theorem}

\begin{proof}
Observe that
\[
A(0,z_2) - t B(0,z_2) = z_2^M(A_M(0,1)+ z_2(A_{M+1}(0,1)-t B_{M+1}(0,1)) + \cdots)
\]
and so for $t$ in any fixed compact interval $I=[s_1,s_2]\subset \R$ we can find $R$ such that
the above expression is non-zero for $0<|z_2|\leq R$ and $t\in I$.  
Then, there exists $r>0$ such that $A(z)-tB(z)\ne 0$
for $|z_1|\leq r$, $|z_2|=R$, $t \in I$.
By the argument principle, $z_2 \mapsto A(z_1,z_2)-tB(z_1,z_2)$ 
has $M$ zeros for $|z_1|\leq r, |z_2|<R, t\in I$.
Since $A-t B$ is real stable, for each fixed 
$x_1\in \R$ we see that the univariate polynomial 
$z_2 \mapsto A(x_1,z_2)- t B(x_1,z_2)$
is either real stable or identically zero
(as follows from Hurwitz's theorem by taking $z_1\in \UHP \to x_1 \in \R$).
A real stable univariate polynomial has only real zeros.
Therefore, for fixed $x_1 \in [0,r]$ and $t\in I$, 
the function of $x_2 \in (-R,R)$
given by  $x_2 \mapsto A(x_1,x_2) - t B(x_1,x_2)$
has $M$ (real) zeros.
We can shrink $r$ to force the zeros to be distinct
for $x_1\in (0,r]$.
We can further shrink $r$ if necessary to force
the branches $\psi(x_1;s_1,j), \psi(x_1,s_2,j)$ for $j=1,\dots, M$
to be analytic for $|x_1|\leq r$ and bounded by $R$.
Then, the $M$ real zeros of $x_2 \mapsto A(x)-s_1 B(x)$
are exactly given by $x_2 = -\psi(x_1;s_1,j)$ for
$j=1,\dots, M$ and similarly for $s_2$. 

Define $g(z_2) = A(x_1,z_2)/B(x_1,z_2)$
for fixed $x_1  \in (0,r)$.
Now, $g$ is a one variable non-constant Pick function
(i.e. maps $\UHP$ to $\UHP$)
which implies that $g$ is strictly
increasing on $\R$ except at poles where
it jumps from $\infty$ to $-\infty$.
On the interval $(-R,R)$
we have already established that $g$ attains
every value in the interval $I$ exactly $M$
times.  The points where $g$ attains
the values $s_1$ and $s_2$ must interlace
since $g$ increases.  Among these points
we claim that $g$ attains the value $s_1$
first.  Suppose $g$ attains the value $s_2$
first, say at $y_0 \in (-R,R)$.  Then,
$g$ alternates attaining $s_1$ and $s_2$,
say at points $y_1< y_2< y_3<\dots <y_{2M-1}$
where $g(y_1)=s_1, g(y_{2}) = s_2, \dots, g(y_{2M-1}) = s_1$.
So, $g$ maps the intervals $[y_1,y_2], [y_3,y_4], \dots, [y_{2M-3},y_{2M-2}]$
onto $I$ which accounts for $M-1$ times that $g$ attains the values
in $I$.
On the other hand, $g$ maps $(-R,y_0]$ onto $(g(-R), s_2]$ 
and $[y_{2M-1}, R)$ onto $[s_1, g(R))$.
Note $g(-R)  \geq s_1$ and $g(R) \leq s_2$ since
otherwise $g$ would attain $s_1$ or $s_2$
more than $M$ times.
If $g(R) > g(-R)$ then $g$ attains some values 
in $I$ more than $M$ times since then the
two intervals $(g(-R), s_2]$ and $[s_1,g(R))$
overlap.  
If $g(R) \leq g(-R)$, then the value $g(R)$ is only attained
$M-1$ times since in this case the
two intervals $(g(-R), s_2]$ and $[s_1,g(R))$
miss $g(R)$.  Thus, we conclude $g$ 
must attain the value $s_1$ first.

Therefore, $g$ alternates attaining the
values $s_1,s_2$ starting with $s_1$
exactly $M$ times ending with $s_2$.
These values are attained at
points given by consecutive branches
and thus 
$g$ maps $[-\psi(x_1;s_1,j), -\psi(x_1;s_2,j)]$
onto $I$ for $j=1,\dots, M$.
\end{proof}

Since $\psi'(0;t,j)$ is constant with respect to $t$ these
regions all have at least quadratic pinching at $(0,0)$:
there exists $k\geq 2$ such that 
\[
cx_1^k \leq |\psi(x_1;s_1,j) - \psi(x_1;s_2,j)| \leq C x_1^k.
\]
These regions are called \emph{horn regions} and they are
defined more formally in the next subsection.
If $p$ has contact order $K$ (necessarily even and at least $2$) 
we can always 
choose $j$ and a pair $s_1<s_2$ such that
\[
c x_1^K \leq |\psi(x_1;s_1,j) - \psi(x_1;s_2,j)| \leq C x_1^K.
\]

One basic conclusion of this is that if $x^{n} \to (0,0)$ in $\R^2$ and $f(x^n) \to s \in \R$
then eventually the points $x^n$ are trapped in the region \eqref{levelregions} for $s_1 = s-\delta, s_2= s+\delta$ for $\delta >0$.
Another conclusion is that the closure of $f(\D_r^2\cap \UHP^2)$ contains $\R$ for all $r>0$
since all of the level sets of $f$ pass through $(0,0)$ and points on these level sets can be perturbed 
to points of $\UHP^2$.

\subsection{Horn regions for RSFs via realization formulas}

In this section, we return to rational Schur functions in two variables and prove that they exhibit additional regularity properties possessed by rational inner functions. 
Let us formally define horn regions in $\R^2$.

\begin{definition}  \label{def:horn1} A \emph{non-trivial horn $\mathscr{H}$} at $(0,0)$ in $\mathbb{R}^2$ with slope $a\neq 0$ is a region of points $(x_1, x_2) \in \mathbb{R}^2$ sufficiently close to $(0,0)$ such that 
\begin{equation} \label{eqn:horn} |x_2 -a x_1| \le B x_1^2 \text{ for some fixed $B>0$.}  \end{equation}
A \emph{trivial horn $\mathscr{H}$} at $(0,0)$ in $\mathbb{R}^2$ is either
\begin{itemize}
\item Oriented along the $x_2$-axis and consists of $(x_1, x_2) \in \mathbb{R}^2$ satisfying
$|x_1 | \le B x_2^2$ for some fixed $B>0$, or
\item Oriented along the $x_1$-axis and  consists of $(x_1, x_2) \in \mathbb{R}^2$ satisfying
$|x_2 | \le B x_1^2$ for some fixed $B>0$.
\end{itemize}
\end{definition}

Definition \ref{def:horn1} generalizes the notion of a nontrivial horn from \cite{bps18}. There, the authors consider regions in $\mathbb{R}^2$ near $(0,0)$ with boundaries  
\[ x_2 =\frac{x_1}{m\pm b x_1} \text{ for $m <0$ and $b>0$.}\]
A simple power series computation shows that such regions satisfy an inequality of form \eqref{eqn:horn} and so, are encompassed by Definition \ref{def:horn1}.

The main theorem of this section is as follows.

\begin{theorem} \label{thm:horn} Let $f$ be a nonconstant rational Schur function on $\UHP^2$ with a singularity at $(0,0)$ and nontangential value $f^*(0,0)$. If a sequence $(x^n) \subset \R^2$ satisfies
\[ x^n \rightarrow (0,0)  \ \text{ and } \ f(x^n) \ \rightarrow \zeta_0 \ne f^*(0,0),\]
there is a finite number of horns $\mathscr{H}_1, \dots, \mathscr{H}_L$ in $\R^2$  at $(0,0)$ such that 
 for $n$ sufficiently large, each $x^n \in \bigcup_{\ell =1}^L \mathscr{H}_\ell$. Moreover, the number $L$ and the slopes of the horns do not depend on the sequence  $(x^n).$
\end{theorem}

The first reduction we make is to compose $f$ with a M\"obius map $m:\D\to \UHP$ 
and study a rational Pick function $h = m\circ f$.
By Theorem \ref{makehsimple} we can further apply
conformal maps to $h$ that fix $(0,0)$ to obtain a
rational Pick function $g:\UHP^2 \to \UHP$
with a local PIP realization formula as in Theorem \ref{localrealization}.
The image of a horn region under the image of a pair of M\"obius maps $\UHP\to \UHP$
that fix $0$ is still a horn region so we have not lost anything in our reduction.
We relabel our sequence $x^n$ accordingly and assume $x^n \to (0,0)$ and $g(x^n) \to \eta_0 \ne g^*(0,0)$.
We can assume $\eta_0 \ne \infty$; for instance if we replace our
original $f$ with $\frac{1}{2}f$ then $g$ will take values in a compact 
subset of $\UHP$.
Recall  that the formula for $g$ is 
\[
g(w) %= c - \left \langle \left(\widehat{S}+w_{22} - w_{21} w_{11}^{-1} w_{12}\right)^{-1}  \alpha,  \beta  \right \rangle\\
=c - \left \langle \widehat{S}^{-1}\left(I+(w_{22} - w_{21} w_{11}^{-1} w_{12}) \widehat{S}^{-1} \right)^{-1}  \alpha,  \beta  \right \rangle_{\mathcal{H}}
\]
where we reiterate that $\mathcal{H}$ is a finite dimensional Hilbert space, 
\begin{itemize}
\item $T = \begin{bmatrix} c & \beta^* \\ \alpha & S\end{bmatrix}  \in \mathcal{L}(\mathbb{C} \oplus \mathcal{H})$ has positive imaginary part,
\item $\widehat{S}$ is the compression of $S$ to $\text{Range}(S)$, which is reducing for $S$,  and $\alpha, \beta \in \text{Range}(S)$,
\item  $w_P  = w_1 P + w_2(I-P) =\begin{bmatrix} w_{11} & w_{12} \\ w_{21} & w_{22} \end{bmatrix}$ is
the block decomposition of $w_P$ according to $(\text{ker}(S)) \oplus \text{Range}(S)$, 
where $P$ is a projection on $\mathcal{H}$. 
\end{itemize}

We also define the compression $Y = P_{\text{ker}(S)} P|_{\text{ker}(S)}$ so that $w_{11} = w_1 Y + w_2 (I-Y)$.
Notice that 
\begin{equation} \label{gzero}
g^*(0,0) = \lim_{t\searrow 0} g(it,it) = c- \langle \widehat{S}^{-1} \alpha, \beta \rangle
\end{equation}
since for $w = (it,it)$ we have $w_{22}-w_{21} w_{11}^{-1} w_{12} = it I$.

\begin{lemma}\label{lem:horns}
Assume the setup above and fix a constant $C>0$.
Then, there exist finitely many Horn regions $\mathscr{H}_1,\dots, \mathscr{H}_L$ 
with slopes determined by the eigenvalues of $Y$
such that for all $x \in \R^2$ sufficiently close to $(0,0)$ satisfying
 $\| x \|^2 \| x_{11}^{-1}  \| > C$ 
 we have 
 \[
 x \in \bigcup_{j=1}^{L} \mathscr{H}_j.
 \]
 \end{lemma}
 
\begin{proof}
Let $0 \le t_1, \dots, t_L \le 1$ denote the eigenvalues of the positive matrix $Y$. Then, $x_{11}$ is diagonalizable and has eigenvalues $t_\ell x_1 + (1-t_\ell)x_2$.  This implies 
\[ \| x_{11}^{-1} \| = \max_{1 \le \ell \le L}  \frac{1}{\left| x_1 t_\ell + x_2(1-t_\ell)\right|}.\]
Then  $ \| x \|^2 \| x_{11}^{-1}  \| > C$ implies that 
\[ x_1^2 +x_2^2 > C\left| x_1 t_\ell + x_2(1-t_\ell)\right|\]
for some $\ell$. The set of such $x$ is the exterior of a union of two circles that are tangent to the line  $ x_1 t_\ell + x_2(1-t_\ell)=0$ at $(0,0)$. It is a simple computation to show that in an $\epsilon$-neighborhood of $(0,0)$, if $t_\ell \ne 0,1$, any such $x$ must satisfy the non-trivial horn inequality 
\[ \left |x_2 +\frac{t_\ell}{1-t_\ell} x_1\right| \le B x_1^2 \]
 for a fixed $B>0$. Similarly, if $t_\ell =0$, then $x$ satisfies the trivial horn inequality, $|x_2| \le Bx_1^2$ and if $t_\ell =1$, $x$ satisfies the trivial horn inequality, $|x_1| \le Bx_2^2$. Here, each $B$ depends on $C$ and $t_{\ell}$, but not $x$. For each $\ell$, let $\mathscr{H}_{\ell}$ denote the horn associated to the $t_{\ell}, C$ horn inequality. Then if $x$ is sufficiently close to $(0,0)$ and $ \| x \|^2 \| x_{11}^{-1}  \| > C$, then $x \in \cup _{\ell=1}^L \mathscr{H}_\ell$. 
\end{proof}

\begin{proof}[Proof of Theorem \ref{thm:horn}]
For $x\in \R^2$ we write as above
$x_P = \begin{bmatrix} x_{11} & x_{12} \\ x_{21} & x_{22} \end{bmatrix}$ 
and set $\hat{x} = x_{22} - x_{21}x_{11}^{-1} x_{12}$.  

We claim there is a constant $C>0$ 
so that for 
$g(x)$ sufficiently close to $\eta_0$
we have
\[
\|\hat{x}\| > C.
\]
Then, since 
\[
\|\hat{x}\| \leq \|x\| + \|x\|^2 \|x_{11}^{-1}\|
\]
we will have 
$\|x\|^2\|x_{11}^{-1}\| > \tilde{C}>0$ 
for $\|x\|$ sufficiently small and $g(x)$ close to $\eta_0$.
The theorem will then follow from Lemma \ref{lem:horns}.

First by \eqref{gzero},
$g(x) - g^*(0,0) = -\langle \widehat{S}^{-1}\hat{x} \widehat{S}^{-1} (I+\hat{x} \widehat{S}^{-1})^{-1} \alpha, \beta\rangle$.
There is no loss in assuming 
\[
\|\hat{x}\| \leq \frac{1}{2\left\|\widehat{S}^{-1}\right\|}.
\]
Then, $\|(I+\hat{x}\widehat{S}^{-1})^{-1}\| \leq (1-\|\hat{x}\widehat{S}^{-1}\|)^{-1} \leq 2$
so
\[
| g(x) - g^*(0,0)|  \leq 2 \|\alpha\|\|\beta\|\|\widehat{S}^{-1}\|^2 \|\hat{x}\|.
\]
Assuming $|g(x)-\eta_0| \leq \frac{1}{2} |\eta_0 - g^*(0,0)|$ we then have $|g(x)-g^*(0,0)| \geq \frac{1}{2}|\eta_0 - g^*(0,0)|$ and therefore
\[
\frac{|\eta_0-g^*(0,0)|}{4\|\alpha\|\|\beta\| \|\widehat{S}^{-1}\|^2}
\leq \|\hat{x}\|,
\]
which proves the claim and completes the proof.
\end{proof}

Using conformal maps, Theorem \ref{thm:horn} can be easily translated to the bidisk, where it says: let $f$ be a rational
Schur function on $\mathbb{D}^2$ with a singularity at $(1,1)$ and nontangential value $f^*(1,1)$. If a sequence $(\tau^n) \subset \T^2$ satisfies
\[ \tau^n \rightarrow (1,1)  \ \text{ and } \ f(\tau^n) \ \rightarrow \lambda \ne f^*(1,1),\]
then there is a finite number of horns $\mathscr{H}_1, \dots, \mathscr{H}_L$ in $\T^2$  at $(1,1)$ such that $(\tau^n)$ gets stuck inside the union of the horns. Here
horns on $\T^2$ at $(1,1)$ correspond to sets of points $(e^{i\theta_1}, e^{i\theta_2})$ where $\theta_1, \theta_2 \in [-\pi,\pi]$ satisfy the same inequalities described in Definition \ref{def:horn1}. 
This conclusion clearly holds for all bounded rational functions on $\mathbb{D}^2$ as well.

The bidisk setting allows us to easily visualize this horn behavior. 

\begin{example} Let $p(z_1,z_2)=2-z_1-z_2$ and let $q = \tfrac{1}{2}(p + \tilde{p}).$ Then
\begin{equation} \label{eqn:rsf_fav} f(z_1, z_2) : = \frac{q(z_1,z_2)}{p(z_1,z_2)} = \frac{(z_1-1)(z_2-1)}{2-z_1-z_2}\end{equation}
is a bounded rational function on $\mathbb{D}^2$ with a singularity at $(1,1)$ and $f^*(1,1)=0.$ 
Consider the following four curves approaching $(1,1)$ in $\mathbb{T}^2$:
\[ 
\begin{aligned}
\gamma_1(s) = (e^{is}, e^{is}),  &\quad \gamma_2(s) = (e^{is}, e^{-is/2}),\\
\gamma_3(s) = (e^{is}, e^{-is}), &\quad \gamma_4(s) = (e^{is}, e^{-i\sin(s)}), \end{aligned} \]
which are graphed in Figure \ref{fig:horns}(a) below, using their arguments on $[-\pi,\pi]^2$. One can check
\[ \lim_{s\rightarrow 0} f(\gamma_1(s)) =   \lim_{s\rightarrow 0} f(\gamma_2(s)) =0 = f^*(1,1),\]
while 
\[ \lim_{s\rightarrow 0} f(\gamma_3(s)) =   \lim_{s\rightarrow 0} f(\gamma_4(s)) =1 \ne f^*(1,1).\]
Theorem \ref{thm:horn} indicates that $\gamma_3$ and $\gamma_4$ should get stuck in a horn region (or union of horn regions) near $(1,1)$, or equivalently near $(0,0)$ when considering their arguments. 
One can see this in Figure \ref{fig:horns}(a). The more general situation can be extracted from Figure \ref{fig:horns}(b), which graphs $|f(e^{i\theta_1}, e^{i\theta_2})|$ for $(\theta_1, \theta_2) \in [-\pi, \pi]^2.$
The picture shows that if a sequence $(\tau^n) \subseteq \mathbb{T}^2$ converges to $(1,1)$ and is not contained in a horn region (mapped onto a narrow ridge in the modulus plot) near the red curves $\gamma_3, \gamma_4$, then 
\[ |f(\tau^n)| \rightarrow 0, \text{ so } f(\tau^n) \rightarrow f^*(1,1).\]
Equivalently, if $ |f(\tau^n)| \rightarrow c \ne 0,$ then $(\tau^n)$ must get stuck in a narrow horn region near $\gamma_3, \gamma_4$.

\begin{figure}[h!]
    \subfigure[\textsl{The curves $\gamma_1, \gamma_2$ (in black) and $\gamma_3, \gamma_4$ (in red) graphed on $[-\pi, \pi]^2$ using their arguments. The curves $\gamma_3, \gamma_4$ approximate a ``horn region'' at $(1,1).$}]  
      {\includegraphics[width=0.4 \textwidth]{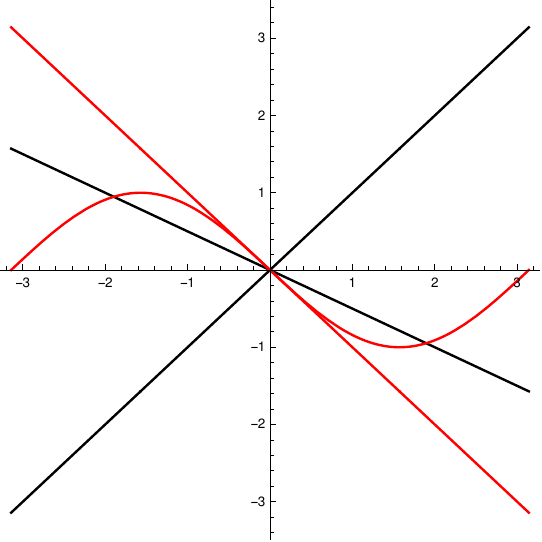}}
 \quad
    \subfigure[\textsl{The modulus $|f(e^{i\theta_1}, e^{i\theta_2})|$ and the curves $\gamma_1, \gamma_2$ (in black) and $\gamma_3, \gamma_4$ (in red) graphed on $[-\pi, \pi]^2.$}]
      {\includegraphics[width=0.5 \textwidth]{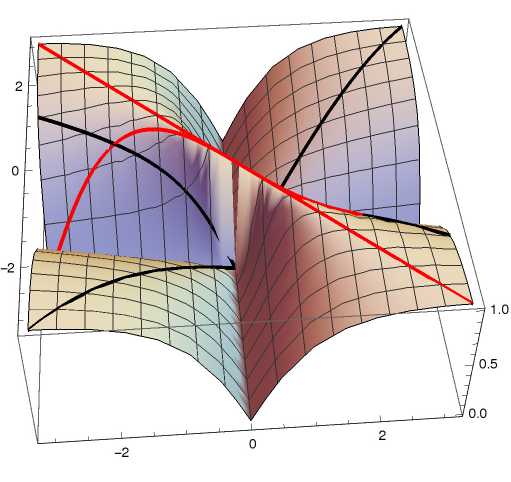}}
  \caption{\textsl{A horn region of $f$ from \eqref{eqn:rsf_fav}.}}
  \label{fig:horns}
\end{figure}
\end{example}

Some rational Schur functions appear to possess even narrower horn regions.

\begin{example} Set $p(z_1, z_2) = 4-z_2-3z_1-z_1z_2+z_1^2$ and $q = \tfrac{1}{2}(\tilde{p}-p).$ As in the previous example, 
\begin{equation} \label{eqn:rsf_amy} f(z_1, z_2) : = \frac{q(z_1,z_2)}{p(z_1,z_2)} = \frac{2z_1^2z_2-z_1^2-z_1z_2+z_1+z_2-2}{4-z_2-3z_1-z_1z_2+z_1^2}\end{equation}
is bounded on $\mathbb{D}^2$ with a singularity at $(1,1)$ and $f^*(1,1)=-1.$ 
Define the following four curves approaching $(1,1)$ in $\mathbb{T}^2$: 
\[ 
\begin{aligned}
\gamma_1(s) = (e^{is}, e^{is}),  &\quad \gamma_2(s) = (e^{is}, e^{-is/2}),\\
\gamma_3(s) = (e^{is}, h(e^{is})), &\quad \gamma_4(s) = (e^{is},  h(e^{is}) e^{i(1-\cos(s^4/20))}), 
\end{aligned}\]
where $h(z) = \frac{2-z+z^2}{1-z+2z^2}.$ The function $h$ is actually chosen so that the curve $\gamma_3$ parametrizes $\mathcal{Z}_q\cap \T^2$.  Then $f(\gamma_3(s)) \equiv 0$ where it is defined and so, must manifestly be different from $-1$ at $s=0.$ This suggests a way of finding horn regions for some functions: first solve $q=0$ on $\mathbb{T}^2$. If that yields a curve on $\mathbb{T}^2$, let that be one curve and then perturb the solution (up to some order, depending on $p$) to obtain another curve. Together these curves allow one to visualize a horn region.  These curves are graphed in Figure \ref{fig:horns2}(a). Similar to the previous example,
\[ \lim_{s\rightarrow 0} f(\gamma_1(s)) =   \lim_{s\rightarrow 0} f(\gamma_2(s)) =-1 = f^*(1,1),\]
while 
\[ \lim_{s\rightarrow 0} f(\gamma_3(s)) =   \lim_{s\rightarrow 0} f(\gamma_4(s)) =0 \ne f^*(1,1).\]
Thus, $\gamma_3, \gamma_4$ must become trapped in a horn region and as shown in Figure  \ref{fig:horns2}(a),
 the horn region appears to be so narrow that $\gamma_3, \gamma_4$ are indistinguishable near $(1,1)$. To see the global situation, consider the picture in Figure \ref{fig:horns2}(b). It is clear that if $(\tau^n)$ converges to $(1,1)$ and $|f(\tau^n)| \not \rightarrow 1$, then $(\tau^n)$ must get caught in a very narrow horn region (which is mapped to a steep valley in the modulus plot) containing the (basically) identical parts of $\gamma_3, \gamma_4.$
 
\begin{figure}[h!]
    \subfigure[\textsl{The curves $\gamma_1, \gamma_2$ (in black) and $\gamma_3, \gamma_4$ (in red) graphed on $[-\pi, \pi]^2$ via their arguments. The curves $\gamma_3, \gamma_4$ are basically indistinguishable near $(0,0)$.}]  
      {\includegraphics[width=0.4 \textwidth]{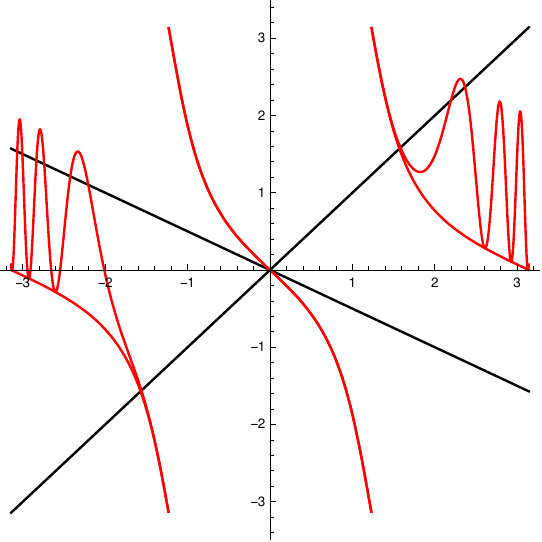}}
 \quad
    \subfigure[\textsl{The modulus $|f(e^{i\theta_1}, e^{i\theta_2})|$ and the curves $\gamma_1, \gamma_2$ (in black) and $\gamma_3, \gamma_4$ (in red) graphed on $[-\pi, \pi]^2.$ There is a very narrow horn region with negative slope along the overlapping red curves where $|f(e^{i\theta_1}, e^{i\theta_2})|$ is near $0$.}]
      {\includegraphics[width=0.5 \textwidth]{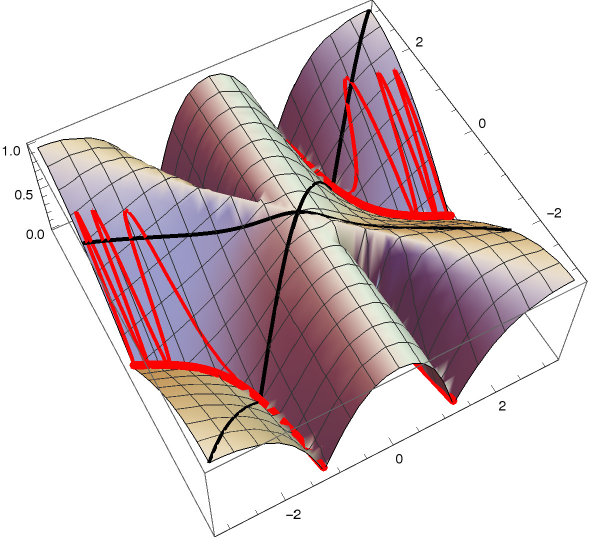}}
  \caption{\textsl{A horn region of $f$ from \eqref{eqn:rsf_amy}.}}
  \label{fig:horns2}
\end{figure}
\end{example}

%%%%%%%%%%%%%%%%%%%%%%%%%%%%%%%%%%%%%%%%%%%
\section{The ideal of admissible numerators}\label{sec:numerator}
In this section, we address the question of
characterizing the ideal
\[
\mathcal{I}^{\infty}_p = \{q\in \C[z_1,z_2]: q/p \in H^{\infty}(\D^2)\}
\]
for a given stable denominator $p\in \C[z_1,z_2]$.
Recall that $H^{\infty}(\D^2)$ 
is the set of bounded analytic functions on $\D^2$.
Recall from Theorems \ref{thm:nontanchar} and \ref{thm:uhplim}
(or rather their polydisk counterparts),
if $f=q/p$ is bounded on $\D^2$ and
$p(1,1) =0$, 
then assuming $p(1+z_1,1+z_2)$ has lowest order
homogeneous term $p_M$ at $(0,0)$ we have
that the $M$-th order homogeneous term of $q$
is a multiple of $p_M$.
We bring this up for two reasons.
First, we necessarily have
$\mathcal{Z}_p\cap\T^2 \subset \mathcal{Z}_q\cap \T^2$,
so
if $p$ has any toral factors then they
divide every element of $\mathcal{I}^{\infty}_p$.
We may therefore assume $p$ has no toral factors and
hence has only finitely many zeros on $\T^2$ (see Section \ref{sub:global}).
Second, this condition provides a necessary
condition for our problem.
Unfortunately,
this necessary condition is far from sufficient
and is actually somewhat misleading.

\begin{example}\label{ex:153}
Example \ref{ex153homog} is the polynomial 
with no zeros on $\D^2$
\[
p(z_1, z_2) = 4-5z_1-2z_2+2z_1 z_2+3z_1^2-z_1^2z_2-z_1^3 z_2.
\]
In \cite{Kne15} it is shown that for $q\in \C[z_1,z_2]$,
$\frac{q}{p} \in L^2(\T^2)$
if and only if $0= q(1,1) = \frac{\partial q}{\partial z_1}(1,1) - \frac{\partial q}{\partial z_2}(1,1)$ and
\begin{equation} \label{complicatedcondition}
\frac{\partial^2 q}{\partial z_1^2}(1,1) - 2\frac{\partial^2 q}{\partial z_1 \partial z_2}(1,1) + \frac{\partial^2 q}{\partial z_2^2}(1,1) +  2 \frac{\partial q}{\partial z_1}(1,1) \color{black} = 0.
\end{equation}
The condition $\frac{\partial q}{\partial z_1}(1,1) - \frac{\partial q}{\partial z_2}(1,1) = 0$ 
amounts to the requirement that
the first order homogeneous term of $q$ is a multiple of
the first order homogeneous term of $p$
(which again implies existence of non-tangential limits).
The last condition \eqref{complicatedcondition} is more complicated,
thus
making the condition of Theorem \ref{thm:uhplim}
not sufficient to even guarantee $\frac{q}{p} \in L^2(\T^2)$.
Moreover, this example also shows that for certain polynomials $p$
every $q$ with $q/p\in L^2(\T^2)$
has non-tangential limits at every point of $\T^2$.
We continue this example in Example \ref{ex:153:ideal}. \eox
\end{example}

On the other hand, 
a simple sufficient condition for $f=q/p$ to be bounded on $\D^2$
is that $q$ belongs to the polynomial ideal generated by $p$ and $\tilde{p}$.
This sufficient condition is not necessary and
we can get broader sufficient conditions by
looking at the local factorization of $p$
into Puiseux series, Theorem \ref{thm:purelp},
in the upper half plane setting.
Notice that we have reduced to the
case where $p$ has finitely many
zeros on the distinguished boundary, so we
can focus on a single isolated zero
of $p$. 

So, let us now work with $p \in \C[z_1,z_2]$ which
is pure stable (no zeros on $\UHP^2$ and no factors in common with
$\bar{p}$) such that $p(0,0)=0$.
We wish to describe all $q\in \C[z_1,z_2]$ 
such that $\frac{q}{p}$ is bounded
on a neighborhood of $(0,0)$ in $\C^2$
intersected with $\UHP^2$.
In general, if a function is analytic and bounded on $\UHP^2$ intersected with a neighborhood of
$(0,0)$ in $\C^2$ we shall say it is \emph{locally $H^{\infty}$} at $(0,0)$.
Let us define
\begin{equation}
\mathcal{I}^{\infty}(p, 0) = \left\{q \in \C[z_1,z_2]: \frac{q}{p} \text{ is locally } H^{\infty} \text{ at } (0,0) \right\}.
\label{def:Hinftyloc}
\end{equation}

Let $R_0 = \C\{z_1,z_2\}$ be the ring of convergent power series  centered at $(0,0)$.
 In order to state our broadest sufficient conditions for inclusion in $\mathcal{I}^{\infty}(p,0)$
we recall Theorem \ref{thm:purelp}
on the Puiseux expansion for pure stable polynomials.
We can factor $p=u p_1\cdots p_k$ where $u\in R_0$ is a unit
and each $p_j$ is an irreducible Weierstrass factor in $z_2$
 of pure stable type
\[
p_j(z_1,z_2) = \prod_{m=1}^{M_j} \left(z_2+q_j(z_1) + z_1^{2L_j} \psi_j(\mu_j^m z_1^{1/M_j})\right).
\]
Here $L_j$ is a positive integer,  $q_j(z_1) \in \R[z_1], q_j(0)=0, q_j'(0)>0, \deg q_j < 2L_j, \psi_j \in \C\{z_1\}, \psi_j(0) \in \UHP$,
$\mu_j = \exp(2\pi i/M_j)$.
Let us call each $z_2+q_j(z_1)$ an initial segment with cutoff $2L_j$ and multiplicity $M_j$ for $j=1,\dots k$.
Theorem \ref{PuiseuxLowerBound}
shows that 
\[
H(z_1,z_2) := \frac{z_1^{2L_j}}{z_2+q_j(z_1) + z_1^{2L_j} \psi_j(\mu_j^m z_1^{1/M_j})}
\]
is locally $H^{\infty}$ since $\Im(q_j(z_1) + z_1^{2L_j} \psi_j(\mu_j^m z_1^{1/M_j})) \geq c |z_1|^{2L_j}$.
Note we take $z_1^{1/M_j}$ to have a branch cut in 
the lower half plane.
Evidently, the function $z_1^{2L_j} \psi_j(\mu_j^m z_1^{1/M_j}) H(z_1,z_2)$ is still locally $H^{\infty}$ so we see that
\[
1- z_1^{2L_j}\psi_j(\mu_j^m z_1^{1/M_j}) H(z_1,z_2) = \frac{z_2+ q_j(z_1)}{z_2+q_j(z_1) + z_1^{2L_j} \psi_j(\mu_j^m z_1^{1/M_j})}
\]
is also locally $H^{\infty}$.  
This implies that if $q$ belongs to the product ideal $(z_2+q_j(z_1), z_1^{2L_j})^{M_j} R_0$  then $q/p_j$
is locally $H^{\infty}$.  
(Recall this means $q$ is in the ideal in $R_0$ generated by products $f_1\cdots f_{M_j}$ 
where each $f_i \in \{z_2+q_j(z_1), z_1^{2L_j}\}$.)  
Applied over all irreducible Weierstrass factors
this  implies the following theorem.

\begin{theorem} \label{thm:segmentinclusion} Assume the setting and subsequent conclusion of Theorem \ref{thm:purelp}. Then
\[
\C[z_1,z_2] \cap \prod_{j=1}^{k} (z_2+q_j(z_1), z_1^{2L_j})^{M_j} R_0 \subset \mathcal{I}^{\infty}(p,0).
\]
\end{theorem}

It is not hard to see that 
$p_j \in (z_2+q_j(z_1), z_1^{2L_j})^{M_j} R_0$ 
since  symmetric functions of $\psi(\mu_j^m z_1^{1/M_j})$ over $m=1,\dots, M_j$
are analytic at $0$.  Therefore, since $p = u p_1\cdots p_k$ we have
\[
p \in \prod_{j=1}^{k} (z_2+q_j(z_1), z_1^{2L_j})^{M_j} R_0. 
\]
Since every $q_j(z_1)$ has real coefficients, we also have $\bar{p}$
in this ideal.
Setting 
\begin{equation} \label{squarep}
[p](z_1,z_2) = \prod_{j=1}^{k} (z_2+q_j(z_1) + i z_1^{2L_j})^{M_j}
\end{equation}
we see that
\[
\frac{p(z_1,z_2)}{[p](z_1,z_2)} \text{ and } \frac{[p](z_1,z_2)}{p(z_1,z_2)}
\]
are locally $H^{\infty}$ by Theorem \ref{PuiseuxLowerBound}.
This allows a big conceptual reduction:
\begin{theorem}\label{thm:squarep} Assume the setting and subsequent conclusion of Theorem \ref{thm:purelp}. Then
\[
\mathcal{I}^{\infty}(p,0) = \mathcal{I}^{\infty}([p],0).
\]
\end{theorem}
This shows Puiseux series do not play any role in our problem.
Instead, the combinatorics
of the different ways the initial segments can overlap 
becomes the crux of the matter.
As a warm-up we can give a satisfying 
answer when $p$ vanishes to order $1$ at $(0,0)$.

\begin{theorem}\label{thm:onebranch}
Assume $p\in \C[z_1,z_2]$ is pure stable and vanishes
to order 1 at $(0,0)$. 
Then, 
\[
\mathcal{I}^{\infty}(p,0) = (p,\bar{p})R_0 \cap \C[z_1,z_2].
\]
\end{theorem}

\begin{proof}
In this situation, $p= u p_1$ has one degree $1$ irreducible Weierstrass factor 
$p_1(z_1,z_2) = z_2 + q(z_1) + z_1^{2L} \psi(z_1)$ as above.
The proof breaks down into showing
\[
(p,\bar{p})R_0 = (z_2+q(z_1), z_1^{2L})R_0
\]
and
\begin{equation} \label{degree1inclusion}
\mathcal{I}^{\infty}(p,0) \subset (z_2+q(z_1), z_1^{2L})R_0. 
\end{equation}

For the first equality note
\[
\bar{p}(z_1,z_2) = \bar{u}(z_1,z_2)(z_2+q(z_1) + z_1^{2L} \bar{\psi}(z_1)).
\]
The factors $u, \bar{u}$ are units in $R_0$ and so the ideal $(p,\bar{p}) R_0$ is then equal to 
$(p_1,\bar{p}_1)R_0$.
Recall that in any commutative ring the ideal generated by $A,B$ equals
the ideal generated by $A, B+CA$ for any other element $C$.
We use this fact below
\[
\begin{aligned}
&(z_2 + q(z_1) + z_1^{2L} \psi(z_1), z_2 + q(z_1) + z_1^{2L} \bar{\psi}(z_1))R_0 \\
&= (z_2 + q(z_1) + z_1^{2L} \psi(z_1), z_1^{2L}(\psi(z_1)-\bar{\psi}(z_1))R_0 \\
&= (z_2 + q(z_1) + z_1^{2L} \psi(z_1), z_1^{2L})R_0 \\
&= (z_2 + q(z_1), z_1^{2L})R_0. 
\end{aligned}
\]
The second equality is because $\psi(z_1)-\bar{\psi}(z_1)$
is a unit (from $\Im \psi(0)\ne 0$).  

Next we prove \eqref{degree1inclusion} by writing $f\in R_0$
as 
\[
f(z_1,z_2) = f_0(z_1) + (z_2+q(z_1)) f_1(z_1,z_2)
\]
for $f_0 \in \C\{z_1\}$ and $f_1 \in R_0$.
This can be accomplished for instance via the change
of variable $w_2 = z_2+ q(z_1)$.
If $f \in \mathcal{I}^{\infty}(p,0)$ then we wish to show $f\equiv 0$
modulo the ideal $\mathcal{I} := (z_2 + q(z_1), z_1^{2L})R_0$.
We can freely reduce $f$ mod $\mathcal{I}$ and assume
$f_1 \equiv 0$ and $\deg f_0 < 2L$.

Finally, we examine $f/p$ along the path $z_2+q(z_1) = 0$
for $z_1 \in \R$.  Now, $p(z_1, -q(z_1))$ vanishes to order $2L$
and therefore $f(z_1, -q(z_1)) = f_0(z_1)$ 
must vanish to at least the same order.   Otherwise,
$f/p$ is unbounded in $\R^2$ along a path tending to $(0,0)$ 
and we can find arbitrarily
large values in $\UHP^2$ close to $(0,0)$ contradicting that
$f$ is locally $H^\infty$.  
Since $f_0$
has degree at most $2L$ we get $f_0\equiv 0$.
\end{proof}

We record the following useful corollary of the proof.

\begin{corollary}\label{cor:genericdim}
If $p\in \C[z_1,z_2]$ is pure stable and vanishes to order
$1$ at $(0,0)$, then $\dim R_0/(p,\bar{p})R_0$
is equal to the contact order of $p$ at $(0,0)$.
\end{corollary}

\begin{proof}
In the proof above we have $(p,\bar{p})R_0 = (z_2+q(z_1), z_1^{2L})R_0$
where we note that $2L$ is the contact order of $p$
at $(0,0)$.  Using the change of variables  $w_2 = z_2+ q(z_1)$, it is not hard to prove
\[
\dim R_0/(z_2+q(z_1), z_1^{2L})R_0 = \dim R_0/(z_2, z_1^{2L})R_0 = 2L.
\]
\end{proof}

 \begin{example} \label{ex:153:ideal}
 Example \ref{ex:153} falls under the assumptions of
 the above theorem.
 The Cayley transformed version of Example \ref{ex:153}
 (designed to send $(1,1)$ to $(0,0)$)
 is
 \[
 p(z_1,z_2) = -4i(z_1+z_2- 2z_1^3-6z_1^2z_2 - iz_1(z_1+z_2-4z_1^2z_2)
  \]
 which vanishes to order $1$
 at $(0,0)$. This is the same polynomial
 as that of Examples \ref{ex153homog}, \ref{ex153puiseux}.
 Example \ref{ex153puiseux} shows 
 the pure stable Puiseux factor of $p$ starts out
 \[
 z_2+\underset{q(z_1)}{\underbrace{z_1+4z_1^3+24z_1^5}}+8iz_1^6 +\cdots.
 \]
 In this case $L = 3$ and 
 the ideal $\mathcal{I}^{\infty}(p,0)$ is therefore
 \[
 (z_2+z_1+4z_1^3+24z_1^5, z_1^6)R_0 \cap \C[z_1,z_2].
 \]
 \eox
 \end{example}

We can address three cases where the combinatorics is 
simple enough to prove the reverse inclusion
of Theorem \ref{thm:segmentinclusion}.
Here are the three cases:
\begin{description}
\item[Repeated segments] All of the initial segments of $p$ are the
same but we allow different cutoffs.
\item[Double points] $p$ vanishes to order $2$  at $(0,0)$.
\item[Ordinary multiple points] $p$ has distinct tangents
at $(0,0)$.
\end{description}

\begin{theorem}\label{thm:reverseinclusion}
Suppose $p\in \C[z_1,z_2]$ is pure stable with $p(0,0)=0$.
If any of the above conditions hold, then
using the notation of Theorem \ref{thm:segmentinclusion}
\[
\mathcal{I}^{\infty}(p,0) = \C[z_1,z_2]\cap 
\prod_{j=1}^{k} (z_2+q_j(z_1), z_1^{2L_j})^{M_j} R_0.
\]
\end{theorem}

The proofs of all three cases follow the same strategy:
\begin{itemize}
\item Find a generating set of 
\[
\mathcal{I} = \prod_{j=1}^{k} (z_2+q_j(z_1),z_1^{2L_j})^{M_j} R_0
\]
with elements of the form $z_1^{K_j}$ times a product of 
$j$ initial segments where we want $K_j$ minimal.
\item Write an arbitrary $f\in R_0$ as a combination
of products of initial segments and 
for $f \in \mathcal{I}^{\infty}(p,0)$, reduce $f$ modulo $\mathcal{I}$.
\item Show that $f$'s coefficients are forced to vanish
to an order higher than their degree because of the order
of vanishing of $[p]$ along certain initial segments (or curves 
with high order of contact with initial segments).
\end{itemize}

\begin{proof}[Proof of Theorem \ref{thm:reverseinclusion} for Repeated segments]
Suppose $p$ has the single initial segment $z_2+q(z_1)$
repeated with cutoffs $2L_1,\dots, 2L_M$.  The cutoffs
need not be distinct.  We assume without loss
of generality that $2L_1\leq 2L_2\leq \cdots \leq 2L_M$.
Let $S_k = 2\sum_{j=1}^{k} L_j$ and $S_0 = 0$.
The product ideal
\[
\mathcal{I} := \prod_{j=1}^{M} (z_2+q(z_1), z_1^{2L_j}) R_0
\]
is generated by $z_1^{S_k} (z_2+q(z_1))^{M-k}$ for $k=0,\dots, M$
with respect to the ring $R_0$. 
Every $f\in R_0$ can be represented
\[
f(z) = \sum_{j=0}^{M-1} f_j(z_1) (z_2+q(z_1))^{j} + f_M(z_1,z_2) (z_2+q(z_1))^{M}
\]
where $f_j\in \C\{z_1\}$, $j=0,\dots, M-1$, $f_M \in R_0$.
If $f \in \mathcal{I}^{\infty}(p,0)$ then we can reduce $f$ modulo $\mathcal{I}$
and assume $\deg f_j < S_{M-j}$ and $f_M\equiv 0$.
Our goal is to show $f_0,\dots, f_{M-1} \equiv 0$.
To achieve this we examine $f/p$ on certain
curves in $\R^2$ and show that boundedness along
these curves implies all of the $f_j$ are zero.
At this stage there is no harm in replacing $z_2$ with $z_2-q(z_1)$
and assuming $q\equiv 0$.  
With this reduction we have
\[
[p](z) = \prod_{j=1}^{M} (z_2 + i z_1^{2L_j}).
\]

To show $f_0\equiv 0$ note that 
\[
\frac{f(z_1,0)}{[p](z_1,0)} = \frac{f_0(z_1)}{c z_1^{S_M}}.
\]
For this to be bounded yet $\deg f_0 < S_M$ 
we must have $f_0\equiv 0$.

The inductive argument might
be clearer if we present the next step.
Note $[p](z_1,t z_1^{2L_M})$ vanishes
 to order $S_M$ so that
 $f(z_1, t z_1^{2L_M})$ must also vanish to
 order at least $S_M$ for every $t\in \R$.
 Then,
 \[
 \lim_{t\to 0} \frac{1}{t} f(z_1,tz_1^{2L_M}) = f_1(z_1)z_1^{2L_M}
 \]
 must vanish to order at least $S_M$
 and so $f_1$ vanishes to order at least $S_{M-1} = S_M-2L_M$
 implying $f_1\equiv 0$ since $\deg f_1 < S_{M-1}$.
 
Now suppose
 $f_0,\dots, f_{j} \equiv 0$.
Note that $[p](z_1, t z_1^{2L_{M-j}})$ vanishes to order
$S_{M-j}+ j 2L_{M-j}$ and therefore $f(z_1, t z_1^{2L_{M-j}})$
must vanish to at least this order.  But,
\[
\lim_{t \to 0} \frac{1}{t^{j+1}} f (z_1, t z_1^{2L_{M-j}}) = f_{j+1}(z_1) z_1^{(j+1) 2L_{M-j}}
\]
must also vanish to at least this order implying $f_{j+1}(z_1)$
vanishes to order at least $S_{M-j}+j 2L_{M-j} - (j+1)2L_{M-j}= S_{M-j-1}$.
Since $\deg f_{j+1} < S_{M-j-1}$ we see that $f_{j+1}\equiv 0$.
\end{proof}

\begin{proof}[Proof of Theorem \ref{thm:reverseinclusion} for Double points]
For this case, we assume $p$ has two initial segments
$z_2+q_1(z_1), z_2+q_2(z_1)$ with cutoffs $2L_1,2L_2$.  
We may assume $q_1 \ne q_2$ by the previous case
and we assume $L_1\leq L_2$.
Let $K$ be the order of vanishing of $q_1(z_1)-q_2(z_1)$.  
Set $\mathcal{I} = (z_2+q_1(z_1),z_1^{2L_1})(z_2+q_2(z_1), z_1^{2L_2})R_0$.
Then,
\[
\begin{aligned}
\mathcal{I} %&=% (z_2+P_1, z_1^{2L_1})(z_2+P_2,z_1^{2L_2}) R_0\\
&= ((z_2+q_1)(z_2+q_2), z_1^{2L_1}(z_2+q_2), z_1^{2L_2} (z_2+q_1), z_1^{2(L_1+L_2)}) R_0 \\
&=((z_2+q_1)(z_2+q_2), z_1^{2L_1}(z_2+q_2), z_1^{2L_2} (q_1-q_2), z_1^{2(L_1+L_2)}) R_0 \\
&= ((z_2+q_1)(z_2+q_2), z_1^{2L_1}(z_2+q_2), z_1^{2L_2+K}, z_1^{2(L_1+L_2)}) R_0 \\
&=((z_2+q_1)(z_2+q_2), z_1^{2L_1}(z_2+q_2), z_1^{2L_2+N} )R_0 \\
\end{aligned}
\]
where $N = \min\{2L_1, K\}$.  
Also note $[p](z) = (z_2+ q_1(z_1) + i z_1^{2L_1})(z_2+q_2(z_1) + i z_1^{2L_2}).$

Any $f \in R_0$  can be written
\[
f(z_1,z_2) = f_0(z_1) + f_1(z_1) (z_2+q_2(z_1)) 
+ f_2(z_1,z_2) (z_2+q_2(z_1))(z_2+q_1(z_1))
\]
for $f_0,f_1 \in \C\{z_1\}, f_2 \in R_0$. 
If $f\in \mathcal{I}^{\infty}(p,0)$ we
wish to show $f\in \mathcal{I}$ so we
may freely reduce $f$ modulo $\mathcal{I}$
and assume $f_2 \equiv 0$, $\deg f_0 < 2L_2+N$, $\deg f_1 < 2L_1$.

Now, $[p](z_1,-q_2(z_1))$ vanishes to order $2L_2+N$
hence $f(z_1,-q_2(z_1)) = f_0(z_1)$ must vanish to at least
that same order.
This implies $f_0\equiv 0$ by our degree bound.

Next, $q_2(z_1)-q_1(z_1) + t z_1^{2L_1}$
generically (in $t$) vanishes to order $N$
and so 
$[p](z_1, tz_1^{2L_1}-q_1(z_1))$
generically vanishes to order $2L_1 + N$.
Therefore, 
$f(z_1, tz_1^{2L_1} - q_1(z_1)) = f_1(z_1) (q_2(z_1)-q_1(z_1) + t z_1^{2L_1})$
vanishes to at least this order
implying that $f_1(z_1)$ vanishes to order at least $2L_1$.
Since $\deg f_1 < 2L_1$, we have $f_1\equiv 0$.
This shows $f \in \mathcal{I}$.
\end{proof}

\begin{proof}[Proof of Theorem \ref{thm:reverseinclusion} for 
Ordinary multiple points]
For this case we assume we have 
initial segments $z_2+q_j(z_1)$ with cutoffs $2L_j$, $j=1,\dots, M$
and each difference $q_j(z_1) - q_k(z_1)$ vanishes to order $1$ for $j\ne k$.
Assume $2L_1\leq 2L_2\leq \cdots \leq 2L_M$.  
Define the ideal
\[
\mathcal{I} := \prod_{j=1}^{M} (z_2+q_j(z_1), z_1^{2L_j}) R_0.
\]
Our strategy now is to iteratively simplify the generators of this
ideal.  
Since $H_k:= z_1^{2L_k} \prod_{j\ne k} (z_2+q_j(z_1)) \in \mathcal{I}$,
we see that $z_1^{2(L_{k+1}-L_k)} \color{black} H_k - H_{k+1}$ equals
 $z_1^{2L_{k+1} + 1} \prod_{j\ne k,k+1} (z_2+q_j(z_1))$
 times a unit.  
 Repeating this argument we see
 \[
 z_1^{2L_{k+2} +2} \prod_{j\ne k,k+1,k+2} (z_2+q_j(z_1)) \in \mathcal{I}
 \]
 and more generally
 \[
 z_1^{2L_{k+n}+ n} \prod_{j\ne k,\dots, k+n} (z_2+q_j(z_1)) \in \mathcal{I}
 \]
 for $k+n \leq M$.
 In particular,
 \[
 z_1^{2L_{n} +n-1}  \prod_{j=n+1}^{M} (z_2+q_j(z_1)) \in \mathcal{I}
 \]
 and it is these generators for $n=1,\dots, M$ that we
 will use to reduce elements of $\mathcal{I}^{\infty}(p,0)$ modulo $\mathcal{I}$. 

Given $f\in R_0$ we may write
\[
f(z_1,z_2) = f_0(z_1) + \sum_{n=1}^{M-1} f_n(z_1) \prod_{j=M-n+1}^{M} (z_2+q_j(z_1)) + f_M(z_1,z_2) \prod_{j=1}^M (z_2+q_j(z_1)) 
\]
where $f_n(z_1) \in \C\{z_1\}$ for $n=0,\dots, M-1$ and $f_M\in R_0$.
If $f\in \mathcal{I}^{\infty}(p,0)$ we may safely reduce $f$ modulo $\mathcal{I}$
and assume $f_M \equiv 0$, while $\deg f_n < 2L_{M-n} + M-n -1$.

Finally, we compare the order of vanishing of $[p]$ along different segments to $f$ and
show that $f \equiv 0$.
Note that $[p](z_1, -q_j(z_1))$ vanishes to order $2L_j + M-1$.
Then, $f(z_1,-q_M(z_1)) = f_0(z_1)$ must vanish to order at least
$2L_M + M-1$ implying $f_0\equiv 0$ by our degree bound.
Next, $f(z_1,-q_{M-1}(z_1)) = f_1(z_1) (q_{M}(z_1) - q_{M-1}(z_1))$ vanishes to order
at least $2L_{M-1} +M-1$ but $q_{M-1}-q_{M}$ vanishes to order $1$ because of
distinct tangents.
So, $f_1$ vanishes to order at least $2L_{M-1} + M-2$ again implying $f_1 \equiv 0$.
Continuing in this fashion, if $f_0,\dots, f_{n-1}\equiv 0$,
then $f(z_1, -q_{M-n}(z_1)) = f_{n}(z_1) \prod_{j=M-n+1}^{M} (q_j(z_1)- q_{M-n}(z_1))$
vanishes to order at least $2L_{M-n} + M-1$ but the product
vanishes to order $n$ so $f_n$ vanishes to order at least $2L_{M-n} +M-n-1$
implying $f_n\equiv 0$.
In the end we arrive at $f \equiv 0$ modulo $\mathcal{I}$.
\end{proof}

%%%%%%%%%%%%%%%%%%%%%%%%%%%%%%%%%%%%%%%%%%%
\section{Local and global integrability of derivatives}\label{sec:integrability}

In \cite{bps18, bps19a, bps19b}, the integrability of an RIF's partial derivatives on $\mathbb{T}^2$ 
was used to measure how badly behaved the function was near its singularities. In this section, we partially extend that analysis to rational Schur functions on $\mathbb{T}^2$. The denominators of rational Schur functions are exactly the atoral, stable polynomials on $\mathbb{D}^2$ and so those polynomials will be featured in this section.

\subsection{Integrability of functions with an atoral, stable denominator} 

We now count the number of integrability indices associated to a given atoral, stable $p$ on $\mathbb{D}^2$, where:

\begin{definition} A (possibly infinite) number $\p > 0$ is an \emph{integrability index of $p$} if there is a $q \in  \mathbb{C}[z_1, z_2]$ such that 
\[ \p = \sup_{\p' > 0} \left\{ \p': \tfrac{q}{p} \in L^{\p'}(\mathbb{T}^2) \right\}.\]
Here, $q$ is said to \emph{witness $p$ attaining its integrability index $\p$} (or, in brief, {\it witness $\p$}), and  $\p$ is called the \emph{integrability cut-off of $q$}.
\end{definition}

 It is not hard to show that each $q\in  \mathbb{C}[z_1, z_2]$ has a well-defined integrability cut-off; the details are given below in Remark \ref{rem:IIexists}. 
As $p$ has at most finitely many zeros on $\mathbb{T}^2$, the integrability of any $q/p$ on $\mathbb{T}^2$ will only depend on its integrability near those zeros. To examine this local integrability, let $\tau \in \mathbb{T}^2$ be a zero of $p$ and for $\p > 0$, define
\[ L^\p_\tau(\mathbb{T}^2) = \left \{ f  : f \in L^\p(U_{\tau})  \text{ for some open } U_{\tau} \subseteq \mathbb{T}^2 \text{ with } \tau \in U_{\tau} \right \}.\]
Let $R_\tau$ denote the local ring of convergent power series in $z_1, z_2$ centered at $\tau$. Then the integrability indices of $p$ at $\tau$ are defined as follows: 

\begin{definition}  A (possibly infinite) number  $\p > 0$ is an \emph{integrability index of $p$ at $\tau$} if there is a $q \in R_\tau$ such that 
\[ \p = \sup_{\p' > 0} \left\{ \p': \tfrac{q}{p} \in L^{\p'}_\tau(\mathbb{T}^2) \right\}.\]
Here, $q$ is said to \emph{witness $\p$} and $\p$ is called the \emph{$\tau$ integrability cut-off of $q$}.
\end{definition}

Note that $\p = \infty$ is always an integrability index of $p$ since it is witnessed by $q=p$. To study other $q$, let $N_{\tau}(p, \tilde{p})$ denote the intersection multiplicity of $p$ and $\tilde{p}$ at $\tau$. The positive integer $N_{\tau}(p, \tilde{p})$ is an algebraic characteristic of $p$ and can for example be computed via the equation
\begin{equation} \label{eqn:intmult} N_{\tau}(p, \tilde{p}) =\dim \left( R_\tau/ ( p, \tilde{p} ) R_\tau \right),\end{equation} 
see Section $12$ in \cite{Kne15} and Proposition 2.11 in \cite[Chapter 4]{clo05}. For more details about intersection multiplicity, we refer the reader to Section $12$ in \cite{Kne15}.
Here since $p$ and $\tilde{p}$ have no common factors, $N_{\tau}(p, \tilde{p})$ is finite. 
This allows one to show that each $q \in R_\tau$ has a well-defined $\tau$ integrability cut-off.

\begin{remark} \label{rem:IIexists} Let $M =N_{\tau}(p, \tilde{p})$ and let  $[1], [z_1], \dots, [z_1^M]$
 denote the cosets of  $1, z_1, \dots, z_1^M$ in $R_\tau/ ( p, \tilde{p} ) R_\tau$. By \eqref{eqn:intmult},
 these cosets are linearly dependent, thus there is some nonzero polynomial $r \in \mathbb{C}[z_1]$
such that $r \in  ( p, \tilde{p} ) R_\tau$. Fixing any $q \in R_\tau$, there is a small open set $U \subseteq \mathbb{T}^2$ containing $\tau$ such that 
\[\int_U \left |\tfrac{q}{p}(z) \right |^{\p} |dz| = \int_U \left |\tfrac{rq}{p}(z) \right |^{\p}   \left | \tfrac{1}{r} (z_1) \right |^{\p}  |dz| \lesssim \int_U  \left |\tfrac{1}{r(z_1)} \right |^{\p}  |dz| \le \int_{\mathbb{T}}  \left |\tfrac{1}{r(z_1)} \right |^{\p} |d z_1|.  \]
Note $|dz|$ is shorthand for $|dz_1||dz_2|$.
By the fundamental theorem of algebra, there is certainly some $\p >0$ such that $\tfrac{1}{r} \in L^{\p}(\mathbb{T})$,
which implies  that $q$ has a well-defined $\tau$ integrability cut-off.
Applying this argument at each zero of $p$ on $\mathbb{T}^2$ also shows that each $q\in  \mathbb{C}[z_1, z_2]$ has a well-defined integrability cut-off.
\end{remark} 

The result below bounds the number of finite integrability indices of $p$ at $\tau$ in terms of  $N_{\tau}(p, \tilde{p})$.

\begin{proposition} \label{prop:int} Let $p$ be an atoral, stable polynomial on $\mathbb{D}^2$ and let $\tau \in \mathbb{T}^2$ be a zero of $p$. Then, $p$ has at most $N_{\tau}(p, \tilde{p})$ finite integrability indices at $\tau.$
\end{proposition} 

\begin{proof} First, define the following sets: $Q_0^\tau = R_\tau$,
\[ 
Q^{\tau}_\p = \left \{ q\in R_{\tau}:  \tfrac{q}{p} \in L^\p_\tau(\mathbb{T}^2) \right\} \quad \text{ for } 0 < \p <\infty, 
\]
and $Q^{\tau}_\infty = \cap_{\p>0} Q^{\tau}_\p.$ These sets are all ideals in $R_\tau$ and if $\p \le \hat{\p}$, then $Q_{\hat{\p}}^\tau \subseteq Q^\tau_{\p}.$ 
Now, assume that 
 $\p_1, \p_2, \dots, \p_M$ is an increasing sequence of distinct finite integrability indices of $p$ at $\tau$. We will show that $M \le N_{\tau}(p, \tilde{p})$. 
Choose $q_1, \dots, q_M \in R_\tau$ such that $q_j$ witnesses $\p_j$ for $j=1, \dots, M.$ Observe that if $\p'_1 <\p_j < \p'_2$, then $q_j \in Q^{\tau}_{\p'_1}$, but $q_j \not \in Q^{\tau}_{\p'_2}.$ Choose numbers $\kappa_1, \dots, \kappa_M$ such that 
\[ 0 < \p_1 < \kappa_1 < \p_2 < \kappa_2 < \dots < \p_M < \kappa_M < \infty.\] 
By the above observations and the fact that $q_1 \in Q_0^\tau$ automatically, we have
\[ q_1 \in Q_0^\tau \setminus Q_{\kappa_1}^\tau, \ \ q_2 \in Q_{\kappa_1}^\tau \setminus Q_{\kappa_2}^\tau, \dots, \ \ q_M  
\in Q_{\kappa_{M-1}}^\tau \setminus Q_{\kappa_M}^\tau.\]
Let $[q_1], \dots, [q_M]$ be the cosets generated by $q_1, \dots, q_M$ in $R_\tau/ ( p, \tilde{p} ) R_\tau$. Because $q_1/p, \dots, q_M/p$ have different integrability behaviors near $\tau$, $[q_1], \dots, [q_M]$ must be linearly independent in $R_\tau/ (p, \tilde{p} ) R_\tau$. By \eqref{eqn:intmult},
we must have $M \le N_{\tau}(p, \tilde{p})$, and the result follows.
\end{proof}

To move from local to global bounds, we will let $N_{\mathbb{T}^2}(p,\tilde{p})$ denote the sum of  $N_{\tau}(p, \tilde{p})$ over all of the common zeros $\tau$ of $p, \tilde{p}$ on $\mathbb{T}^2$. 
 As discussed in \cite{Kne15}, B\'ezout's theorem implies that if $\deg p = (n_1,n_2)$, then $N_{\mathbb{T}^2}(p,\tilde{p}) \le 2n_1n_2$.  Then Proposition \ref{prop:int} gives an immediate bound on the number of (global) integrability indices of $p$ and the possible integrability cut-offs for derivatives of rational functions with denominator $p$. 

\begin{corollary} \label{cor:int}  Let $p$ be an atoral, stable polynomial on $\mathbb{D}^2$. 
Then, $p$ has at most $N_{\mathbb{T}^2}(p, \tilde{p})$ finite integrability indices.
\end{corollary}

\begin{corollary} \label{cor:RSFderint}  Let $p$ be an atoral, stable polynomial on $\mathbb{D}^2$, and let the real numbers $\p_1, \dots, \p_M$ denote the finite integrability indices of $p^2$ and $\p_{M+1} = \infty.$ Then $M \le 4 N_{\mathbb{T}^2}(p, \tilde{p})$ and for each 
$q \in \mathbb{C}[z_1, z_2]$, there is a $j$ with $1 \le j \le M+1$ such that  
\[ \p_ j = \sup_{\p' > 0}  \left\{ \p':  \partial_{z_1}(\tfrac{q}{p}) \in  L^{p'}(\mathbb{T}^2)\right\}.\]
\end{corollary}

\begin{proof} Note that $\partial_{z_1}(q/p)$ has denominator $p^2$ and $4 N_{\mathbb{T}^2}(p, \tilde{p}) = N_{\mathbb{T}^2}(p^2, \tilde{p}^2)$. Then the result follows immediately from Corollary \ref{cor:int}.
\end{proof}

\subsection{Derivative Integrability}
We now obtain more refined integrability results for partial derivatives of rational Schur functions.

\begin{definition} A number $\p > 0$ is a \emph{$z_1$-derivative integrability index of $p$} if there is a $q \in  \mathbb{C}[z_1, z_2]$ such that $q/p$ is a rational Schur function and
\begin{equation} \label{eqn:z1dindex} \p = \sup_{\p' > 0} \left\{ \p': \partial_{z_1}(\tfrac{q}{p}) \in L^{\p'}(\mathbb{T}^2) \right\}.\end{equation}
Since multiplying by a nonzero constant does not affect integrability, the $z_1$-derivative integrability indices of $p$
are exactly the numbers $\p$ such that \eqref{eqn:z1dindex} holds for some $q\in   \mathbb{C}[z_1, z_2]$ with $q/p \in H^{\infty}(\mathbb{D}^2).$
In both cases, $q$ is said to \emph{witness $\p$} and $\p$ is called the  \emph{$z_1$-derivative integrability cut-off of $q$}.  \end{definition}

We can easily identify some $z_1$-derivative integrability indices of $p$, using the concept of contact order defined in Section \ref{subsec:seg} and estimates derived from work in \cite{bps18}. To avoid a lengthy digression in the proof, we present some of those estimates in the following remark.

\begin{remark} \label{rem:estimates} Let $\phi = \frac{\tilde{p}}{p}$ be the RIF associated to $p$. If $K$ is the maximum contact order of $p$ at its zeros on $\mathbb{T}^2$, then Theorem 4.1 in  \cite{bps18} states that for $1 \le \p <\infty$, 
\[ \tfrac{\partial \phi}{\partial z_1} \in L^\mathfrak{p}(\mathbb{T}^2) \text{ if and only if } \mathfrak{p} < \frac{K+1}{K}.\]
The proof of Proposition \ref{prop:ii2} below includes  a local version of this for functions of the form $(z_2-\tau_2)^n \frac{\partial \phi}{\partial z_1}$. We provide some of the necessary estimates here. To that end, fix $\p$ with $1 \le \p < \infty$ and recall that standard properties of finite Blaschke products (as in Lemmas $4.2$ and $4.3$ in \cite{bps18}) show that for every interval $I \subseteq \mathbb{T}$ and almost every $z_2 \in \mathbb{T}$, 
\begin{equation} \label{eqn:FBP1}  \int_{I}  \left| \tfrac{\partial \phi}{\partial z_1}(z_1, z_2) \right|^\p |dz_1 | \approx \max_j \int_I \left( \frac{ 1-|a_j|^2}{|1-\bar{a}_j z_1|^2} \right)^\p |dz_1|,\end{equation}
where the $a_j$ are the zeros of $\tilde{p}(\cdot, z_2)$ in $\mathbb{D}$. Let $\tau=(\tau_1, \tau_2) \in \mathbb{T}^2$  be a zero of $p$. Then, as Theorem \ref{thm:purelp} can be translated to atoral stable polynomials on $\mathbb{D}^2$ (and the roles of $z_1, z_2$ interchanged), it can be used to obtain a parameterization 
\[ z_1 = \psi_1(z_2), \dots, z_1 = \psi_L(z_2)\]
of the components of the zero set $\mathcal{Z}_{\tilde{p}}$ that go through $\tau$ in some small neighborhood $U$ of $\tau$. Let $K_\tau^\ell$ denote the contact order of each such curve, so that if $z_2$ is close to $\tau_2$, then
\[ 1-|\psi_\ell (z_2) | \approx | z_2 - \tau_2|^{K_\tau^\ell}.\]
This is basically \eqref{eqn:CO} at the level of branches.  
Then if we let $I_1$ and $I_2$ be sufficiently small intervals in $\mathbb{T}$ around $\tau_1$ and $\tau_2$ respectively, \eqref{eqn:FBP1} implies that for each $z_2 \in I_2$
\begin{equation} \label{eqn:FBP2}  \int_{I_1}  \left| \tfrac{\partial \phi}{\partial z_1}(z_1, z_2) \right|^\p |dz_1 | \approx \max_{\ell} \int_{I_1} \left( \frac{ 1-|\psi_\ell(z_2)|^2}{|1-\overline{\psi_\ell(z_2)}
 z_1|^2} \right)^\p |dz_1|. \end{equation}
 The arguments in Lemma 4.3 \cite{bps18} control the integrals on the right-hand-side of \eqref{eqn:FBP2} when $I_1 = \mathbb{T}$ and $|\psi_\ell(z_2)| \ge 1/2$. However, the dominating part is the integral over a small interval in $\mathbb{T}$ centered at $e^{i\text{Arg}(\psi_\ell(z_2))}$. Thus, if one shrinks $I_2$ so that for each $z_2 \in I_2$, $|\psi_\ell(z_2)| >1/2$ and $I_1$ contains an interval of fixed length centered at $e^{i\text{Arg}(\psi_\ell(z_2))},$ the estimates in Lemma $4.3$  give
 \[ \int_{I_1} \left( \frac{ 1-|\psi_\ell(z_2)|^2}{|1-\overline{\psi_\ell(z_2)}
 z_1|^2} \right)^\p |dz_1| \approx (1-|\psi_\ell (z_2) |)^{1-\p} \approx | z_2 - \tau_2|^{K_\tau^\ell(1-\p)}\]
 for all $z_2 \in I_2$. Letting $K_\tau$ denote the contact order of $\phi$ at $\tau$ (or equivalently, the largest of the $K_\tau^\ell$), this gives
  \begin{equation} \label{eqn:FBP3} \int_{I_1}  \left| \tfrac{\partial \phi}{\partial z_1}(z_1, z_2) \right|^\p |dz_1 | \approx  | z_2 - \tau_2|^{K_\tau(1-\p)},\end{equation}
  for all $z_2 \in I_2$, which is exactly the estimate we need in the following proof.
\end{remark}

%(see \eqref{eqn:CO} for the equation in the bidisk setting). 

\begin{proposition} \label{prop:ii2}  Let $p$ be an atoral, stable polynomial on $\mathbb{D}^2$ with a zero at $\tau \in \mathbb{T}^2$. Let $K_\tau$ denote the contact order of $\phi = \tilde{p}/p$ at $\tau$. Then 
\begin{equation} \label{eqn:intlist}  \frac{K_{\tau}+1}{K_\tau},  \ \ \frac{K_{\tau}+1}{K_\tau-1}, \ \  \frac{K_{\tau}+1}{K_\tau-2}, \dots, \ \  \frac{K_{\tau}+1}{1}, \ \ \infty, \end{equation}
are  $z_1$-derivative integrability indices of $p$.
\end{proposition}
\begin{proof} We first produce a polynomial that witnesses each index in \eqref{eqn:intlist} \emph{near} $\tau$. Write $\tau = (\tau_1, \tau_2)$ and let $q_n = (z_2-\tau_2)^n \tilde{p}$, for each $0 \le n \le K_\tau.$ Then, $ \partial z_1(q_n/p) \in  L_\tau^{\p}(\mathbb{T}^2)$ if and only if $(z_2-\tau_2)^n \frac{\partial \phi}{\partial z_1} \in L_{\tau}^{\p}(\mathbb{T}^2).$ Now  for $1 \le \p < \infty$,  recall from Remark \ref{rem:estimates} that we can choose intervals $I_1, I_2 \subseteq \mathbb{T}$ around $\tau_1, \tau_2$ respectively such that \eqref{eqn:FBP3} holds for almost every $z_2 \in I_2.$
Then
\[ 
\begin{aligned}
\int_{I_1 \times I_2} \left |(z_2-\tau_2)^n \tfrac{\partial \phi}{\partial z_1}(z) \right|^\p |d z| &=
\int_{I_2} | z_2- \tau_2|^{n \p} \int_{I_1}  | \tfrac{\partial \phi}{\partial z_1}(z)|^\p |dz_1 | |dz_2| \\
&\approx \int_{I_2} |z_2 - \tau_2|^{n\p} |z_2 -\tau_2|^{(1-\p)K_{\tau}} |dz_2|.
\end{aligned}
\]
It follows immediately that the above $L^\p$ integral is finite if and only if $n \p + (1-\p) K_{\tau} > -1$. Solving for $\p$ gives  $\p < \frac{K_\tau+1}{K_\tau-n}$ when $n<K_{\tau}$. If $n = K_{\tau}$, then the integral is finite for all $\p$. For this last piece, we do not need to worry about whether $\p \ge 1$. This follows because for each value of $n$, the $L^\p$ integral is finite for some $\p >1$. Then the nested properties of these $L^\p$ spaces imply that integral is also finite for all $ 0 < \p \le 1$. 

To show that each value $\p =  \frac{K_\tau+1}{K_\tau-n}$ is a global $z_1$-derivative integrability index, we must construct a polynomial $s_n$ such that $s_n/p$ is bounded on $\mathbb{D}^2$ and $s_n$ witnesses $\p$.  
If $p$ has a single zero
on $\T^2$ we are done, so let
 $\lambda \ne \tau$ be any other zero of $p$ on $\mathbb{T}^2$ and without loss of generality, assume $\tau_1 \ne \lambda_1$. Then for $N = N_{\lambda}(\tilde{p},p)$, we know $\dim (R_\lambda /( p, \tilde{p}) R_\lambda )= N$ and so $[1], [z_1], \dots,  [z_1^N]$ are linearly dependent in $R_\lambda /( p, \tilde{p}) R_\lambda$. Thus, there exist $a_0, \dots, a_N \in \mathbb{C}$ not all zero so
\[ r_\lambda(z) := \sum_{j=0}^N a_j z_1^j \in ( p, \tilde{p}) R_\lambda.\]
Since we can divide out any factors of $(z_1-\tau_1)$ without affecting inclusion in $ ( p, \tilde{p}) R_\lambda$, we can further assume that $r_\lambda(\tau) \ne 0$. Furthermore, by construction, 
\begin{equation} \label{eqn:infty}  \sup_{\p' > 0}  \left\{ \p':   r_\lambda \tfrac{\partial \phi}{\partial z_1}\in  L^{p'}_\lambda(\mathbb{T}^2)\right\} =\infty.\end{equation}
Let $r$ be the product of these $r_\lambda$ where $\lambda$ 
varies over all zeros different
from $\tau$ on $\T^2$
and set $s_n = rq_n$.  
Then $s_n/p$ is a bounded rational function. As $r(\tau) \ne 0$, our previous arguments combined with \eqref{eqn:infty} imply that the $z_1$-derivative integrability cut-off of $s_n$ is $\p= \frac{K_\tau+1}{K_\tau-n}$ if $n < K_{\tau}$ and $\p=\infty$ if $n=K_{\tau}.$
\end{proof}

There is also a simple bound for the number of ``local'' $z_1$-derivative integrability indices of $p$
for certain restricted numerators.

\begin{theorem} \label{thm:ii}  Let $p$ be an atoral stable polynomial on $\mathbb{D}^2$ with a zero at $\tau \in \mathbb{T}^2$.  Then there is a list of integrability indices   $\p_1, \dots, \p_M, \p_{M+1}$ of $p^2$ at $\tau$ with $M \le  N_\tau(p, \tilde{p})$ and $\p_{M+1} = \infty$ such that if
$q \in  ( p, \tilde{p}) R_\tau,$
then 
\[\sup_{\p' > 0}  \left\{ \p':  \partial_{z_1}(q/p) \in  L^{\p'}_\tau(\mathbb{T}^2) \right\} = \p_j \text{ for some $j$ with $1 \le j \le M+1$} .\]
\end{theorem}

\begin{proof} We first need to limit the number of finite indices to at most $N_\tau(p, \tilde{p})$. To begin, write $\phi = \frac{\tilde{p}}{p}$ and $\frac{\partial \phi}{\partial z_1} = \frac{Q}{p^2}$, for some $Q \in \mathbb{C}[z_1, z_2].$ Mimicking the setup of Proposition \ref{prop:int}, define $R_0^\tau = R_\tau,$
\[ 
R^{\tau}_\p = \left \{ \hat{q}\in  R_\tau : \hat{q}Q/p^2 \in  L_{\tau}^{\p}(\mathbb{T}^2) \right\} \  \text{ for }  0 < \p < \infty, \\
\]
 and $R^{\tau}_\infty = \cap_{\p>0} R^{\tau}_\p.$ As in Proposition \ref{prop:int}, these are nested ideals in $R_\tau$. Observe that
\begin{equation} \label{eqn:ii}  \sup_{\p' > 0} \left\{ \p': \hat{q} \in R^{\tau}_{\p'} \right\} = \sup_{\p' > 0} \left\{ \p': \hat{q}Q/p^2 \in  L_{\tau}^{\p'}(\mathbb{T}^2) \right\} \end{equation} 
is well defined for each $\hat{q} \in R_\tau$. Clearly $\p=\infty$ equals \eqref{eqn:ii} when $\hat{q} = p^2$. 

Now we will show that there are at most $N_\tau(p, \tilde{p})$ finite values that \eqref{eqn:ii} can take. To that end, assume that $\p_1, \dots, \p_M$ is an  increasing list of distinct finite numbers such that for each $j$, there is a $ \hat{q}_j \in R_\tau$ such that $\p_j$ equals the quantity in \eqref{eqn:ii}.
Choose numbers $\kappa_1, \dots, \kappa_M$ such that 
\[ 0 < \p_1 < \kappa_1 < \p_2 < \kappa_2 < \dots < \p_M < \kappa_M < \infty.\] 
Then  $\hat{q}_1, \dots, \hat{q}_M$ satisfy
\[ \hat{q}_1 \in R_0^\tau \setminus R_{\kappa_1}^\tau, \ \ \hat{q}_2 \in R_{\kappa_1}^\tau \setminus R_{\kappa_2}^\tau, \dots, \ \ \hat{q}_M  
\in R_{\kappa_{M-1}}^\tau \setminus R_{\kappa_M}^\tau.\]
Now let $[\hat{q}_1], \dots, [\hat{q}_M]$ be the cosets generated by $\hat{q}_1, \dots, \hat{q}_M$ in $R_\tau/ ( p, \tilde{p} ) R_\tau$. Because $\hat{q}_1Q/p^2, \dots, \hat{q}_MQ/p^2$ have different integrability behaviors near $\tau$ on $\mathbb{T}^2$, the elements $[\hat{q}_1], \dots, [\hat{q}_M]$ must be linearly independent in $R_\tau/ ( p, \tilde{p} ) R_\tau$. Then \eqref{eqn:intmult} implies that $M \le N_{\tau}(p, \tilde{p})$. 

This argument shows that  \eqref{eqn:ii}  can take at most $N_{\tau}(p, \tilde{p})$ finite values as $\hat{q}$ ranges over $R_\tau$. With a slight abuse of notation, we label those values $\p_1, \dots, \p_M$ with $M \le  N_{\tau}(p, \tilde{p})$ and let $\p_{M+1} =\infty.$

To finish the proof
 fix $q\in (p,\tilde{p})R_{\tau}$, so that locally $f := q/p = r_1 + r_2 \phi$ for some $r_1, r_2 \in R_\tau$. Then, near $\tau$ we can write
\[ \frac{\partial f}{\partial z_1} = \frac{\partial r_1 }{\partial z_1} + r_2 \frac{\partial \phi}{\partial z_1} + \phi \frac{\partial r_2 }{\partial z_1}.\]
The only term that is not necessarily bounded near $\tau$ is $ r_2  \frac{\partial \phi}{\partial z_1}$, so it will determine the integrability of $\frac{\partial f}{\partial z_1}$. Recall that $r_2 \frac{\partial \phi}{\partial z_1} = r_2 \frac{Q}{p^2}$. By our previous argument, one of the $\p_j$ must satisfy 
\[ \p_j = \sup_{\p' > 0} \left\{ \p': r_2 Q/p^2 \in  L_{\tau}^{\p'}(\mathbb{T}^2) \right\}  = \sup_{\p' > 0}  \left\{ \p':  \partial_{z_1}(q/p) \in  L^{\p'}_\tau(\mathbb{T}^2)\right\},\]
which completes the proof. 
\end{proof}

In a generic situation Theorem \ref{thm:onebranch}, coupled with Theorem \ref{thm:ii} and Proposition \ref{prop:ii2}, allows us to precisely identify all of the $z_1$-derivative integrability indices of an atoral stable polynomial. 

\begin{theorem} \label{thm:RSFii}  Let $p$ be an atoral stable polynomial on $\mathbb{D}^2$ with  zeros $\tau^1, \dots, \tau^m \in \mathbb{T}^2$ and respective contact orders $K_{\tau^1}$, \dots, $K_{\tau^m}$. Further assume that $p$ vanishes to order 1 at each $\tau^j$.  Then, the  $z_1$-derivative integrability indices of $p$ are exactly 
\begin{equation} \label{eqn:globalint}  \frac{K_{\tau^j}+1}{K_{\tau^j}},  \ \ \frac{K_{\tau^j}+1}{K_{\tau^j}-1}, \ \  \frac{K_{\tau^j}+1}{K_{\tau^j}-2}, \dots, \ \  \frac{K_{\tau^j}+1}{1}\end{equation}
for $j=1, \dots, m$ and $\infty.$
\end{theorem}

\begin{proof} Proposition \ref{prop:ii2} immediately implies that both $\infty$ and each number in \eqref{eqn:globalint} are $z_1$-derivative integrability indices of $p$. 

Now fix $q \in \mathbb{C}[z_1,z_2]$ so that $q/p$ is a rational Schur function. 
Since $p$ vanishes to order 1 at each boundary zero, Theorem \ref{thm:onebranch} implies $q \in ( p, \tilde{p}) R_{\tau^j}$ for each $j$. (Although Theorem \ref{thm:onebranch} is stated on the bi-upper half plane, its conclusions can be easily moved to the bidisk via conformal maps.) 
Fix $j$ with $1\le j \le m$. By Theorem \ref{thm:ii}, 
 \begin{equation} \label{eqn:localii} \sup_{\p' > 0}  \left\{ \p':  \partial z_1(q/p) \in  L^{p'}_{\tau^j}(\mathbb{T}^2)\right\} \end{equation}
must equal one of at most $N_{\tau^j}(p, \tilde{p})$ finite numbers or $\infty$. By the argument in the proof of Proposition \ref{prop:ii2}, $K_{\tau^j}$ of those finite numbers are given in \eqref{eqn:globalint}. 
Because $p$ vanishes to order 1 at $\tau$, $N_{\tau^j}(p,\tilde{p}) = K_{\tau^j}$ by Corollary \ref{cor:genericdim}
(again by translating to the bidisk).
This means \eqref{eqn:globalint} must contain all finite numbers that could equal \eqref{eqn:localii}. As the $z_1$-derivative integrability cut-off of $q$ is the minimum of  \eqref{eqn:localii} over $j=1, \dots, m$, this establishes the claim.
\end{proof}

\begin{example}  Let $p(z_1,z_2) = 2-z_1-z_2.$ Then $p$ has a single zero on $\mathbb{T}^2$ at $\tau = (1,1)$ with  contact order  $K_{\tau}=2$, and $p$ vanishes to order $1$ at $(1,1)$. 
Theorem \ref{thm:RSFii} implies that the $z_1$-derivative integrability indices of $p$ 
are exactly $\tfrac{3}{2},$ $3,$ $\infty$. Furthermore, the proof of Proposition  \ref{prop:ii2} implies that $q=\tilde{p}$ witnesses $\p =\tfrac{3}{2},$ $q= (z_2-1)\tilde{p}$ witnesses $\p=3$ and $ q=(z_2-1)^2\tilde{p}$ witnesses $\p=\infty$.

Similarly, let  $p(z_1,z_2) = 4-3z_1-z_2-z_1z_2+z_1^2$. Then $p$ again has a single zero at $\tau=(1,1)$, this time with contact order $K_\tau=4$ and order of vanishing $1$. In this case, Theorem \ref{thm:RSFii} implies that the $z_1$-derivative integrability indices of $p$ 
are exactly $\frac{5}{4}$ $\frac{5}{3},$ $\frac{5}{2}$, $5$, $\infty$. As before, these are witnessed by polynomials $q= (z_2-1)^n \tilde{p}$ for $n=0, \dots, 4$. 

See \cite[p.~298]{bps18} for both contact order computations. \eox
\end{example}

Our next example has multiple zeros on $\mathbb{T}^2$ with different associated contact orders.

\begin{example}
Consider $p$ defined by
\[p(z_1,z_2)=4-z_2+z_1z_2-3z_1^2z_2-z_1^3z_2.\]
A slight variant of this polynomial is discussed in  \cite[p.~491]{bps19a}. It is also not hard to check directly that this $p$ does not vanish on $\mathbb{D}^2$ and has precisely 
two zeros on $\mathbb{T}^2$. These occur at $\tau^1=(1,1)$ and $\tau^2=(-1,1)$, respectively, with associated 
contact orders $K_{\tau^1}=2$ and $K_{\tau^2}=4$. One can verify these contact orders by first solving $\tfrac{\tilde{p}}{p}(z_1, z_2) = \lambda$ for $z_2$ to get $z_2 = \psi(z_1; \lambda)$ and then finding the power series expansions of $ \psi(z_1; \lambda)$ around $1$ and $-1$. The first series has coefficients $a_n$ depending on $\lambda$ starting with $n = 2$ (giving contact order $2$) and the second has coefficients $a_n$ depending on $\lambda$ starting with $n = 4$ (giving contact order $4$). 
Since $p$ 
vanishes to order $1$ at both zeros, Theorem \ref{thm:RSFii} implies that the 
$z_1$-derivative integrability indices of $p$ are $\frac{3}{2},3, \infty,$
coming from $K_{\tau^1}$, and $\frac{5}{4}, \frac{5}{3}, \frac{5}{2}, 5, \infty,$
from $K_{\tau^2}$.
Combining these in increasing order, we obtain the global list of $z_1$-integrability indices for $p$:
\[\frac{5}{4}, \frac{3}{2}, \frac{5}{3}, \frac{5}{2}, 3, 5, \infty. \] 
The first part of the proof of Proposition \ref{prop:ii2} implies that the integrability indices $\tfrac{5}{4}, \tfrac{5}{3}, \tfrac{5}{2}, 5, \infty$ are witnessed by $q=(1-z_2)^n\tilde{p}$ for $n=0,\dots, 4$. To witness the integrability indices $\tfrac{3}{2}, 3$,  the second part of the proof of Proposition \ref{prop:ii2} suggests that we should find a polynomial $r$ satisfying
\[ r \in (p,\tilde{p})R_{\tau^2} \ \text{ and } \ r(\tau^1) \ne 0.\]
To that end, let $r(z) = (1+z_1)(1+z_1 z_2)$. Then $r(\tau^1)  \ne 0$ and one can check that
\[ r(z) = -\tfrac{1}{(1-z_1)^2} \left( z_1 p(z) + \tilde{p}(z) \right) \in  (p,\tilde{p})R_{\tau^2}.\]
Set $q =(1-z_2)^nr \tilde{p}$. Then  the proof of Proposition \ref{prop:ii2}  implies that if $n=0$, $q$ witnesses $\p = \tfrac{3}{2}$ and if $n=1$, $q$ witnesses $\p = 3$.\eox
\end{example}

\subsection*{Acknowledgments}
Part of this paper was completed during visits to Stockholm University and the University of Florida, 
whom we thank for their hospitality and stimulating research environments.  
The authors would like to sincerely thank Professor Ryan Tully-Doyle for
helpful comments on a draft of this paper.
Finally, thank you to the anonymous referees for numerous
suggestions and a thorough reading of the paper.

%%%%%%%%%%% To ease editing, use normal size for the references:

\end{document}